\documentclass{amsart}
\usepackage{lipsum}
\makeatletter
\g@addto@macro{\endabstract}{\@setabstract}
\newcommand{\authorfootnotes}{\renewcommand\thefootnote{\@fnsymbol\c@footnote}}%
\makeatother

\usepackage[T1]{fontenc}
\usepackage{tipa}
\usepackage{enumitem}
\usepackage[dvipsnames]{xcolor}
\usepackage{setspace}
\usepackage[english]{babel}
\usepackage{amsmath,amsfonts,amssymb,amsthm,epsfig,epstopdf,url,array}
\usepackage[margin=1.0in]{geometry}
\usepackage{tikz-cd}
\usepackage{calc}
\usepackage{color,hyperref}
\usepackage[noabbrev]{cleveref}
\hypersetup{colorlinks}
\hypersetup{colorlinks={true},linkcolor={red},citecolor=green}
\usepackage{graphicx}
\usepackage{mathrsfs}
\graphicspath{{images/}}
\theoremstyle{plain}
\newtheorem{theorem}{Theorem}

\newtheorem{thm}{Theorem}[subsection]
\newtheorem{lem}[thm]{Lemma}
\newtheorem{prop}[thm]{Proposition}

\newcommand\scalemath[2]{\scalebox{#1}{\mbox{\ensuremath{\displaystyle #2}}}}
\theoremstyle{plain}
\newtheorem{defn}[thm]{Definition}

\theoremstyle{remark}
\newtheorem{rem}[thm]{Remark}

\newcommand{\mycomment}[1]{}
\usepackage{tikz}
\usepackage{mathtools}

\makeatletter

\renewcommand{\tocsection}[3]{%
  \indentlabel{\@ifnotempty{#2}{\bfseries\ignorespaces#1 #2\quad}}\bfseries#3}
\renewcommand{\tocsubsection}[3]{%
  \indentlabel{\@ifnotempty{#2}{\ignorespaces#1 #2\quad}}#3}

\renewcommand{\tocsubsubsection}[3]{%
  \indentlabel{\@ifnotempty{#2}{\ignorespaces#1 #2\quad}}#3}

\newcommand\@dotsep{4.5}
\def\@tocline#1#2#3#4#5#6#7{\relax
  \ifnum #1>\c@tocdepth 
  \else
    \par \addpenalty\@secpenalty\addvspace{#2}%
    \begingroup \hyphenpenalty\@M
    \@ifempty{#4}{%
      \@tempdima\csname r@tocindent\number#1\endcsname\relax
    }{%
      \@tempdima#4\relax
    }%
    \parindent\z@ \leftskip#3\relax \advance\leftskip\@tempdima\relax
    \rightskip\@pnumwidth plus1em \parfillskip-\@pnumwidth
    #5\leavevmode\hskip-\@tempdima{#6}\nobreak
    \leaders\hbox{$\m@th\mkern \@dotsep mu\hbox{.}\mkern \@dotsep mu$}\hfill
    \nobreak
    \hbox to\@pnumwidth{\@tocpagenum{\ifnum#1=1\bfseries\fi#7}}\par
    \nobreak
    \endgroup
  \fi}
\AtBeginDocument{%
\expandafter\renewcommand\csname r@tocindent0\endcsname{0pt}
}
\def\l@subsection{\@tocline{2}{0pt}{2.5pc}{5pc}{}}
\makeatother

\usepackage{xcolor}
\usepackage[new]{old-arrows}
\usepackage{bm} 
\usepackage{multicol}
\usetikzlibrary{decorations.pathmorphing}
\DeclareMathOperator{\GL}{GL}

\DeclareMathOperator{\ch}{ch}

\usepackage{mathtools}

\DeclareMathOperator{\vol}{vol}

\newcommand\blfootnote[1]{%
  \begingroup
  \renewcommand\thefootnote{}\footnote{#1}%
  \addtocounter{footnote}{-1}%
  \endgroup
}

\date{} 

\begin{document}
\hypersetup{citecolor=blue}
\hypersetup{linkcolor=red}

\begin{center}
  \LARGE 
 INTEGRALITY OF $\GL_2\times\GL_2$ RANKIN-SELBERG INTEGRALS\\ FOR RAMIFIED REPRESENTATIONS
  \normalsize
  \\
  \bigskip
  Alexandros Groutides \par 

  Mathematics Institute, University of Warwick\par
\end{center}

\begin{abstract}
    \textcolor{black}{Let $\pi_1,\pi_2$ be irreducible admissible generic tempered representations of $\mathrm{GL}_2(F)$ for some finite extension $F/\mathbf{Q}_p$ of odd residue characteristic. Inspired by work of Loeffler and previous work of the author on unramified zeta-integrals, we introduce a natural general notion of $(\pi_1\times\pi_2)$-\textit{integral} data at which the 
 Rankin-Selberg zeta-integral can be evaluated. We then establish an integral refinement of Jacquet-Langland's GCD-result for this zeta-integral, when evaluated at $(\pi_1\times\pi_2)$-integral data. This is compatible with the notion of integrality coming from the Fourier coefficients of newforms of even integral weights. Our approach relies on a reinterpretation of the Rankin-Selberg zeta-integral, and works of Assing and Saha on values of $p$-adic Whittaker new vectors.}
\end{abstract}
\blfootnote{The author was supported from the following research grant: ERC Grant No. 101001051—Shimura varieties and the Birch–Swinnerton-Dyer conjecture.}
\vspace{-2em}
\section{Introduction}
\subsection{Background}
Let $\pi_1,\pi_2$ be irreducible, admissible, generic representations of $\mathrm{GL}_2(F)$, where $F/\mathbf{Q}_p$ is a finite extension and $p$ is an odd prime. \textcolor{black}{Let $\varpi$ denote a uniformizer of $F$ and $q$ the cardinality of the residue field of $F$. For an integer $r\geq 0$, write $K_r$ for the open compact subgroup given by $\{k\in \GL_2(\mathcal{O}_F)\ |\ k\equiv \left[\begin{smallmatrix}
    * & * \\
    0 & 1
\end{smallmatrix}\right]\mod \varpi^r\}$}. A famous theorem of Casselman \cite{casselman1973some}, tells us that for every such representation $\pi_i$ there's a smallest non-negative integer $c(\pi_i)\in\mathbf{Z}_{\geq 0}$, called the conductor of $\pi_i$, for which $\pi_i^{K_{c(\pi_i)}}\neq 0$. In this case, this space of $K_{c(\pi_i)}$-invariant vectors is one-dimensional and generated by the unique Whittaker new vector $W_{\pi_i}^\mathrm{new}$ normalized to send the identity matrix to $1$, \textcolor{black}{upon identifying $\pi_1,\pi_2$ with appropriate Whittaker models.} Another celebrated theorem, this time due to Jacquet \cite{automorphic_on_GL_II}, tells us that the local Rankin-Selberg integral
\begin{align}\label{eq: intro A}
Z(\phi,g_1W_{\pi_1}^\mathrm{new},g_2W_{\pi_2}^\mathrm{new};s):=\int_{N(F)\backslash \GL_2(F)} W_{\pi_1}^\mathrm{new}(hg_1)W_{\pi_2}^\mathrm{new}(hg_2)\phi((0,1)h)\ |\det(h)|^s\ dh
\end{align}
for $\phi\in\mathcal{S}(F^2)$ and $g_1,g_2\in \GL_2(F)$, converges absolutely for $\Re(s)$ large enough, independently of $(\phi,g_1,g_2)$, and admits \textcolor{black}{a} unique meromorphic continuation as a rational function in $\mathbf{C}(q^s,q^{-s})$. Moreover, there exists a local Euler factor $L(\pi_1\times\pi_2,s)$, only depending on $\pi_1,\pi_2$, and a polynomial $\Phi(\phi,g_1W_{\pi_1}^\mathrm{new},g_2W_{\pi_2}^\mathrm{new};X)\in\mathbf{C}[X,X^{-1}]$, for which 
$$Z(\phi,g_1W_{\pi_1}^\mathrm{new},g_2W_{\pi_2}^\mathrm{new};s)=L(\pi_1\times\pi_2,s)\cdot \Phi(\phi,g_1W_{\pi_1}^\mathrm{new},g_2W_{\pi_2}^\mathrm{new};q^s).$$

If for example $F=\mathbf{Q}_p$ and $\pi_1=\pi_{f_1,p},\pi_2=\pi_{f_2,p}$, come from normalized newforms $f_1,f_2$ of even integral weights, then these local pieces are also of arithmetic interest. It turns out, that there's a number field $K$ containing $\mathbf{Q}(f_1,f_2)$, for which 
$$\Phi(\phi, g_1W_{\pi_1}^\mathrm{new}, g_2W_{\pi_2}^\mathrm{new};X)\in   K[X,X^{-1}]$$
for all primes $p$ and all rational data $(\phi,g_1,g_2)\in\mathcal{S}(\mathbf{Q}_p^2,\mathbf{Q})\times\GL_2(\mathbf{Q}_p)^2$. However, it quickly becomes evident, that the local Rankin-Selberg integral does not satisfy the naive notation of integrality. In other words, if one takes $(\phi,g_1,g_2)\in\mathcal{S}(\mathbf{Q}_p^2,\mathbf{Z})\times \GL_2(\mathbf{Q}_p)^2$, then $\Phi(\phi,g_1W_{\pi_1}^\mathrm{new}, g_2W_{\pi_2}^\mathrm{new};X)$ does \textit{not} necessarily lie in $\mathcal{O}_K[p^{-1}][X,X^{-1}]$, or in $\mathcal{O}_L[p^{-1}][X,X^{-1}]$ for any number field $L$, for that matter. Hence, this is not the right notion of integrality to impose on the data $(\phi,g_1,g_2)$.
\subsection{New results}
Following the background discussion above, we introduce a new notion of integrality inspired by the work of \cite{Loeffler_2021}, \cite{loeffler2021zetaintegralsunramifiedrepresentationsgsp4}, and \cite{groutides2024rankinselbergintegralstructureseuler}, on unramified zeta-integrals. \textcolor{black}{Before doing so, we note that whenever we regard $\GL_2$ as a subgroup of $\GL_2\times \GL_2$ we do so via the diagonal embedding $h\mapsto(h,h)$.} Going back to general $F/\mathbf{Q}_p$ and $\Pi:=\pi_1\times\pi_2$, we say that $(\phi,g_1,g_2)\in\mathcal{S}(F^2)\times \GL_2(F)^2$ is a $\Pi$-\textit{integral datum}, if the Schwartz function $\phi$ is valued in the principal ideal
\begin{align}\label{eq: intro integral}\frac{1}{\vol_{\GL_2(F)}(\mathrm{Stab}_{\GL_2(F)}(\phi)\cap (g_1,g_2)K_\Pi(g_1,g_2)^{-1})}\cdot \mathbf{Z}\subseteq\mathbf{Z}\end{align}
where $\vol_{\GL_2(F)}$ denotes the normalized Haar measure on $\GL_2(F)$ giving $\GL_2(\mathcal{O}_F)$ volume $1$, and the open compact subgroup $K_\Pi\subseteq\GL_2(\mathcal{O}_F)^2$ is given by $K_{c(\pi_1)}\times K_{c(\pi_2)}$. \textcolor{black}{In particular, as a subgroup of $\GL_2(F)$, $\mathrm{Stab}_{\GL_2(F)}(\phi)\cap (g_1,g_2)K_\Pi(g_1,g_2)^{-1}$ coincides with $\mathrm{Stab}_{\GL_2(F)}(\phi)\cap g_1K_{c(\pi_1)}g_1^{-1}\cap g_2 K_{c(\pi_2)}g_2^{-1}.$ } Before stating the main result, we set 
$$\tau=\tau(\pi_1,\pi_2)=\max\{c(\pi_1),c(\pi_2)\},\ \nu=\nu(\pi_1,\pi_2)=\begin{dcases}
    1,\ &\mathrm{if}\ c(\pi_1)=c(\pi_2)=0\\
    q^2-1,\ &\mathrm{otherwise.}
\end{dcases}$$
\begin{theorem}[\Cref{thm: main theorem}]\label{thm: intro A}
    Let $\pi_1,\pi_2$ be irreducible admissible generic tempered representations of $\GL_2(F)$. Set $\Pi:=\pi_1\times\pi_2$ and let $A\subseteq\mathbf{C}$ be a $\mathbf{Z}[q^{-1},\mu_{\nu q^\tau}]$-algebra containing the spherical Hecke eigenvalues of $\pi_1,\pi_2$ \emph{(}if spherical\emph{)} and \textcolor{black}{the values of} certain \emph{(}at most four\emph{)} unramified unitary characters of $F^\times$ determined by $\pi_1,\pi_2$. Let $(\phi,g_1,g_2)$ be any $\Pi$-integral datum. Then,  $\Phi(\phi,g_1W_{\pi_1}^\mathrm{new},g_2W_{\pi_2}^\mathrm{new};X)\in A[X,X^{-1}]$ and
    $$Z(\phi,g_1W_{\pi_1}^\mathrm{new},g_2W_{\pi_2}^\mathrm{new};s)=L(\Pi,s)\cdot \Phi(\phi,g_1W_{\pi_1}^\mathrm{new},g_2W_{\pi_2}^\mathrm{new};q^s)$$
    in $L(\Pi,s)A[q^s,q^{-s}]\subseteq A(q^s,q^{-s}).$
    
\end{theorem}

We remark that \Cref{thm: intro A} also holds for \textcolor{black}{``$A$-valued'' $\Pi$-integral data. These are data that satisfy \eqref{eq: intro integral} with $\mathbf{Z}$ replaced by the algebra $A$.} This is clear from the proof. Moreover, it also holds for $2$-adic fields $F$ under the extra assumption that the $\pi_i$ are dihedral if they are supercuspidal representations. We refer to \Cref{rem p=2} for more details on this. From a purely local point of view, the specific volume factor in our definition of integrality naturally appears in the standard unfolding of the Rankin-Selberg zeta-integral coming from the Iwasawa decomposition of $\GL_2(F)$. However, it quickly becomes evident that for this level of generality on $(\phi,g_1,g_2)$, it is futile to attempt to prove our result using the standard unfolding directly. Our approach crucially relies on a technical reinterpretation of the Rankin-Selberg integral and a reduction step, found in \Cref{sec 3}. In \Cref{sec 4} we combine this with the results of Assing \cite{Assing_2018}, and Saha \cite{Saha_2015} on values of $p$-adic Whittaker new vectors. The explicit integral representations of these values in terms of generalized Kloosterman sums, proved in \cite{Assing_2018}, are used extensively in our work. 

The second main result concerns an application of the above local result to automorphic forms. Let $f_1,f_2$ be normalized cuspidal new eigenforms of levels $N_1,N_2$ and even integral weights $k_1\geq k_2\geq 2$. Let $\varphi_{f_1,},\varphi_{f_2}$ be the associated cuspidal automorphic forms, and $\pi_{f_1},\pi_{f_2}$ the associated unitary cuspidal automorphic representations in the style of \cite{gelbart1975automorphic}. We write $L(\pi_{f_1}\times\pi_{f_2},s)$ for the standard automorphic $L$-function of \cite{jacquet1981euler}. For a Schwartz function $\phi_\mathbf{A}\in\mathcal{S}(\mathbf{A}^2)$ and cusp forms $\varphi_1\in\pi_{f_1}$, $\varphi_2\in\pi_{f_2}$, we write $I(\phi_\mathbf{A},\varphi_1,\varphi_2;s)$ for the global Rankin-Selberg integral of \cite{jacquet1981euler}. Following \cite{chen2020primitive}, we set $\phi_\infty^{(k)}:=2^{-k}(x+\sqrt{-1}y)^ke^{-\pi(x^2+y^2)}\in\mathcal{S}(\mathbf{R}^2)$ for $k\in\mathbf{Z}_{\geq 0}$. We write $L_{f_1,f_2}$ for the Galois closure of the number field $\mathbf{Q}(f_1,f_2)$. Finally, we write $E_{f_1,f_2}$ for the number field obtained by adjoining to $L_{f_1,f_2}$ a finite collection of roots of unity, only depending on $f_1,f_2$ (for more details we refer to \Cref{sec application fo newforms}). From standard facts regarding Hecke eigenvalues, together with \cite{loeffler2012computation}, the local unramified unitary characters and spherical eigenvalues mentioned in \Cref{thm: intro A}, corresponding to $\pi_{f_1,p}\times\pi_{f_2,p}$ at any prime $p$, are always contained in $\mathcal{O}_{E_{f_1,f_2}}[p^{-1}]$. Using this, one can obtain the following integrality result.

\begin{theorem}[\Cref{thm 5.2.2}]\label{thm intro B}
    Let $f_1,f_2$ be normalized cuspidal new eigenforms of even weights $k_1\geq k_2\geq 2$ and levels $N_1,N_2$ such that the $2$-component of $\pi_{f_i}$ is dihedral, if it's a supercuspidal representation. Set $\Pi:=\pi_{f_1}\times\pi_{f_2}$ and let $(\phi_\mathrm{f},g_1,g_2)\in\mathcal{S}(\mathbf{A}_\mathrm{f}^2)\times \GL_2(\mathbf{A}_\mathrm{f})^2$ be an arbitrary locally $\Pi_p$-integral datum for all primes $p$. Let $S$ be the finite set of primes made up of \textcolor{black}{the prime divisors of $N_1N_2$} and all other primes for which $(\phi_{\mathrm{f},p},g_{1,p},g_{2,p})$ is not an unramified datum. Finally, let $\phi_\mathbf{A}:=2^{k_1+1}\phi_\infty^{(k_1-k_2)}\otimes\phi_\mathrm{f}\in\mathcal{S}(\mathbf{A}^2)$. Then,
    \begin{align*}
I(\phi_\mathbf{A},g_1\varphi_{f_1},g_2\varphi_{f_2};s)=L(\pi_{f_1}\times\pi_{f_2},s)\cdot \Phi(p^s : p\in S)
\end{align*}
where $\Phi(X_p :p\in S)$ is a polynomial in $ \mathcal{O}_{E_{f_1,f_2}}[S^{-1}][X_p^{\pm 1} : p\in S]$. 
\end{theorem}

It is interesting to note that the notion of integral data introduced in this work, also has a global algebro-geometric interpretation which we briefly describe here. Let $\ell\nmid N_1N_2$ be an auxiliary prime. By work of Loeffler-Skinner-Zerbes \cite{loeffler2021euler} there exists an infinite level $\ell$-adic Lemma-Eisenstein map 
$$\mathcal{LE}:\mathcal{S}_{(0)}(\mathbf{A}_\mathrm{f}^2,\mathbf{Q})\otimes C_c^\infty((\GL_2\times\GL_2)(\mathbf{A}_\mathrm{f}),\mathbf{Q})\rightarrow H^3_\mathrm{\acute{e}t}(Y_{\GL_2}\times Y_{\GL_2},E_{f_1,f_2,\ell}\text{-}\mathrm{coefficieints})$$
where $Y_{\GL_2}$ denotes the infinite level Shimura variety, and $E_{f_1,f_2,\ell}$ denotes the completion of $E_{f_1,f_2}$ at a place above $\ell$. For more details and definitions, we refer the reader to \textit{op.cit.} The important relation to the present work is that everywhere locally $\Pi_p$-\textit{integral} data $(\phi_\mathrm{f},g_1,g_2)$ as in \Cref{thm intro B}, are precisely the ones that give rise to classes in the image of the \textit{integral} \'etale cohomology of the corresponding modular curves $H^3_\mathrm{\acute{e}t}(Y_1(N_1)\times Y_1(N_2),\mathcal{O}_{E_{f_1,f_2,\ell}}\text{-}\mathrm{coefficients})$
under the assignment
$(\phi_\mathrm{f},g_1,g_2)\mapsto\mathcal{LE}\left(\phi_\mathrm{f}\otimes\ch(g_1U_1(N_1)\times g_2U_1(N_2)\right).$ Such classes play a crucial role in the theory of Rankin-Selberg Euler systems and $p$-adic $L$-functions, and \Cref{thm intro B} can be regarded as the analogous integrality result on the representation-theoretic side of things.

\subsection{Future remarks} 
It is an open problem to determine in full generality whether the local $L$-factors $L(\pi_1\times\pi_2)$ can actually be obtained integrally by evaluating the Rankin-Selberg zeta-integral at some $A$-linear combination of $A$-valued $(\pi_1\times\pi_2)$-integral data. Here $A$ denotes a $\mathbf{Z}[q^{-1},\mu_{\nu q^\tau}]$-algebra as in \Cref{thm: intro A}. Our results in this paper give one inclusion of a stronger version of this open problem. More specifically, \Cref{thm: intro A} tells us that the fractional ideal of $A[q^s,q^{-s}]$ given by the $A$-span of the Rankin-Selberg zeta-integral evaluated at $A$-valued $(\pi_1\times\pi_2)$-integral data, is contained in the principal fractional ideal $L(\pi_1\times\pi_2,s)A[q^s,q^{-s}]\subseteq A(q^s,q^{-s})$. However, we do not know if it is as big as possible or even principal. This is
a natural problem of independent interest through the lenses of integral structures in local representation
theory. Additionally, as pointed out in \cite[\S $3.4$]{loeffler2025euleradjointmodularform}, answers to these questions will have immediate applications to the theory of Rankin-Selberg Euler systems with optimal local factors at the bad primes. Finally, we expect the analogous results of this paper to hold for the local Asai zeta-integral associated to representations of $\GL_2(E)$ with $E/F$ a quadratic field extension. The aforementioned problem in the Asai case is also open and is related to Asai-Flach Euler systems.
\subsection{Acknowledgments} I would like to thank my advisor David Loeffler for his constant support
and mathematical guidance. Additionally, I would like to thank Edgar Assing for useful discussions
regarding his work on Whittaker newforms.
{
  \hypersetup{linkcolor=black}
  \setcounter{tocdepth}{2}
  \tableofcontents
}
\section{Preliminaries}\label{sec 2}
Let $p$ be an odd prime and $F/\mathbf{Q}_p$ be a finite extension. We write $\mathcal{O}=\mathcal{O}_F$ for the ring of integers of $F$, $\varpi$ for a fixed choice of uniformizer of $F$, and $q$ for the cardinality of the residue field of $F$. Additionally, we fix once and for all an additive character $\psi:F\rightarrow \mathbf{C}^\times$ of conductor $\mathcal{O}$. \textcolor{black}{In other words, the character $\psi$ is trivial on $\mathcal{O}$ but non-trivial on $\varpi^{-1}\mathcal{O}$.} The normalized valuation on $F$ is given by $v(\varpi)=1$, and for the normalized absolute value we have $|\varpi|=q^{-1}$. We fix once and for all the additive Haar measure $dx$ on $F$ which gives $\mathcal{O}$ volume $1$, and the multiplicative Haar measure $d^\times x$ on $F^\times$ which gives $\mathcal{O}^\times$ volume $1$.
We define the algebraic groups
    $$H:=\GL_2,\ \ G:=\GL_2\times \GL_2.$$
Furthermore, we write $B$ to denote the upper triangular Borel subgroup of $H$, $N$ its unipotent radical and $P$ the  mirabolic subgroup of $H$. Whenever we regard $H$ as a subgroup of $G$, we do so via the diagonal embedding.
\subsection{$\GL(1)$ and $\GL(2)$ local factors}
Let $\chi:F^\times\rightarrow \mathbf{C}^\times$ be a character. The $L$-factor associated to $\chi$ is given by 
\begin{align}\label{eq: dirichlet l-factor}L(\chi,s)=\begin{dcases}
    (1-\chi(\varpi)q^{-s})^{-1},& \ \mathrm{if}\ \chi\ \mathrm{is}\ \mathrm{unramified}\\
    1,& \ \mathrm{otherwise}. 
\end{dcases}\end{align}
Given characters $\chi,\mu:F^\times\rightarrow \mathbf{C}^\times$, with $\chi\mu^{-1}\neq |\cdot|^{\pm 1}$, we write $I(\chi,\mu)$ for the irreducible principal-series representation given by the \textcolor{black}{normalized parabolic induction} $\mathrm{ind}_{B(F)}^{H(F)}\left[\begin{smallmatrix}
    \chi & \\
     & \mu
\end{smallmatrix}\right]$. If $\chi\mu^{-1}=|\cdot|$, we write $\mathrm{St}(\chi,\mu)$ for the associated Steinberg representation given by the irreducible co-dimension $1$ constituent of $\mathrm{ind}_{B(F)}^{H(F)}\left[\begin{smallmatrix}
    \chi & \\
     & \mu
\end{smallmatrix}\right]$. Now let $\pi$ be an arbitrary irreducible admissible generic representation of $H(F)$. Following \cite{automorphic_on_GL}, the $L$-factor associated to $\pi$ is given by 
\begin{align}\label{eq: gl2 l-factor}
    L(\pi,s)=\begin{dcases}
        L(\mu,s)L(\chi,s),&\ \mathrm{if}\ \pi=I(\chi,\mu)\\
        L(\chi,s), &\ \mathrm{if}\ \pi=\mathrm{St}(\chi,\mu)\\
        1,&\ \mathrm{if}\ \pi\ \mathrm{is}\ \mathrm{supercuspidal}
    \end{dcases}
\end{align}
where each individual Dirichlet $L$-factor is given by \eqref{eq: dirichlet l-factor}. There are also $\epsilon$-factors $\epsilon(s,\chi)=\epsilon(s,\chi,\psi)$ and $\epsilon(s,\pi)=\epsilon(s,\pi,\psi)$ attached to each such $\chi$ and $\pi$. For the precise definitions, see \cite{tate1967fourier} for the abelian case, and \cite{automorphic_on_GL}. For a compact summary of all their relevant properties, we refer to \cite{schmidt2002some}. For completeness, we recall the integral representation for ramified $\chi$, given by $$\epsilon(s,\chi)=\int_{\varpi^{-c(\chi)}\mathcal{O}^\times}|x|^{-s}\chi^{-1}(x)\psi(x)\ dx$$
which will frequently allow us to express such $\GL_1$ epsilon-factors as finite sums of certain roots of unity.
\subsection{Gauss sums and partial Gauss sums}\label{sec Gauss sums}
Given a character $\chi:F^\times\rightarrow\mathbf{C}^\times$ with $\chi(\varpi)=1$, we write 
\begin{align*}\mathscr{G}(\chi,x)&:=\int_{\mathcal{O}^\times}\chi(y)\psi(xy)\ d^\times y,\ x\in F^\times\\
\mathscr{G}_{\ell}(\chi,x)&:=\int_{1+\varpi^\ell\mathcal{O}}\chi(y)\psi(xy)\ d^\times y,\ x\in F^\times,\ell\in\mathbf{Z}_{>0}.
\end{align*}
We refer to $\mathscr{G}(\chi,-)$ as a Gauss sum, and to $\mathscr{G}_\ell(\chi,-)$ as a partial Gauss sum, for $\chi$. Later on, we will frequently need to evaluate expressions of this form. To do that we use the following expressions which can be extrapolated from \cite[Lemma $2.3$]{Corbett_2018}, \cite[$1.3.1$]{Assing_2019}  and \cite[Remark $3.3.7$]{Assing_2019}. For $x\in F^\times$ we have 
\begin{align*}
    \mathscr{G}(\chi,x)=\begin{dcases}
        1,\ &\mathrm{if}\ c(\chi)=0\ \&\ v(x)\geq 0\\
        \tfrac{-q}{q-1},\ &\mathrm{if}\ c(\chi)=0\ \&\ v(x)=-1\\
        \tfrac{q}{q-1}q^{-c(\chi)/2}\epsilon(\tfrac{1}{2},\chi^{-1})\chi^{-1}(x),\ &\mathrm{if}\ c(\chi)\geq 1\ \&\ v(x)=-c(\chi)\\
        0,\ &\mathrm{otherwise.}
    \end{dcases}
\end{align*}
For $u\in\mathcal{O}^\times$ and $a\in\mathbf{Z}$ we have
\begin{align*}
    \mathscr{G}_\ell(\chi,u\varpi^{a})=\begin{dcases}
        \tfrac{q}{q-1}\epsilon(\tfrac{1}{2},\chi^{-1})\chi^{-1}(u)q^{-a/2},\ &\mathrm{if}\ \ell\leq \lfloor\tfrac{c(\chi)}{2}\rfloor, a=-c(\chi) \ \&\ u\in-b_\chi+\varpi^\ell\mathcal{O}\\
        \tfrac{q}{q-1}\psi(u\varpi^a),\ &\mathrm{if}\ \lceil\tfrac{c(\chi)}{2}\rceil\leq \ell <c(\chi), a=-c(\chi)\  \&\ u\in-b_\chi+\varpi^{c(\chi)-\ell}\mathcal{O}\\
        \tfrac{q}{q-1}q^{-\ell}\psi(u\varpi^a),\ &\mathrm{if}\ \ell\geq c(\chi)\ \&\ a\geq -\ell\\
        0,\ &\mathrm{otherwise.}
    \end{dcases}
\end{align*}
where $b_\chi\in\mathcal{O}^\times$ is a certain unit associated to $\chi$ as in the third and last part of \cite[Lemma $4.1$]{Assing_2018}. The nature and properties of this unit are not important for our purposes. The interested reader can consult \textit{loc.cit.} and also \cite[Lemma $2.4$]{cesnavicius2022manin}. We also note that whenever $b_\chi$ appears it is implicit from the inequalities on the right that $c(\chi)\geq 2$.
\subsection{New vectors for $\GL(2)$}
    For an integer $r\in\mathbf{Z}_{\geq 0}$, we write $K_r$ for the open compact subgroup of $H(\mathcal{O})$ given by
    $$\textcolor{black}{K_r:=\left[\begin{matrix}
        \mathcal{O} & \mathcal{O}\\
        \varpi^r\mathcal{O} & 1+\varpi^r\mathcal{O}
    \end{matrix}\right]\cap H(\mathcal{O})\subseteq H(\mathcal{O}).}$$
 We have the following celebrated theorem of Casselman on the behavior of $K_r$-fixed vectors in irreducible admissible generic representations of $H(F)$:
    \begin{thm}[\emph{\cite{casselman1973some}}]\label{thm local new vectors}
        Let $\pi$ be an irreducible admissible generic representation of $H(F)$. \textcolor{black}{Denote the space of $K_{r}$-fixed vectors of $\pi$ by $\pi^{K_{r}}.$} Then,
        \begin{enumerate}
            \item There exists a least non-negative integer $c(\pi)\in\mathbf{Z}_{\geq 0}$, the \textit{conductor} of $\pi$, such that $\pi^{K_{c(\pi)}}\neq 0$.
            \item We have $\dim\ \pi^{K_{c(\pi)}}=1.$
        \end{enumerate}
    \end{thm}
We always have $c(\pi)=c(\pi^\vee)$. In light of \Cref{thm local new vectors}, we will identify such $\pi$ with its Whittaker model $\mathcal{W}(\pi,\psi)$ (or $\mathcal{W}(\pi,\psi^{-1})$), \textcolor{black}{where we recall that the additive character $\psi$ has been fixed, and has conductor $\mathcal{O}$. With this in mind, we define the following essential vectors.}
\textcolor{black}{\begin{defn}
    The vector $W_\pi^\mathrm{new}$ is the unique normalized new vector in $\pi$ invariant under $K_{c(\pi)}$ that satisfies $W_\pi^\mathrm{new}(1)=1$.
\end{defn}
\begin{defn}
    The vector $W_\pi^\mathrm{new^*}$ is the unique normalized conjugate-new vector in $\pi$ invariant under $\left[\begin{smallmatrix}
        & \varpi^{-c(\pi)}\\
        1 & 
\end{smallmatrix}\right]K_{c(\pi)}\left[\begin{smallmatrix}
        & 1\\
        \varpi^{c(\pi)} & 
    \end{smallmatrix}\right]$ that satisfies $W_\pi^\mathrm{new^*}(1)=1$.
\end{defn}
\begin{rem}
    We remark that the superscript ``*'' above, is used the other way round in \cite[Definitions $2.1\ \&\ 2.2$]{Saha_2015}. Moreover, the fact that both $W_\pi^\mathrm{new}$ and $W_\pi^\mathrm{new^*}$ are supported at the identity relies on $\psi$ (and hence also $\psi^{-1}$) having conductor $\mathcal{O}$, and follows from \cite[Lemma $2.5$]{Saha_2015} which itself extrapolates from \cite{schmidt2002some}.
\end{rem}}
When $\pi$ has unramified central character these two new vectors coincide. Even when the central character of $\pi$ exhibits ramification, we have the following very useful result.
    \begin{prop}[\cite{Saha_2015} Corollary $2.27$]\label{prop new vs conj. new}
        Let $\pi$ be an irreducible, admissible, unitary, generic $H(F)$-representation with central character  $\omega_\pi$ satisfying $\omega_\pi(\varpi)=1$. Then 
        $$W_\pi^\mathrm{new}(g)=\omega_\pi(\det(g))\cdot W_{\pi^\vee}^{\mathrm{new}^*}(g),\ \ g\in H(F).$$
    \end{prop}
     \begin{rem}The assumption that $\omega_\pi(\varpi)=1$ is not restrictive at all since every irreducible \textcolor{black}{representation} is of this form up to an unramified twist. 
\end{rem}
\subsection{Local Rankin-Selberg integral and $L$-factors} Given two representations $\pi_1,\pi_2$ as in \Cref{thm local new vectors}, we identify them with the Whittaker models $\mathcal{W}(\pi_1,\psi)$ and $\mathcal{W}(\pi_2,\psi^{-1})$ respectively. Then, for any Schwartz function $\phi\in\mathcal{S}(F^2)$ and $W_1\in\mathcal{W}(\pi_1,\psi), W_2\in \mathcal{W}(\pi_2,\psi^{-1})$, one has the local Rankin-Selberg zeta-integral of Jacquet-Langlands (see also \cite{jacquet1983rankin}):
$$Z(\phi,W_1,W_2;s):=\int_{N(F)\backslash H(F)}W_1(h)W_2(h)\phi((0,1)h)|\det(h)|^s\ dh\in L(\pi_1\times \pi_2,s)\mathbf{C}[q^s,q^{-s}]\subseteq \mathbf{C}(q^s,q^{-s}).$$
Here $L(\pi_1\times \pi_2,s)$ denotes the local Rankin-Selberg convolution $L$-factor. By \cite{automorphic_on_GL_II} and \cite{gelbart1976relation}, it is given explicitly by 
\begin{align}
   \begin{dcases}
       L(\pi_1\otimes\chi,s)L(\pi_1\otimes \mu,s),& \mathrm{if}\ \pi_2=I(\chi,\mu)\\
       L(\pi_1\otimes \chi,s),& \mathrm{if}\ \pi_2=\mathrm{St}(\chi,\mu)\ \mathrm{and}\ \pi_1\ \mathrm{is}\ \mathrm{not}\ \mathrm{a}\ \mathrm{Steinberg}\ \mathrm{representation}\\
       L(\chi_1\chi_2)L(\mu_1\chi_2,s),& \mathrm{if}\ \pi_1=\mathrm{St}(\chi_1,\mu_1)\ \mathrm{and}\ \pi_2=\mathrm{St}(\chi_2,\mu_2)\\
       L(|\cdot|^{-s_0},s),&\mathrm{if}\ \pi_1,\pi_2\ \mathrm{are}\ \mathrm{supercuspidal}, \pi_2^\vee=\pi_1\otimes|\cdot|^{s_0}\ \mathrm{and}\ \pi_1\otimes|\cdot|^{s_0}\neq \pi_1\otimes\eta|\cdot|^{s_0}\\
       L(|\cdot|^{-2s_0},2s),&\mathrm{if}\ \pi_1,\pi_2\ \mathrm{are}\ \mathrm{supercuspidal}, \pi_2^\vee=\pi_1\otimes|\cdot|^{s_0} \ \mathrm{and}\ \pi_1\otimes|\cdot|^{s_0}= \pi_1\otimes\eta|\cdot|^{s_0}\\
       1 ,& \mathrm{otherwise}
    \end{dcases}
\end{align}
where each individual $L$-factor is given as in \eqref{eq: dirichlet l-factor} and \eqref{eq: gl2 l-factor}. The character $\eta$ is the unramified quadratic character of $F^\times$. If $\Pi=\pi_1\times\pi_2$ is an irreducible admissible generic representation of $G(F)$, we write $W_\Pi^\mathrm{new}:=W_{\pi_1}^\mathrm{new}\times W_{\pi_2}^\mathrm{new}$ and $K_\Pi:=K_{c(\pi_1)}\times K_{c(\pi_2)}$. \textcolor{black}{Moreover, $\omega_{\pi_i}$ always denotes the central character of $\pi_i$, $\omega_\Pi(x_1,x_2):=\omega_{\pi_1}(x_1)\omega_{\pi_2}(x_2)$ for $x_i\in F^\times$, and $\omega_\Pi(x):=\omega_\Pi(x,x)$ for $x\in F^\times.$} Finally, if $g=(g_1,g_2)$ is any element of $G(F)$, then we sometimes write $Z(\phi,gW_\Pi^\mathrm{new};s)$ for the zeta-integral $Z(\phi,g_1W_{\pi_1}^\mathrm{new},g_2W_{\pi_2}^\mathrm{new};s)$. We are interested in the integral-away-from-$q$ behavior of $Z(\phi,gW_\Pi^\mathrm{new};s)$ evaluated on certain \textit{integral data} $(\phi,g)$, which we define similarly to \cite[\S $6.4$]{loeffler2021zetaintegralsunramifiedrepresentationsgsp4}.
\begin{defn}\label{def integral datum}
 A datum $ (\phi,g)\in\mathcal{S}(F^2) \times G(F)$ is $\Pi$-\textit{integral} if $\phi$ is valued in the ideal
    $$\frac{1}{\vol(\mathrm{Stab}(\phi)\cap gK_\Pi g^{-1})}\cdot\mathbf{Z}\subseteq\mathbf{Z}.$$
    \noindent Here $\mathrm{vol}(-)$ denotes the normalized Haar measure on $H(F)$ giving $H(\mathcal{O})$ volume $1$, and the action of $H(F)$ on $\mathcal{S}(F^2)$ is via right translation.
\end{defn}
\section{Reinterpreting the local Rankin-Selberg integral}\label{sec 3}
\subsection{The general construction}\label{sec: general construction}
Let $\Pi=\pi_1\times\pi_2$ be an irreducible admissible generic $G(F)$-representation. The local Rankin-Selberg zeta-integral gives rise to a linear map
$$\mathcal{Z}_\Pi(-;s):\mathcal{S}(F^2)\otimes C_c^\infty(G(F)/K_\Pi)\longrightarrow L(\Pi,s)\mathbf{C}[q^s,q^{-s}]$$
characterized by $\mathcal{Z}_\Pi(\phi\otimes \ch(gK_\Pi);s):=Z(\phi,\ch(gK_\Pi)\cdot W_\Pi^\mathrm{new}\textcolor{black}{;s})=\mathrm{Vol}(K_\Pi)\cdot Z(\phi,g_1W_{\pi_1}^\mathrm{new},g_2W_{\pi_2}^\mathrm{new}\textcolor{black}{;s})$. Throughout this section, we will always regard 
 $\mathcal{S}(F^2)$, respectively $C_c^\infty(G(F)/K_\Pi)$, as smooth left $H(F)$-representations via the actions $h\cdot\phi=\phi((-)h)$ for $\phi\in\mathcal{S}(F^2)$, respectively $h\cdot \xi=\xi(h^{-1}(-))$ for $\xi\in C_c^\infty(G(F)/K_\Pi)$. The construction of \cite[\S $3$]{groutides2024rankinselbergintegralstructureseuler} gives us a linear map $$\Xi:\mathcal{S}(F^2)\otimes C_c^\infty(G(F)/K_\Pi)\longrightarrow \left(\mathrm{Ind}_{P(F)}^{G(F)}\mathbf{1}\right)^{K_\Pi}$$
$$\Xi(\phi\otimes \xi)(g):=\int_{H(F)}\textcolor{black}{(h\cdot\xi)}(g)\phi((0,1)h)\ dh$$
where $dh$ denotes the normalized Haar measure on $H(F)$ giving $H(\mathcal{O})$ volume $1$. Convergence of the integral follows from compactness of the open subgroup $H(F)\cap \gamma K_\Pi\gamma^{-1}\subseteq H(F)$ for any $\gamma\in G(F)$. We can then compose $\Xi$ with the linear endomorphism of $\left(\mathrm{Ind}_{P(F)}^{G(F)}\mathbf{1}\right)^{K_\Pi}$ given by $$f(-)\mapsto f(-)- f(\left[\begin{smallmatrix}
    \varpi & \\
    & \varpi
\end{smallmatrix}\right]^{-1}(-)).$$If we call this composition $\Xi_c$, then the linear map $\Xi_c$ lands in the compact induction $\left(\mathrm{c}\text{-}\mathrm{Ind}_{P(F)}^{G(F)}\mathbf{1}\right)^{K_\Pi}$. \textcolor{black}{This is beacause $\Xi_c(\phi\otimes\xi)=\Xi(\phi_c\otimes\xi)$, with $\phi_c(-):=\phi(-)-\phi((-)\left[\begin{smallmatrix}
    \varpi& \\
    & \varpi
\end{smallmatrix}\right]^{-1})$ which is always an element of $\mathcal{S}_0(F^2)$; the space of Schwartz functions on $F^2$ that vanish at $(0,0)$. But as in \cite[Page $397$]{jacquet1983rankin}, we have a natural embedding of $H(F)$-representations
$$\mathcal{S}(F^2)\hookrightarrow \mathrm{Ind}_{P(F)}^{H(F)}\mathbf{1},\ \phi\mapsto(h\mapsto \phi((0,1)h))$$
which identifies $\mathcal{S}_0(F^2)$ with the compact induction $c$-$\mathrm{Ind}_{P(F)}^{H(F)}\mathbf{1}$.} 

A similar construction as in \cite[\S $4$]{groutides2024rankinselbergintegralstructureseuler} allows us to define a linear map 
$$\Lambda_\Pi(-;s): \left(\mathrm{c}\text{-}\mathrm{Ind}_{P(F)}^{G(F)}\mathbf{1}\right)^{K_\Pi}\longrightarrow L(\Pi,s)\mathbf{C}[q^s,q^{-\textcolor{black}{s}}]$$
via means of meromorphic continuation,
which is given (for $\Re(s)$ large enough independently of $f$) by
\begin{align*}
\Lambda_\Pi(f;s)
    &:=\int_{P(F)\backslash G(F)} f(g)|\det(g_2)|^s\int_{F^\times} (gW_\Pi^\mathrm{new})(\left[\begin{smallmatrix}
        y & \\
        & 1
    \end{smallmatrix}\right])|y|^{s-1}\ d^\times y\ dg
\end{align*}
where $d^\times y$ is the normalized Haar measure on $F^\times$ giving $\mathcal{O}^\times$ volume $1$, and $dg$ is the right $G(F)$-invariant normalized quotient Haar measure introduced in \cite[\S $4.1$]{groutides2024rankinselbergintegralstructureseuler} following \cite[\S $2.2$]{kurinczuk2017rankin}.
 Gathering things together, we want to establish a relation between the two different paths provided by the following diagram
\[\begin{tikzcd}[ampersand replacement=\&,cramped]
	{\mathcal{S}(F^2)\otimes C_c^\infty(G(F)/K_\Pi)} \&\& {L(\Pi,s)\mathbf{C}[q^s,q^{-s}]} \\
	{\left(\mathrm{c}\text{-}\mathrm{Ind}_{P(F)}^{G(F)}\mathbf{1}\right)^{K_\Pi}} \&\& {L(\Pi,s)\mathbf{C}[q^s,q^{-s}]}
	\arrow["{{\mathcal{Z}_\Pi(-;s)}}", from=1-1, to=1-3]
	\arrow["{{\Xi_c}}"', from=1-1, to=2-1]
	\arrow["{{\Lambda_\Pi(-;s)}}"', from=2-1, to=2-3]
\end{tikzcd}\]
\begin{prop}\label{prop lambda and zeta}
    Let $\Pi=\pi_1\times\pi_2$ be an irreducible admissible generic $G(F)$-representation. Let $g=(g_1,g_2)\in G(F)$ and $\phi\in\mathcal{S}(F^2)$ be arbitrary. There's an equality of rational functions
    \begin{align}\label{eq: prop equality of rational funcctions}\Lambda_\Pi\left(\Xi_c(\phi\otimes \ch(gK_\Pi));s\right)=\scalemath{1}{\mathrm{Vol}(K_\Pi)\cdot q^{-v(\det(g_2))s}(1-\omega_\Pi(\varpi)q^{-2s})} Z(\phi,gW_\Pi^\mathrm{new};s)\end{align}
    in $L(\Pi,s)\mathbf{C}[q^s,q^{-s}]$,
    where $\mathrm{Vol}$ denotes the product measure on $G(F)$.
    \begin{proof}
        The local Rankin-Selberg zeta-integral $Z(\phi,gW_\Pi^\mathrm{new};s)$ is given in the notation above, by $\mathcal{Z}_\Pi(\phi \otimes \ch(gK_\Pi);s)$. By the $H(F)$-equivariance property of the linear map $\mathcal{Z}_\Pi(-;s)$, and the $H(F)$-invariance of the linear map $\Xi_c$, it suffices to establish \eqref{eq: prop equality of rational funcctions} when $g=(g_1,1)$. To shorten notation, we write $$K_{\Pi,\phi,g}:=\mathrm{Stab}(\phi)\cap gK_\Pi g^{-1}\subseteq H(\mathcal{O}).$$ The standard unfolding of the Rankin-Selberg integral \textcolor{black}{(e.g. \cite[Page $14$]{automorphic_on_GL_II})} gives,
        \begin{align}\label{eq: 5} {\mathcal{Z}_\Pi(\phi\otimes\ch(gK_\Pi);s)=\vol(K_{\Pi,\phi,g})\sum_{\kappa\in H(\mathcal{O})/K_{\Pi,\phi,g}}\int_{F^\times}\int_{F^\times} W_\Pi^\mathrm{new}(\left[\begin{smallmatrix}
            y & \\
            & 1
        \end{smallmatrix}\right]\kappa g) \phi((0,x)\kappa) \omega_\Pi(x)|x|^{2s} |y|^{s-1}\ d^\times x\ d^\times y.}
        \end{align}
        where as usual the measures $d^\times x$ and $d^\times y$ are normalized to give $\mathcal{O}^\times$ volume $1$.
        On the the other hand, if we write $f:=\Xi_c(\phi\otimes \ch(gK_\Pi))$, then by construction, $f$ is also equal to $\Xi(\phi_c \otimes \ch(gK_\Pi))$ where now $\phi_c$ denotes the Schwartz function $\phi(-)-\phi((-)\left[\begin{smallmatrix}
            \varpi & \\
            & \varpi
        \end{smallmatrix}\right]^{-1})\in\mathcal{S}(F^2)$. Applying $\Lambda_\Pi(-;s)$, splitting the measure as in \cite[Lemma $4.1.1$]{groutides2024rankinselbergintegralstructureseuler}, and applying a change of variables on $\gamma$, we have 
        \begin{align}\label{eq: 6}
        \Lambda_\Pi(f;s)&=\int_{F^\times}\int_{H(\mathcal{O})} \int_{H(F)} 
        f(x(k\gamma,k))|\det(xk)|^s\int_{F^\times} (x(k\gamma,k)W_\Pi^\mathrm{new})(\left[\begin{smallmatrix}
        y & \\
        & 1
    \end{smallmatrix}\right])|y|^{s-1}\ 
        \  d\gamma\ dk\ d^\times y\ d^\times x\\
        \nonumber &=\vol(K_{c(\pi_2)})\sum_{k} \int_{F^\times}\int_{H(F)} f(x(\gamma,k))|\det(xk)|^s\int_{F^\times} (x(\gamma,k)W_\Pi^\mathrm{new})(\left[\begin{smallmatrix}
        y & \\
        & 1
    \end{smallmatrix}\right])|y|^{s-1}\  d\gamma\ d^\times y\ d^\times x
        \end{align}
        where $k$ runs through $H(\mathcal{O})/K_{c(\pi_2)}.$
        Unraveling the definitions of $\mathcal{W}_{f,\Pi}$ and $f=\Xi(\phi_c\otimes \ch(gK_\Pi))$, and writing $\xi:=\ch(gK_\Pi)=\ch((g_1,1)K_\Pi)$ for convenience, \eqref{eq: 6} is equal to $\vol(K_{c(\pi_2)})$ times
        \begin{align*}
            {\sum_{k}\int_{F^\times}\int_{F^\times}\int_{H(F)}\int_{H(F)}\xi(h^{-1}(\gamma,k))\phi_c((0,x)h) W_\Pi^\mathrm{new}(\left[\begin{smallmatrix}
                y & \\
                & 1
            \end{smallmatrix}\right](\gamma,k))\omega_\Pi(x)|x|^{2s}|y|^{s-1}\ dh\ d\gamma\ d^\times y\ d^\times x}
        \end{align*}
        where $k$ runs through $ H(\mathcal{O})/K_{c(\pi_2)}$.
        For each $k$, the value of $\xi(h^{-1}(\gamma,k))$ is non-zero if and only if $(\gamma,k)\in H(F)gK_\Pi$. If this is the case then $k=hk_2$ and $\gamma=hg_1k_1$ with $h\in H(F)$ and $k_i\in K_{c(\pi_i)}$. In other words, as a function of $\gamma$, $\xi(h^{-1}(\gamma,k))$ is supported on $kK_{c(\pi_2)}g_1K_{c(\pi_1)}.$ It is clear that we have a decomposition
        $$kK_{c(\pi_2)}g_1K_{c(\pi_1)}=\bigsqcup_{u\in K_{c(\pi_2)}/g_1K_{c(\pi_1)}g_1^{-1}\cap\ K_{c(\pi_2)}} kug_1 K_{c(\pi_1)}.$$
        Combining this with the above expression, and recalling that the functions $\xi$ and $W_\Pi^\mathrm{new}$ are fixed under right translations by $K_\Pi$, we see that \eqref{eq: 6} is given by 
        \begin{align*}
            {\mathrm{Vol}(K_\Pi) \sum_{k} \sum_{u}\int_{F^\times}\int_{F^\times}\int_{H(F)}\xi(h^{-1}ku(g_1,1))W_\Pi^\mathrm{new}(\left[\begin{smallmatrix}
                y & \\
                & 1
            \end{smallmatrix}\right]ku(g_1,1))\phi_c((0,x)h)\omega_\Pi(x)|x^{2s}|y|^{s-1}\ dh\ d^\times y\ d^\times x}
        \end{align*}
        where $k$ runs through $H(\mathcal{O})/K_{c(\pi_2)}$ and $u$ runs through $K_{c(\pi_2)}/g_1K_{c(\pi_1)}g_1^{-1}\cap\ K_{c(\pi_2)}$. Simplifying further, and recalling that $\xi=\ch(gK_\Pi)$, this is nothing more than 
        \begin{align*}
            {\mathrm{Vol}(K_\Pi)\sum_\kappa\int_{F^\times}\int_{F^\times}\int_{\kappa\left(H(F)\cap gK_\Pi g^{-1}\right)}W_\Pi^\mathrm{new}(\left[\begin{smallmatrix}
                y & \\
                & 1
            \end{smallmatrix}\right]\kappa g)\phi_c((0,x)h) \omega_\Pi(x) |x|^{2s} |y|^{s-1}\ dh\ d^\times y\ d^\times x}
        \end{align*}
        where $\kappa$ runs through $ H(\mathcal{O})/(g_1K_{c(\pi_1)}g_1^{-1}\cap\ K_{c(\pi_2)})$.
        Since $W_\Pi^\mathrm{new}(\left[\begin{smallmatrix}
                y & \\
                & 1
            \end{smallmatrix}\right]\kappa h g)=W_\Pi^\mathrm{new}(\left[\begin{smallmatrix}
                y & \\
                & 1
            \end{smallmatrix}\right]\kappa g)$ for all $h\in H(F)\cap gK_\Pi g^{-1}$, we can decompose this open compact $g_1K_{c(\pi_1)}g_1^{-1}\cap\ K_{c(\pi_2)}=H(F)\cap gK_\Pi g^{-1}$ into left $K_{\Pi,\phi,g}=\mathrm{Stab}(\phi)\cap gK_\Pi g^{-1}$ cosets and conclude that \eqref{eq: 6} coincides with 
        \begin{align*}
            \mathrm{Vol}(K_\Pi)\vol(K_{\Pi,\phi,g})\sum_{\kappa\in H(\mathcal{O})/K_{\Pi,\phi,g}} \int_{F^\times}\int_{F^\times} W_\Pi^\mathrm{new}(\left[\begin{smallmatrix}
                y & \\
                & 1
            \end{smallmatrix}\right] \kappa g)\phi_c((0,x)\kappa)\omega_\Pi(x)|x|^{2s}|y|^{s-1}\ d^\times y\ d^\times x.
        \end{align*}
        Finally, using the definition of $\phi_c(-)=\phi(-)-\phi((-)\left[\begin{smallmatrix}
            \varpi & \\
            & \varpi
        \end{smallmatrix}\right]^{-1})$ together with a change of variables on $x$, and comparing with \eqref{eq: 5}, the result follows. 
        \end{proof}
\end{prop}

\begin{defn}

    Let $\Pi$ be an irreducible admissible generic $G(F)$-representation. The $K_\Pi$-\textcolor{black}{right-invariant} function $\mathcal{I}_\Pi^\mathrm{new}(-;s)$ is given by
    $$\mathcal{I}_\Pi^\mathrm{new}(g;s):=\int_{F^\times}W_\Pi^\mathrm{new}(\left[\begin{smallmatrix}
        y & \\
        & 1
    \end{smallmatrix}\right]g)|y|^{s-1}\ d^\times y,\ \ g\in G(F).$$
\end{defn}
\begin{rem}
    Using this definition, we can re-write 
    $$\Lambda_\Pi(f;s)=\int_{P(F)\backslash G(F)}f(g)|\det(g_2)|^s\ \mathcal{I}_\Pi^\mathrm{new}(g;s)\ dg,\ \ f\in \left(\mathrm{c}\text{-}\mathrm{Ind}_{P(F)}^{G(F)}\mathbf{1}\right)^{K_\Pi}.$$
\end{rem}
\noindent The following proposition determines the behavior of $\Lambda_\Pi(-;s)$ evaluated on characteristic functions, in terms of $\mathcal{I}_\Pi^\mathrm{new}(-;s)$. In what follows, we write $\vol_{P(F)}$ for the right or left normalized Haar measure on $P(F)$, giving $P(\mathcal{O})$ volume $1$. We will not distinguish between left or right, since the volume factor $\vol_{P(F)}(P(F)\cap gK_\Pi g^{-1})$ is independent of such distinction, for any $g\in G(F)$.
\begin{prop}\label{prop lambda on char functions}
    Let $\Pi=\pi_1\times \pi_2$ be an irreducible admissible generic $G(F)$-representation and $g\in G(F)$. Then, there is an equality of rational functions
    \begin{align}
        \Lambda_\Pi\left(\ch(P(F) g K_\Pi);s\right)=\mathrm{Vol}(K_\Pi)\cdot \vol_{P(F)}(P(F)\cap gK_\Pi g^{-1})^{-1}\cdot q^{-v(\det(g_2))s}\ \mathcal{I}_\Pi^\mathrm{new}(g;s) 
    \end{align}
    in $L(\Pi,s)\mathbf{C}[q^s,q^{-s}].$
    \begin{proof}
       From the double coset structure of  $P(F)\backslash G(F)/K_\Pi$, and the left $P(F)$-equivariance properties of $q^{-v(\det(-))s}\mathcal{I}_\Pi^\mathrm{new}(-;s)$, it is enough to assume that \begin{align}\label{eq: 2.1.4(1)}g=\left[\begin{smallmatrix}
           \varpi^a &\\
           & \varpi^a
       \end{smallmatrix}\right]\cdot(\gamma_0,k_0),\ a\in\mathbf{Z},\ \gamma_0\in H(F),\ k_0\in H(\mathcal{O}).\end{align}
       In this case the volume factor $\vol_{P(F)}(P(F)\cap gK_\Pi g^{-1})^{-1}$ can also be written as $[P(\mathcal{O}):P(\mathcal{O})\cap gK_\Pi g^{-1}]$. 
       For ease of notation, we write $f:=\ch(P(F) g K_\Pi)$. Unfolding $\Lambda_\Pi(f;s)$ as in \eqref{eq: 6}, we have 
       \begin{align}\label{eq: 8}
           {\Lambda_\Pi(f;s)=\vol(K_{c(\pi_2)})\sum_{k\in H(\mathcal{O})/K_{c(\pi_2)}}\int_{H(F)}\int_{F^\times} f(x(\gamma,k)) \mathcal{I}_\Pi^\mathrm{new}(x(\gamma,k);s)|x|^{2s}\ d^\times x\ d\gamma.}
       \end{align}
       Since $H(F)=\sqcup_{i\in\mathbf{Z} }\ P(F)\left[\begin{smallmatrix}
           \varpi^i & \\
           & \varpi^i
       \end{smallmatrix}\right] H(\mathcal{O})$, by looking at the second entry of $f(x(\gamma,k))$, we see that it is supported on $\varpi^a\mathcal{O}^\times$ as a function of $x$. Using this together  with a double change of variables on $\gamma$ and $k$ to eliminate the $\mathcal{O}^\times$ contribution, we see that \eqref{eq: 8} simplifies to
       $${\vol(K_{c(\pi_2)})\sum_{k\in H(\mathcal{O})/K_{c(\pi_2)}}\int_{H(F)} f\left(\left[\begin{smallmatrix}
           \varpi^a & \\
           & \varpi^a
       \end{smallmatrix}\right](\gamma,k)\right) \mathcal{I}_\Pi^\mathrm{new}(\left[\begin{smallmatrix}
               \varpi^a & \\
               &\varpi^a
           \end{smallmatrix}\right](\gamma,k);s)q^{-2as}\  d\gamma.}$$
           It is easy to see that as a function of $\gamma$, $f\left(\left[\begin{smallmatrix}
           \varpi^a & \\
           & \varpi^a
       \end{smallmatrix}\right](\gamma,k)\right)$ is supported on $P(\mathcal{O})\gamma_0K_{c(\pi_1)}$. Decomposing this double coset as 
       $$P(\mathcal{O})\gamma_0K_{c(\pi_1)}=\bigsqcup_{m\in P(\mathcal{O})/ P(\mathcal{O}) \cap \gamma_0K_{c(\pi_1)}\gamma_0^{-1}} m\gamma_0 K_{c(\pi_1)}$$we see that \eqref{eq: 8} is given by
       $$\mathrm{Vol}(K_\Pi)\sum_{k\in H(\mathcal{O})/K_{c(\pi_2)}}\ \ \sum_{m\in P(\mathcal{O})/ P(\mathcal{O}) \cap \gamma_0K_{c(\pi_1)}\gamma_0^{-1} } f(\left[\begin{smallmatrix}
           \varpi^a & \\
           & \varpi^a
       \end{smallmatrix}\right](m\gamma_0,k))\  \mathcal{I}_\Pi^\mathrm{new}(\left[\begin{smallmatrix}
               \varpi^a & \\
               &\varpi^a
           \end{smallmatrix}\right](m\gamma_0,k);s)\ q^{-2as}.$$
           From the $P(F)$-invariance of $f$, the $P(F)$-equivariance property of $\mathcal{I}_\Pi^\mathrm{new}$, and the fact that $m\in P(\mathcal{O})$, one notes that the above expression is independent of each $m\in P(\mathcal{O})/ P(\mathcal{O}) \cap \gamma_0K_{c(\pi_1)}\gamma_0^{-1}$. Thus, \eqref{eq: 8} is in fact equal to
           $$\mathrm{Vol}(K_\Pi)[P(\mathcal{O}): P(\mathcal{O}) \cap \gamma_0K_{c(\pi_1)}\gamma_0^{-1}] \sum_{k\in H(\mathcal{O})/K_{c(\pi_2)}} f(\left[\begin{smallmatrix}
           \varpi^a & \\
           & \varpi^a
       \end{smallmatrix}\right](\gamma_0,k))\  \mathcal{I}_\Pi^\mathrm{new}(\left[\begin{smallmatrix}
               \varpi^a & \\
               &\varpi^a
           \end{smallmatrix}\right](\gamma_0,k);s)\ q^{-2as}.$$ 
           The value of $f(\left[\begin{smallmatrix}
           \varpi^a & \\
           & \varpi^a
       \end{smallmatrix}\right](\gamma_0,k))$ is non-zero if and only if $(\gamma_0,k)$ is contained in $P(F)(\gamma_0,k_0) K_\Pi$. In other words, 
       $$\gamma_0=m\gamma_0k_1,\ k=mk_0k_2,\ m\in P(F),\ k_1\in 
 K_{c(\pi_1)},\ k_2\in K_{c(\pi_2)}.$$
    Hence as a function of $k\in H(\mathcal{O})/K_{c(\pi_2)}$, $f(\left[\begin{smallmatrix}
           \varpi^a & \\
           & \varpi^a
       \end{smallmatrix}\right](\gamma_0,k))$ is equal to $1$ on $$\left(P(\mathcal{O})\cap \gamma_0K_{c(\pi_1)}\gamma_0^{-1}\right)k_0K_{c(\pi_2)}=\bigsqcup_{w\in P(\mathcal{O})\cap \gamma_0K_{c(\pi_1)}\gamma_0^{-1}/P(\mathcal{O})\cap gK_\Pi g^{-1}} wk_0K_{c(\pi_2)}$$
       and $0$ everywhere else. However, for each such $w$, we have  $$f(\left[\begin{smallmatrix}
           \varpi^a & \\
           & \varpi^a
       \end{smallmatrix}\right](\gamma_0,wk_0))\  \mathcal{I}_\Pi^\mathrm{new}(\left[\begin{smallmatrix}
               \varpi^a & \\
               &\varpi^a
           \end{smallmatrix}\right](\gamma_0,wk_0);s)=f(\left[\begin{smallmatrix}
           \varpi^a & \\
           & \varpi^a
       \end{smallmatrix}\right](\gamma_0,k_0))\  \mathcal{I}_\Pi^\mathrm{new}(\left[\begin{smallmatrix}
               \varpi^a & \\
               &\varpi^a
           \end{smallmatrix}\right](\gamma_0,k_0);s).$$ Putting things together and concatenating the two indices, the result now follows at once.
    \end{proof}
\end{prop}
\begin{rem}
    The above proposition essentially recovers expression $(9)$ in \cite[Proposition $4.3.6$]{groutides2024rankinselbergintegralstructureseuler}, in the case where $\Pi$ is spherical. 
\end{rem}

\begin{lem}\label{lem integrality lemma}
    Let $\phi\otimes \ch(g K_\Pi)\in \mathcal{S}(F^2)\otimes C_c^\infty(G(F)/K_\Pi)$ and assume that $(\phi,g)$ is a $\Pi$-integral datum. Then 
    $\Xi_c(\phi\otimes \ch(gK_\Pi))$ takes values in $\mathbf{Z}$, as a function in 
 $C_c^\infty(P(F)\backslash G(F)/K_\Pi)$.
    \begin{proof}
        The proof is identical to \cite[Proposition $3.2$]{groutides2024rankinselbergintegralstructureseuler} upon translating notation.
    \end{proof}
\end{lem}
\subsection{The problem and its reduction}\label{sec reduction}  Let $\Pi=\pi_1\times \pi_2$ be an irreducible admissible generic unitary representation of $G(F)$ and for now let $A=A(\Pi)$ be an avatar for a sufficiently large $\mathbf{Z}[q^{-1}]$-algebra in which $q^{-1}\cdot\#H(\mathcal{O}/\varpi\mathcal{O})=(q+1)(q-1)^2$ is \textit{not} invertible. The ``size'' of $A$ will depend on the class of representations involved but it will always be large enough so that it contains the central character of $\Pi$, and $L(\Pi,s)^{-1}\in A[q^s,q^{-s}]$.\\
\\
\textbf{The Problem:} Given an arbitrary $\Pi$-integral datum $(\phi,g_1,g_2)\in \mathcal{S}(F^2)\times G(F)$ in the sense of \Cref{def integral datum}, do we always have
\begin{align}\label{eq: Problem}\tag{\textbf{P}}
Z(\phi,g_1W_{\pi_1}^\mathrm{new}, g_2W_{\pi_2}^\mathrm{new};s)\overset{?}{\in} L(\Pi,s)A[q^s,q^{-s}]\ ?
\end{align}
In other words, is the entire part of the Rankin-Selberg integral always an $A$-valued Laurent polynomial in the variable $q^{ s}$?
\subsubsection{A commutative diagram}
\noindent Due to \Cref{prop lambda and zeta}, it is natural to also consider the linear map 
\begin{align*}\mathcal{Z}^H_\Pi(-;s):\mathcal{S}(F^2)\otimes C_c^\infty(G(F)/K_\Pi)&\longrightarrow L(\Pi,s)\mathbf{C}[q^s,q^{-s}]\\
\phi\otimes \ch(gK_\Pi)&\mapsto |\det(g_2)|^s \mathcal{Z}_\Pi(\phi\otimes \ch(gK_\Pi);s).\end{align*}
The modified map $\mathcal{Z}_\Pi^H(-;s)$ is $H(F)$-invariant and hence \Cref{prop lambda and zeta} gives a commutative diagram

\[\begin{tikzcd}[ampersand replacement=\&,cramped]
	{\mathcal{S}(F^2)\otimes C_c^\infty(G(F)/K_\Pi)} \&\& {L(\Pi,s)\mathbf{C}[q^s,q^{-s}]} \\
	{\left(\mathrm{c}\text{-}\mathrm{Ind}_{P(F)}^{G(F)}\mathbf{1}\right)^{K_\Pi}} \&\& {L(\Pi,s)\mathbf{C}[q^s,q^{-s}]}
	\arrow["{\mathcal{Z}_\Pi^H(-;s)}", from=1-1, to=1-3]
	\arrow["{\Xi_c}"', from=1-1, to=2-1]
	\arrow["{(1-\omega_\Pi(\varpi)q^{-2s})}", from=1-3, to=2-3]
	\arrow["{\Lambda_\Pi(-;s)}", from=2-1, to=2-3]
\end{tikzcd}\]

\noindent Let $\delta_P$ denote the modular character of the mirabolic of $H$ satisfying $\delta_P d^Lx=d^Rx$ where $d^Lx,d^Rx$ denote the left, respectively right, normalized Haar measure on $P(F)$. We'll use the subscript $(-)_{P(F),\delta_P}$ to denote $\delta_P$-twisted $P(F)$-coinvariants; if $M$ is a $P(F)$-module, we set 
$$M_{P(F),\delta_P}:=M/\langle xm-\delta_P(x)m\ |\ m\in M,\ x\in P(F)\rangle.$$
We have a linear isomorphism $C_c^\infty(P(F)\backslash G(F)/K_\Pi)\simeq C_c^\infty(G(F)/K_\Pi)_{P(F),\delta_P}$ defined on characteristic functions as $\ch(P(F)\gamma K_\Pi)\mapsto \vol_{P(F)}(P(F)\cap \gamma K_\Pi \gamma^{-1})^{-1} \ch(\gamma K_\Pi)$. For more details on why this is well-defined we refer to \cite[Proposition $6.1.1$]{groutides2024rankinselbergintegralstructureseuler}. Moreover, we have the standard Hecke action  $\xi\cdot W_\Pi^\mathrm{new}=\int_{G(F)}\xi(g)gW_\Pi^\mathrm{new}dg$ where the measure $dg$ is normalized to give $G(\mathcal{O})$ volume $1$ as usual. It is clear that $\ch(g K_\Pi)$ acts as $\mathrm{Vol}(K_\Pi)g$. As with $\mathcal{Z}_\Pi$, we also consider a slightly modified version of $\mathcal{I}_\Pi^\mathrm{new}$:
 \begin{align*}
     \mathcal{I}_{\Pi,\delta_P}^\mathrm{new}(-;s)&:C_c^\infty(G(F)/K_\Pi)_{P(F),\delta_P}\longrightarrow L(\Pi,s)\mathbf{C}[q^s,q^{-s}]\\ \mathcal{I}_{\Pi,\delta_P}^\mathrm{new}(\ch(gK_\Pi);s)&:= |\det(g_2)|^s\int_{F^\times} \left(\ch(gK_\Pi)\cdot W_\Pi^\mathrm{new}\right)(\left[\begin{smallmatrix}
         y & \\
         & 1
     \end{smallmatrix}\right])|y|^{s-1}\ d^\times y\\
     &=\mathrm{Vol}(K_\Pi)\cdot |\det(g_2)|^s\mathcal{I}_\Pi^\mathrm{new}(g;s).
 \end{align*}
It is straightforward to verify that $\mathcal{I}_{\Pi,\delta_P}^\mathrm{new}(-;s)$ is well-defined by an unraveling of definitions and change of variables. Combining what's been said so far with \Cref{prop lambda on char functions}, we can expand our commutative diagram to look like
\[\begin{tikzcd}[ampersand replacement=\&,cramped]
	{\mathcal{S}(F^2)\otimes C_c^\infty(G(F)/K_\Pi)} \&\& {L(\Pi,s)\mathbf{C}[q^s,q^{-s}]} \\
	{C_c^\infty(P(F)\backslash G(F)/K_\Pi)} \&\& {L(\Pi,s)\mathbf{C}[q^s,q^{-s}]} \\
	{C_c^\infty(G(F)/K_\Pi)_{P(F),\delta_P}} \&\& {L(\Pi,s)\mathbf{C}[q^s,q^{-s}].}
	\arrow["{{\mathcal{Z}_\Pi^H(-;s)}}", from=1-1, to=1-3]
	\arrow["{{\Xi_c}}"', from=1-1, to=2-1]
	\arrow["{(1-\omega_\Pi(\varpi)q^{-2s})}", from=1-3, to=2-3]
	\arrow["{{\Lambda_\Pi(-;s)}}", from=2-1, to=2-3]
	\arrow["\simeq"', from=2-1, to=3-1]
	\arrow["1", from=2-3, to=3-3]
	\arrow["{\mathcal{I}_{\Pi,\delta_P}^\mathrm{new}(-;s)}", from=3-1, to=3-3]
\end{tikzcd}\]
By \Cref{lem integrality lemma} we know that an element $\phi\otimes \ch(gK_\Pi)\in \mathcal{S}(F^2)\otimes C_c^\infty(G(F)/K_\Pi)$ with $(\phi,g)$ a $\Pi$-integral datum, gets mapped under $\Xi_c$ to an element of $C_c^\infty(P(F)\backslash G(F)/K_\Pi)$ that is integral in the naive sense; i.e. it is a $\mathbf{Z}$-valued function on $P(F)\backslash G(F)/K_\Pi.$
\begin{lem}
    If ${(1-\omega_\Pi(\varpi)q^{-2s})}\tfrac{Z(\phi,g_1W_{\pi_1}^\mathrm{new},g_2 W_{\pi_2}^\mathrm{new};s)}{L(\pi_1\times \pi_2,s)}\in A[q^s,q^{-s}]$, then $\tfrac{Z(\phi,g_1W_{\pi_1}^\mathrm{new},g_2 W_{\pi_2}^\mathrm{new};s)}{L(\pi_1\times \pi_2,s)}\in A[q^s,q^{-s}]$.
    \begin{proof}
        This follows by comparing coefficients since we have, a priori, that $$Z(\phi,g_1W_{\pi_1}^\mathrm{new},g_2 W_{\pi_2}^\mathrm{new};s)L(\pi_1\times \pi_2,s)^{-1}=\Phi(\phi,g_1W_{\pi_1}^\mathrm{new},g_2 W_{\pi_2}^\mathrm{new};q^s)\in \mathbf{C}[q^s,q^{-s}]$$ and $\omega_\Pi\in A^\times$.
    \end{proof}
\end{lem}
\noindent Thus, using the above lemma, the commutativity of the diagram preceding it, and \Cref{lem integrality lemma}, we deduce that the sought for $A$-integrality follows from the corresponding $A$-integrality of $\mathcal{I}_{\Pi,\delta_P}^\mathrm{new}(-;s)$ on the lattice of $C_c^\infty(G(F)/K_\Pi)_{P(F),\delta_P}$ traced out by the image of $C_c^\infty(P(F)\backslash G(F)/K_\Pi,\mathbf{Z})$. More precisely, combining everything so far, we conclude that the following implications hold:
\begin{align}\label{eq: reduction step}
    \begin{tikzcd}[ampersand replacement=\&,sep=tiny]
	\begin{array}{c} \left\{ \substack{ \mathrm{Vol}(K_\Pi)^{-1}\cdot\mathcal{I}_{\Pi,\delta_P}^\mathrm{new}(-;s)\ \text{is in } L(\Pi,s)A[q^s,q^{-s}]\ \text{when}\\ \text{evaluated on the lattice inside }C_c^\infty(G(F)/K_\Pi)_{P(F),\delta_P} \\ \text{corresponding to the image of }C_c^\infty(P(F)\backslash G(F)/K_\Pi,\mathbf{Z})} \right\}  \end{array} \\
	\& \begin{array}{c} \left\{  \substack{  \text{The Rankin-Selberg integral }Z(\phi,g_1W_{\pi_1}^\mathrm{new}, g_2W_{\pi_2}^\mathrm{new};s) \\ \text{is in }L(\Pi,s)A[q^s,q^{-s}]\ \text{for any }\Pi\text{-integral datum }(\phi,g_1,g_2)   } \right\}  \end{array} \\
	\begin{array}{c} \left\{ \substack{ \mathcal{I}_{\Pi}^\mathrm{new}(-;s)\ \text{is in } L(\Pi,s)A[q^s,q^{-s}]\ \text{when evaluated}\\ \text{on elements }\mathrm{vol}_{P(F)}(P(F)\cap\gamma K_\Pi \gamma^{-1})^{-1}\cdot \gamma \\ \text{for any }\gamma\ \text{in}\ P(F)\backslash G(F)/K_\Pi     } \right\}  \end{array}
	\arrow[Rightarrow, from=1-1, to=2-2]
	\arrow[shift right=5, Rightarrow, from=1-1, to=3-1]
	\arrow[shift right=5, Rightarrow, from=3-1, to=1-1]
	\arrow[Rightarrow, from=3-1, to=2-2]
\end{tikzcd}
\end{align}

\section{Proving the reduced problem}\label{sec 4}

The goal of this section is to show that for irreducible admissible generic unitary representations $\pi_1,\pi_2$ of $\GL_2(F)$, the left-hand side of \eqref{eq: reduction step} always holds with $\Pi:=\pi_1\times \pi_2$. In other words that
$$\frac{1}{\vol_{P(F)}(P(F)\cap \gamma K_\Pi\gamma^{-1})}\cdot\mathcal{I}_{\Pi}^\mathrm{new}(\gamma ;s)\in L(\Pi,s)A[q^s,q^{-s}]\subseteq A(q^s,q^{-s})$$
for any $\gamma$ in $P(F)\backslash G(F)/K_\Pi$. \textcolor{black}{Here $A(q^s,q^{-s})$ denotes the fraction field of $A[q^s,q^{-s}]$}. From this, using the reduction steps of \Cref{sec reduction}, we will obtain a positive answer to \eqref{eq: Problem}. Our approach relies on a case-by-case analysis, using the classification of irreducible admissible generic representations of $\GL_2(F)$, and results of \cite{Assing_2018}, \cite{Assing_2019} and \cite{Saha_2015}, on the explicit behavior of the new vector, for any class of such representations.
\begin{defn}[\cite{Saha_2015} Definition $2.12$]
    For $t,k\in\mathbf{Z}$ and $v\in\mathcal{O}^\times$, we write $$g_{t,k,v}:=\left[\begin{smallmatrix}
        \varpi^t & \\
        & 1
    \end{smallmatrix}\right]\left[\begin{smallmatrix}
        & 1 \\
        -1 &
    \end{smallmatrix}\right]\left[\begin{smallmatrix} 
    1 & v\varpi^{-k}\\
    & 1
    \end{smallmatrix}\right].$$
\end{defn}
\begin{lem}[\cite{Saha_2015} Lemma $2.13$]\label{lem general coset decomp}
\textcolor{black}{For any integer $n\geq 0$,} there is a decomposition into pairwise disjoint cosets
$$H(F)=\bigsqcup_{t\in\mathbf{Z}}\ \bigsqcup_{0\leq k\leq \textcolor{black}{n}}\ \bigsqcup_{v\in\mathcal{O}^\times/1+\varpi^{\mathrm{min}\{k,\textcolor{black}{n}-k\}}\mathcal{O}} Z(F)N(F)g_{t,k,v} K_{\textcolor{black}{n}}.$$
\end{lem}
\begin{proof}
    The proof can be found in \cite[\S $3.3$]{cesnavicius2022manin}.
\end{proof}
Using \cite[Proposition $2.23$]{Saha_2015}  as a key tool, Assing (\cite{Assing_2018}, \cite{Assing_2019}) completely characterizes the support and behavior of the Whittaker new vector on each coset representative $g_{t,k,v}$, in the non-spherical case (the spherical case is already well-known). The results of \textit{op.cit.} are once again phrased for the conjugate Whittaker new vector which we denote here by $W_\pi^{\mathrm{new}^*}$, hence we will always ``dualize'' them before applying them, and then implicitly apply \Cref{prop new vs conj. new}. 
\subsection{Case-by-case analysis}
We will now start a case-by-case analysis, distinguishing between various classes of irreducible admissible generic representations of $\GL_2(F)$. The distinction comes from the different behavior of the Whittaker new vector as determined by \cite{Assing_2018}. Most of the time we will assume that the central characters of our representations are trivial on the uniformizer $\varpi$, which is always the case up to an unramified twist. Under this convention, we will consider the following classes of representations:
\begin{enumerate}
    \item Unramified principal-series representation 
    \item Half-ramified principal-series representation 
    \item Fully-ramified principal-series representation (there is a further distinction in this case depending on whether the ramified parts of the characters inducing the representation, are distinct or not)
    \item Unramified Steinberg representation
    \item Ramified Steinberg representation
    \item Supercuspidal representation
\end{enumerate}
and all combinations $\pi_1\times\pi_2$ that can be formed with $\pi_1,\pi_2$ as above. Of course, we do not distinguish between $\pi_1\times \pi_2$ and $\pi_2\times \pi_1$.
\subsubsection{Two unramified principal-series representations}
We start with the case of two unramified principal-series representations. In this case the new vector is the spherical vector. In the unramified case, we have already completely determined the integral behavior of $\Lambda_\Pi$ in \cite[\S $3$]{groutides2024rankinselbergintegralstructureseuler}. For the purposes of \textit{op.cit.}, we work at the level of unramified Hecke modules and invariant periods. As such, we don't establish a direct relation at the level zeta-integrals themselves. Using the results of \Cref{sec: general construction} which establish this relation, we have the following proposition. We write $\mathrm{Sch}_i(x,y):=\frac{x^{i+1}-y^{i+1}}{x-y}$, for the Schur polynomial of homogeneous degree $i$. 
\begin{prop}
    Let $\Pi:=\pi_1\times \pi_2$ where $\pi_1,\pi_2$ are irreducible unramified principal-series representations. Let $A\subseteq\mathbf{C}$ be a $\mathbf{Z}[q^{-1}]$-algebra containing the spherical Hecke eigenvalues of $\Pi$, and its central character. Let $(\phi,g_1,g_2)\in\mathcal{S}(F^2)\times G(F)$ be a $\Pi$-integral datum. Then, \eqref{eq: Problem} holds.
    \begin{proof}
        This follows from \Cref{prop lambda and zeta}, \Cref{prop lambda on char functions}, \Cref{lem integrality lemma}, together with \cite[Lemma $4.3.6$ \& Lemma $4.3.7$]{groutides2024rankinselbergintegralstructureseuler} and their proofs. For a more detailed account on how to link the results of \Cref{sec: general construction} together, we refer to the next case we consider below. Additionally, we implicitly use the fact that the reciprocal of the convolution $L$-factor $L(\Pi,s)^{-1}$ is an element of $A[q^s,q^{-s}]$. This follows with some rearranging from the explicit description given in \eqref{eq: gl2 l-factor} together with the well-known expressions for the spherical Hecke eigenvalues of the $\pi_i$ in terms of their Satake parameters. Additionally, a subtle fact used throughout the proof, is that for any $i\in\mathbf{Z}$, the quantity  $q^{-\tfrac{i}{2}}\mathrm{Sch}_i(\alpha_\pi,\beta_\pi)$ (where $\alpha_\pi,\beta_\pi$ are the Satake parameters) is contained in any $\mathbf{Z}[q^{-1}]$-algebra $A\subseteq \mathbf{C}$ which contains the spherical Hecke eigenvalues of $\pi$. This follows from a simple parity consideration after expanding the Schur polynomial as in \cite[\S $6.3.1$]{groutides2024integral} and using the well-known expressions for these eigenvalues mentioned before.
    \end{proof}
\end{prop}
\begin{rem}
    The exact same result holds for reducible Whittaker-type principal-series representations and the proofs involved all remain the same, as is the case in \cite{groutides2024rankinselbergintegralstructureseuler}.
\end{rem}
\subsubsection{Two Steinberg representations and their unramified twists}\label{unr StxSt}
For ease of notation, we simply write $\mathrm{St}$ for the Steinberg representation $\mathrm{St}(|\cdot|^{1/2},|\cdot|^{-1/2})$. Then every Steinberg representation with unramified central character, is given by a twist $\mathrm{St}_\chi:=\mathrm{St}\otimes \chi$ for some unramified character $\chi$ of $F^\times$ and they all have conductor $1$ by \cite{schmidt2002some}. It is clear that the new vector of $\mathrm{St}_\chi$ is given by $1\otimes W_\mathrm{St}^\mathrm{new}$, or equivalently, in the Whittaker model of $\mathrm{St}_\chi$, by the function $W_{\mathrm{St}_\chi}^\mathrm{new}(g)=\chi(\det(g))W_\mathrm{St}^\mathrm{new}(g)$.

\begin{lem}[\cite{Assing_2018} Lemma $3.2$]\label{lem Assing for steinberg} For $0\leq 
 k\leq 1$ and $v\in \mathcal{O}^\times$, we have
$$W_\mathrm{St}^\mathrm{new}(g_{t,k,v})=\begin{dcases}
    -q^{-t-1}\ &\mathrm{if}\ k=0\ \& \ t\geq -1\\
    q^{-t-2}\psi(-\varpi^{t+1}v^{-1})\ &\mathrm{if}\ k= 1\ \& \ t\geq -2k\\
    0\ &\mathrm{otherwise}.
\end{dcases}$$
\end{lem}

\begin{prop}\label{thm steinberg}
    Let $\chi_1,\chi_2$ be unramified characters of $F^\times$ taking values in some $\mathbf{Z}[q^{-1}]$-algebra $A\subseteq\mathbf{C}$ and let $\Pi:=\mathrm{St}_{\chi_1}\times \mathrm{St}_{\chi_2}$. Let $(\phi,g_1,g_2)\in\mathcal{S}(F^2)\times G(F)$ be a $\Pi$-integral datum. Then, \eqref{eq: Problem} holds.
    \begin{proof}
        For ease of notation we will write $\Pi:=\mathrm{St}_{\chi_1}\times \mathrm{St}_{\chi_2}$ and $g=(g_1,g_2)\in G(F)$. Implicitly in the proof we'll use the fact that $L(\Pi,s)^{-1}\in A[q^s,q^{-s}]$ which is immediate by the explicit expression for the $L$-factor given in \eqref{eq: gl2 l-factor}. Additionally, we will always canonically identify an element of $N(F)$ with the corresponding element of $F$. Given the reduction in \Cref{sec reduction}, we now wish to show that
        \begin{align}
            \vol_{P(F)}(P(F)\cap \gamma K_\Pi \gamma^{-1})^{-1}\cdot \mathcal{I}_{\Pi}^\mathrm{new}(\gamma ;s),\ \ \gamma\in P(F)\backslash G(F)/K_\Pi
        \end{align}
        always lies in $L(\Pi,s)A[q^s,q^{-s}]$.
        We proceed by choosing nice representatives in the sense of \Cref{lem general coset decomp} and using them to study these terms. We distinguish between two cases since the computation in one case is more involved. Before making that distinction, we write $\gamma=(\gamma_1,\gamma_2)$, and we note that for each $i\in\{1,2\}$, we can always write 
        \begin{align}\label{eq: gamma_i}\gamma_i=z_i\cdot n_i\cdot g_{t_i,k_i,1}\cdot u_i\end{align}
        where $z_i\in Z(F),n_i\in N(F), u_i\in K_1$ and $g_{t_i,k_i,1}$ is as in \Cref{lem general coset decomp}. 

\noindent \textit{Case} $1$: In the first case, we assume that one or both of the $k_i\in\{0,1\}$ are $0$. \textcolor{black}{The cases $(k_1,k_2)=(1,0)$ and $(k_1,k_2)=(0,1)$ are symmetric, and the case $(k_1,k_2)=(0,0)$ is exactly the same with only minor notational differences which actually make it simpler. Thus, we} provide the details for when $k_1=1$ and $k_2=0$. If this is the case, then we re-write the double coset $P(F)\gamma K_\Pi$ as 
$$P(F)\gamma K_\Pi=P(F)\left( z_1ng_{t,1,1},z_2g_{0,0,1}\right) K_\Pi$$
for some $n\in N(F)$ and $t\in\mathbf{Z}$. Using this coset representative to evaluate $\mathcal{I}_\Pi^\mathrm{new}(-;s)$, we get 
\begin{align}\label{eq: case 1}
    \mathcal{I}_{\Pi}^\mathrm{new}(\left( z_1ng_{t,1,1},z_2g_{0,0,1}\right);s)&=\int_{F^\times} W_{\mathrm{St}_{\chi_1}}^\mathrm{new}(\left[\begin{smallmatrix}
        y & \\
         & 1
    \end{smallmatrix}\right]z_1ng_{t,1,1})W_{\mathrm{St}_{\chi_2}}^\mathrm{new}(\left[\begin{smallmatrix}
        y & \\
         & 1
    \end{smallmatrix}\right]z_2g_{0,0,1})|y|^{s-1}\ d^\times y\\
    \nonumber &= C\cdot \int_{F^\times}\psi(ny)\chi_1(y)W_{\mathrm{St}}^\mathrm{new}(\left[\begin{smallmatrix}
        y & \\
        & 1
    \end{smallmatrix}\right]g_{t,1,1}) \chi_2(y) W_{\mathrm{St}}^\mathrm{new}(\left[\begin{smallmatrix}
        y & \\
        & 1
    \end{smallmatrix}\right]g_{0,0,1})|y|^{s-1}\ d^\times y\\
   \nonumber &={C\cdot \sum_{i\in \mathbf{Z}} (\chi_1\chi_2)(\varpi)^i q^{i(1-s)} \int_{\mathcal{O}^\times} \psi(n\varpi^i y) W_\mathrm{St}^\mathrm{new}(\left[\begin{smallmatrix}
        y & \\
        & 1
    \end{smallmatrix}\right]g_{t+i,1,1})W_\mathrm{St}^\mathrm{new}(\left[\begin{smallmatrix}
        y & \\
        & 1
    \end{smallmatrix}\right]g_{i,0,1})\ d^\times y}\\
    \nonumber&={C\cdot \sum_{i\in \mathbf{Z}} (\chi_1\chi_2)(\varpi)^i q^{i(1-s)} \int_{\mathcal{O}^\times} \psi(n\varpi^i y) W_\mathrm{St}^\mathrm{new}(g_{t+i,1,y^{-1}})W_\mathrm{St}^\mathrm{new}(g_{i,0,y^{-1}})\ d^\times y}
\end{align}
where the last equality follows from the easy to check identity $$\left[\begin{smallmatrix}
            y & \\
            &1
        \end{smallmatrix}\right]\left[\begin{smallmatrix}
             & 1\\
             -1 & 
        \end{smallmatrix}\right]\left[\begin{smallmatrix}
            1 & \varpi^{-k}\\
            & 1
        \end{smallmatrix}\right]=\left[\begin{smallmatrix}
            y & \\
            &y
        \end{smallmatrix}\right]\left[\begin{smallmatrix}
             & 1\\
             -1 & 
        \end{smallmatrix}\right]\left[\begin{smallmatrix}
            1 & y^{-1}\varpi^{-k}\\
            & 1
        \end{smallmatrix}\right]\left[\begin{smallmatrix}
            y^{-1} & \\
            &1
        \end{smallmatrix}\right]$$ 
        and
        $C:=\chi_1(\det(z_1g_{t,1,1}))\chi_2(\det(z_2))\in A$. Now, we are in a position to apply \Cref{lem Assing for steinberg}. We see that the value of the right-most Whittaker function is identically zero if $i\leq -2$ and the left-most one is zero if $i\leq -3-t$. Thus, if we let $i_0:=\mathrm{max}\{-1,-t-2\}$, we have
        \begin{align}
          \nonumber  \mathcal{I}_{\Pi}^\mathrm{new}(\left( z_1ng_{t,1,1},z_2g_{0,0,1}\right);s)&=-Cq^{-t-3} \cdot \sum_{i\geq i_0} (\chi_1\chi_2)(\varpi)^i q^{-i(1+s)}  \int_{\mathcal{O}^\times} \psi\left(\left(n-\varpi^{t+1}\right)\varpi^iy\right)\ d^\times y\\
            &={-Cq^{-t-3} \cdot\left\{\left(\sum_{i=i_0}^{-v_0-1}(\chi_1\chi_2)(\varpi)^iq^{-i(1+s)}\epsilon_i(v_0)\right) + \left(\sum_{i\geq i_0}(\chi_1\chi_2)(\varpi)^i q^{-i(1+s)}\right)    \right\}}\label{eq: 14}
        \end{align}
        where \begin{align}\label{eq: epsilon}v_0:=v(n-\varpi^{t+1})\ \ \mathrm{and}\ \ \epsilon_i(v_0):=\begin{dcases}
            -1, &\ \mathrm{if}\ i_0\leq i < -v_0-1\\
            \tfrac{-q}{q-1}, &\ \mathrm{if}\ i= -v_0-1 
        \end{dcases}\end{align}
        and the finite sum in \eqref{eq: 14} is zero by convention if $i_0>-v_0-1$. This last step follows from \Cref{sec Gauss sums} and the same splitting of the infinite sum as in \cite[\S $6.3$]{groutides2024integral}; that is, a finite ``error'' term and an $L$-fctor contribution. The last observation we make is that the right-most sum in \eqref{eq: 14} is nothing more than $(\chi_1\chi_2)(\varpi)^{i_0}q^{-i_0(1+s)} L(\mathrm{St}_{\chi_1}\times \mathrm{St}_{\chi_2},s)L(\chi_1\chi_2,s)^{-1}\in L(\mathrm{St}_{\chi_1}\times \mathrm{St}_{\chi_2},s)A[q^s,q^{-s}]$. To conclude the first case it suffices to show that the volume factor
        $$\vol_{P(F)}\left(P(F)\cap \left( z_1ng_{t,1,1},z_2g_{0,0,1}\right) K_\Pi \left( z_1ng_{t,1,1},z_2g_{0,0,1}\right)^{-1}\right)^{-1}$$
        is an integral multiple of $q-1$. Noting that this subgroup is given by $\mathcal{U}_0:=P(F)\cap ng_{t,1,1} K_1 g_{t,1,1}^{-1}n^{-1}\cap g_{0,0,1}K_1g_{0,0,1}^{-1}$, which is trivially contained in $P(F)\cap g_{0,0,1}K_1g_{0,0,1}^{-1}=P(F)\cap \left[\begin{smallmatrix}
            & 1\\
            -1 & 
        \end{smallmatrix}\right]K_1 \left[\begin{smallmatrix}
            & -1\\
            1 & 
        \end{smallmatrix}\right]$. Thus, an element $x=\left[\begin{smallmatrix}
            a & b \\
            & 1
        \end{smallmatrix}\right]\in\mathcal{U}_0$ always satisfies 
$$\left[\begin{smallmatrix}
            & -1\\
            1 & 
        \end{smallmatrix}\right]x\left[\begin{smallmatrix}
            & 1\\
            -1 & 
        \end{smallmatrix}\right]=\left[\begin{smallmatrix}
            1 & \\
            -b & a
        \end{smallmatrix}\right]\in K_1.$$
        Hence $\mathcal{U}_0$ is a subgroup of $\left[\begin{smallmatrix}
            1+\varpi\mathcal{O} & \varpi \mathcal{O}\\
            & 1
        \end{smallmatrix}\right]\subseteq P(\mathcal{O})$ which has index in $P(\mathcal{O})$ equal to $(q-1)q$ as required.\\
        
        \noindent \textit{Case} $2$: In the second and last case we assume that $k_1,k_2$ are both equal to $1$. Starting again with \eqref{eq: gamma_i}, we assume without loss of generality that $t_1\geq t_2$. We then re-write the double coset $P(F)\gamma K_\Pi$ as 
        $$P(F)\gamma K_\Pi= P(F) (z_1ng_{t,1,1},z_2g_{0,1,1}) K_\Pi$$
        with $n\in N(F)$ and $t\in\mathbf{Z}_{\geq 0}$. The integral manipulations in \eqref{eq: case 1} work in the same way, and we get 
        \begin{align}\label{eq: expression case 2}\mathcal{I}_\Pi^\mathrm{new}((z_1ng_{t,1,1},z_2g_{0,1,1});s)=C\cdot \sum_{i\in\mathbf{Z}}(\chi_1\chi_2)(\varpi)^iq^{i(1-s)}\int_{\mathcal{O}^\times} \psi(n\varpi^iy) W_\mathrm{St}^\mathrm{new}(g_{t+i,1,y^{-1}})W_\mathrm{St}^\mathrm{new}(g_{i,1,y^{-1}})\ d^\times y.\end{align}
        Since in this case $t\geq 0$, it follows from \Cref{lem Assing for steinberg}, that both Whittaker functions are non-zero for $i\geq -2$, and \eqref{eq: expression case 2} is equal to
        \begin{align*}
            Cq ^{-t-4}\cdot \sum_{i\geq -2}  (\chi_1\chi_2)(\varpi)^iq^{-i(1+s)}\int_{\mathcal{O}^\times} \psi((n+\varpi-\varpi^{t+1})\varpi^iy)\ d^\times y.
        \end{align*}
        The same splitting as in \eqref{eq: 14}, this time gives
        \begin{align}
            Cq^{-t-4}\left\{\left(\sum_{i=-2}^{-v_1-1} (\chi_1\chi_2)(\varpi)^iq^{-i(1+s)}\epsilon_i(v_1)\right)+\left(\sum_{i\geq -2} (\chi_1\chi_2)(\varpi)^iq^{-i(1+s)}\right)\right\}
        \end{align}
        where now
        $v_1:=v(n+\varpi-\varpi^{t+1})$, $\epsilon_i(v_1)$ is given as before, and once again the finite sum above is taken to be zero by convention if $-2>-v_1-1$. The $L$-function contribution from the right-most sum is the same as in the first case and requires no further treatment. \textcolor{black}{The finite sum on the left is zero by convention if $v_1>1$, and it is an element of $\tfrac{1}{q-1}A[q^s,q^{-s}]$ if $v_1\leq 1$.} In the latter case we again need to show that \textcolor{black}{the factor of $\tfrac{1}{q-1}$} is always killed by the volume factor
        $$\vol_{P(F)}\left(P(F)\cap (z_1ng_{t,1,1},z_2g_{0,1,1})K_\Pi (z_1ng_{t,1,1},z_2g_{0,1,1})^{-1}\right)^{-1}.$$
        Let's denote this subgroup by $\mathcal{U}_1$. We'll once again show that $\mathcal{U}_1\subseteq \left[\begin{smallmatrix}
            1+\varpi\mathcal{O} & \varpi \mathcal{O}\\
            & 1
        \end{smallmatrix}\right]\subseteq P(\mathcal{O})$ whenever $v_1\leq 1$. Of course, $\mathcal{U}_1$ is given by 
        $\mathcal{U}_1=P(F)\cap ng_{t,1,1} K_1 g_{t,1,1}^{-1}n^{-1}\cap g_{0,1,1} K_1 g_{0,1,1}^{-1}.$ 
        Let $x=\left[\begin{smallmatrix}
            a & b \\
            & 1
        \end{smallmatrix}\right]\in\mathcal{U}_1$. Recall that $g_{0,1,1}=\left[\begin{smallmatrix}
            & 1\\
            -1& 
        \end{smallmatrix}\right]\left[\begin{smallmatrix}
            1 & \varpi^{-1}\\
            & 1
        \end{smallmatrix}\right]$. Thus, we must have
        $$g_{0,1,1}^{-1}xg_{0,1,1}=\left[\begin{smallmatrix}
            1+b\varpi^{-1} & \varpi^{-1}(1+b\varpi^{-1}-a)\\
            -b & a-b\varpi^{-1}
        \end{smallmatrix}\right]\in K_1.$$
        This forces $b\in\varpi\mathcal{O}$ and $a-b\varpi^{-1}\in 1+\varpi\mathcal{O}$. We now look at the second inclusion. Recall that $ng_{t,1,1}=\left[\begin{smallmatrix}
            1 & n\\
            & 1
        \end{smallmatrix}\right]\left[\begin{smallmatrix}
            \varpi^t & \\
            & 1
        \end{smallmatrix}\right]\left[\begin{smallmatrix}
            & 1\\
            -1& 
        \end{smallmatrix}\right]\left[\begin{smallmatrix}
            1 & \varpi^{-1}\\
            & 1
        \end{smallmatrix}\right]$. A lengthy calculation shows that $x$ must satisfy
        \begin{align}\label{eq: matrix}g_{t,1,1}^{-1}n^{-1}xng_{t,1,1}=\left[\begin{smallmatrix}
            * & *\\
            (-an-b+n)\varpi^{-t} & a-(an+b-n)\varpi^{-t-1}
        \end{smallmatrix}\right]\in K_1.
        \end{align}
        Looking at the bottom left entry,  we see that $n(a-1)+b\in \varpi\mathcal{O}$ and hence $n(a-1)\in\varpi\mathcal{O}$. If $v_1\leq 0$, which holds if and only if $v(n)\leq 0$, then we are done as $a-1\in\varpi\mathcal{O}$ (we already know, a priori, that $a\in\mathcal{O}^\times$). Hence, we assume that $v_1=1$. In this case, $\varpi^{t}-n\varpi^{-1}-1\in\mathcal{O}^\times$. Looking at the bottom right entry of \eqref{eq: matrix}, we see that
    \begin{align*}
        a\varpi^t-an\varpi^{-1}&\equiv \varpi^t+b\varpi^{-1}-n\varpi^{-1}\mod\varpi\\
        &\equiv \varpi^t+a-1-n\varpi^{-1}\mod\varpi.
    \end{align*}
    Finally, this gives $a(\varpi^t -n\varpi^{-1}-1)\equiv \varpi^t-n\varpi^{-1}-1\mod\varpi$. The result follows since $\varpi^{t}-n\varpi^{-1}-1$ is a unit.
    \end{proof}
\end{prop}
\subsubsection{One unramified twist of a Steinberg representation and one unramified principal-series representation}

 In this section, we treat the case $\Pi=\mathrm{St}_\chi\times \pi$ where $\chi:F^\times\rightarrow \mathbf{C}^\times$ is an unramified character and $\pi$ is an unramified principal-series representation with Satake parameters $\alpha_\pi,\beta_\pi$. It is thus clear that $K_\Pi=K_1\times H(\mathcal{O})$. 
 \begin{prop}\label{thm unram + steinberg}
     Let $\Pi:=\mathrm{St}_\chi\times \pi$ where $\chi:F^\times\rightarrow \mathbf{C}^\times$ is an unramified character and $\pi$ is an unramified principal-series representation. Let $A\subseteq \mathbf{C}$ be a $\mathbf{Z}[q^{-1}]$-algebra containing $\chi$, the spherical Hecke eigenvalues of $\pi$ and its central character. Let $(\phi,g_1,g_2)\in \mathcal{S}(F^2)\times G(F)$ be a $\Pi$-integral datum. Then, \eqref{eq: Problem} holds.
     \begin{proof}
         Again, implicitly we'll use that $L(\Pi,s)^{-1}$ is contained in $A[q^s,q^{-s}]$. This follows from \eqref{eq: gl2 l-factor}. We're again interested in the terms 
         \begin{align}
            \vol_{P(F)}(P(F)\cap \gamma K_\Pi \gamma^{-1})^{-1}\cdot \mathcal{I}_{\Pi}^\mathrm{new}(\gamma ;s),\ \ \gamma\in P(F)\backslash G(F)/K_\Pi.
        \end{align}
        By Iwasawa decomposition and \Cref{lem general coset decomp}, we can work with coset representatives of the following form
         $$(z_1g_{t,k,1},z_2n),\ \ z_i\in Z(F),n\in N(F),t\in\mathbf{Z},k\in\{0,1\}.$$
        Let $C(k):=\chi(\det(z_1g_{t,k,1}))\omega_{\pi}(z_2)\in A$, $i_0:=\mathrm{max}\{0,-t-1\}$ and $i_1:=\mathrm{max}\{0,-t-2\}$. We will again canonically identify $N(F)$ with $F$ where convenient. We evaluate
        \begin{align*}
           \mathcal{I}_\Pi^\mathrm{new}((z_1g_{t,k,1}, z_2n);s)&=
            C\cdot \int_{F^\times}\psi(-ny)\chi(y)W_{\mathrm{St}}(\left[\begin{smallmatrix}
                y & \\
                & 1
            \end{smallmatrix}\right]g_{t,k,1})W_\pi^\mathrm{sph}(\left[\begin{smallmatrix}
                y & \\
                 & 1
            \end{smallmatrix}\right])|y|^{s-1}\ d^\times y\\
           &= {C\cdot \sum_{i\in\mathbf{Z}_{\geq 0}} \chi(\varpi)^i q^{-i(s-1)}W_\pi^\mathrm{sph}(\left[\begin{smallmatrix}
                \varpi^i & \\
                & 1
            \end{smallmatrix}\right])\int_{\mathcal{O}^\times}\psi(-n\varpi^iy)W_\mathrm{St}^\mathrm{new}(g_{t+i,k,y^{-1}})\ d^\times y}\\
           &={\begin{dcases}
                -Cq^{-t-1}\cdot\sum_{i\geq i_0} \chi(\varpi)^iq^{-i(s+\frac{1}{2})}\mathrm{Sch}_i(\alpha_\pi,\beta_\pi)\int_{\mathcal{O}^\times}\psi(-n\varpi^iy)\ d^\times y,\ &\mathrm{if}\ k=0\\
                Cq^{-t-2}\cdot \sum_{i\geq i_1} \chi(\varpi)^iq^{-i(s+\frac{1}{2})}\mathrm{Sch}_i(\alpha_\pi,\beta_\pi)\int_{\mathcal{O}^\times}\psi((-\varpi^{t+1}-n)\varpi^iy)\ d^\times y,\ &\mathrm{if}\ k=1\\
            \end{dcases}}
        \end{align*}
        where in the last equality we have used \Cref{lem Assing for steinberg} for the new vector in the Steinberg representation, and \cite{shintani1976explicit} to express the spherical vector in terms of Satake parameters. Write $D_k$ to denote the constant $-Cq^{-t-1}$ if $k=0$ and $Cq^{-t-2}$ if $k=1$. Using our usual splitting into an $L$-factor term and a finite ``error'' term, we have that $\mathcal{I}_\Pi^\mathrm{new}((z_1g_{t,k,1}, z_2n);s)$ is equal to
        \begin{align*}
            {D_k\sum_{i\geq i_k}\chi(\varpi)^iq^{-i(s+\frac{1}{2})}\mathrm{Sch}_i(\alpha_\pi,\beta_\pi)+D_k\begin{dcases}
                \sum_{i=i_0}^{-v(n)-1}\chi(\varpi)^iq^{-i(s+\frac{1}{2})}\mathrm{Sch}_i(\alpha_\pi,\beta_\pi)\epsilon_i(n),\ &\mathrm{if}\ k=0\\
                \sum_{i=i_1}^{-v(\varpi^{t+1}-n)-1}\chi(\varpi)^iq^{-i(s+\frac{1}{2})}\mathrm{Sch}_i(\alpha_\pi,\beta_\pi)\epsilon_i(-\varpi^{t+1}-n),\ &\mathrm{if}\ k=1
            \end{dcases}
            }
        \end{align*}
        where both finite sums are zero by convention if $i_0>-v(n)-1$ and $i_1>-v(\varpi^{t+1}-n)$ respectively, and $\epsilon_i(-)$ is defined in the same way as in \eqref{eq: epsilon}. 
        A lengthy but standard calculation shows that the left-most term is precisely equal to 
        $${L(\Pi,s)\cdot D_k\chi(\varpi)^{i_k}q^{-i_ks}\cdot\left\{q^{\frac{-i_k}{2}}\mathrm{Sch}_{i_k}(\alpha_\pi,\beta_p)-\chi(\varpi)\omega_\pi(\varpi)q^{-(s+1)}q^{\frac{-(i_k-1)}{2}}\mathrm{Sch}_{i_k-1}(\alpha_\pi,\beta_\pi)\right\}}$$
        which is an element of $L(\Pi,s)A[q^s,q^{-s}].$
        Thus, if the finite sum on the right is zero, we are done. If not, then we have a condition on the valuation of $n$, respectively $-\varpi^{t+1}-n$. We only need to show that under this condition, the factor of $\frac{1}{q-1}$ occurring in the last $\epsilon$-term, is always counteracted by the volume factor
        $$\vol_{P(F)}\left(P(F)\cap (z_1g_{t,k,1}, z_2n) K_\Pi (z_1g_{t,k,1}, z_2n)^{-1}\right)^{-1}.$$
        Let $\mathcal{U}_k:=P(F)\cap (z_1g_{t,k,1}, z_2n) K_\Pi (z_1g_{t,k,1}, z_2n)^{-1}$.
        If $k=0$, then this is in fact independent of the valuation condition on $n$. We re-write $\vol_{P(F)}(\mathcal{U}_0)$ as $$\vol_{P(F)}\left(P(F)\cap (g_{0,0,1}, \left[\begin{smallmatrix}
            \varpi^{-t} & \\
            & 1
        \end{smallmatrix}\right]n)K_\Pi (g_{0,0,1}, \left[\begin{smallmatrix}
            \varpi^{-t} & \\
            & 1
        \end{smallmatrix}\right]n)^{-1}\right)$$ 
        up to an inconsequential power of $q$. In this case, an element $x=\left[\begin{smallmatrix}
            a & b\\
            & 1
        \end{smallmatrix}\right]\in P(F)$, contained in the subgroup above, satisfies $g_{0,0,1}^{-1}xg_{0,0,1}=\left[\begin{smallmatrix}
            1 & \\
            -b & a
        \end{smallmatrix}\right]\in K_1$. Thus, in this case, we are done. If $k=1$, the finite sum is non-zero if and only if $-v(-\varpi^{t+1}-n)=-v(\varpi^{t+1}+n)\geq i_1 +1$. In this case, we re-write $\vol_{P(F)}(\mathcal{U}_1)$ as 
        $$\vol_{P(F)}\left( P(F)\cap (n^{-1}g_{t,1,1},1)K_\Pi (n^{-1}g_{t,1,1},1)^{-1}\right)=\vol_{P(F)}\left( P(\mathcal{O})\cap n^{-1}g_{t,1,1}K_1 g_{t,1,1}^{-1}n\right).$$
        An element $x=\left[\begin{smallmatrix}
            a & b \\
            & 1
        \end{smallmatrix}\right]\in P(\mathcal{O})\cap n^{-1}g_{t,1,1}K_1 g_{t,1,1}^{-1}n$ satisfies 
        \begin{align}\label{eq: A}g_{t,1,1}^{-1}nxn^{-1} g_{t,1,1}=\left[\begin{smallmatrix}
            * & *\\
            (an-b-n)\varpi^{-t} & a+(an-b-n)\varpi^{-t-1}
        \end{smallmatrix}\right]\in K_1.\end{align}
       If $t\geq -1$, then since $-v(\varpi^{t+1}+n)\geq i_1+1\geq 1,$ implies that $v(\varpi^{t+1}+n)\leq -1$, and thus $v(n)\leq -1$. Looking at the bottom left entry of \eqref{eq: A}, we see that under these conditions, $a\equiv 1\mod\varpi$ and hence we are done. Thus, we assume that $t\leq -2$. If $v(n)<t+1$, then once again, the bottom left entry of \eqref{eq: A} shows that $a\equiv 1\mod\varpi$ and we are done. If $v(n)\geq t+1$, then the bottom right entry of \eqref{eq: A}, shows that 
       $$a(1+n\varpi^{-t-1})\equiv 1+n\varpi^{-t-1}\mod\varpi.$$
       However, since we also have $-v(\varpi^{t+1}+n)\geq i_1+1\geq -t-1$, it must be the case that $1+n\varpi^{-t-1}\in\mathcal{O}^\times$. Thus once again, $a\equiv 1\mod\varpi.$ Thus concludes the proof. 
     \end{proof}
 \end{prop}

 \subsubsection{One half-ramified and one unramified principal-series representation} 
 
 In this section, we treat the case $\Pi=\pi_1 \times \pi_2$, where $\pi_2$ is an unramified principal-series representation with Satake parameters $\alpha=\alpha_{\pi_2},\beta=\beta_{\pi_2}$, and $\pi_1$ is a half-ramified unitary principal-series representation. Such a representation $\pi_1$, is up to an unramified twist, of the form $I(\chi\omega,\chi^{-1})$ for an unramified unitary character $\chi:F^\times\rightarrow\mathbf{C}^\times$ and a ramified unitary character $\omega=\omega_{\pi_1}:F^\times\rightarrow \mathbf{C}^\times$ with $\omega(\varpi)=1$.
 In this case, we have $\pi_1^\vee= I(\chi^{-1}\omega^{-1},\chi)$ and $c(\pi_1)=c(\omega)=c(\pi_1^\vee)$. 

\begin{lem}[\cite{Assing_2019} Lemma $3.3.5$]\label{prop Assing half-ramified princ series}
Let $\pi_1=I(\chi\omega,\chi^{-1})$ be as above. Then, for $t\in\mathbf{Z},0\leq k\leq c(\pi_1)$ and $v\in\mathcal{O}^\times$, we have
    \begin{align*}
        W_{\pi_1^\vee}^{\mathrm{new}^*}(g_{t,k,v})=\begin{dcases}
            \chi(\varpi^{t+2c(\pi_1)})q^{-\frac{t+c(\pi_1)}{2}}\epsilon(\tfrac{1}{2},\omega),\ &{\mathrm{if}\ k=0\ \&\ t\geq -c(\pi_1)}\\
            q^{\tfrac{k}{2}}\chi(\varpi^{-t-2k})\epsilon(\tfrac{1}{2},\omega),\ &\substack{\mathrm{if}\ 0<k\leq \left\lfloor \tfrac{c(\pi_1)}{2}\right\rfloor, \ t=-c(\pi_1)-k \\ \&\  v\in b_{\omega^{-1}}^{-1}+\varpi^k\mathcal{O}}\\
            \omega^{-1}(-v^{-1})\psi(-v^{-1}\varpi^{t+k})\chi(\varpi^{-t-2k})q^{\tfrac{c(\pi_1)-k}{2}},\ &\substack{\mathrm{if}\ \left\lceil \tfrac{c(\pi_1)}{2}\right\rceil \leq k< c(\pi_1),\ t=-c(\pi_1)-k\\ \&\ v\in b_{\omega^{-1}}^{-1}+\varpi^{c(\pi_1)-k}\mathcal{O}}\\
            \omega^{-1}(-v^{-1})q^{-\tfrac{t+2k}{2}}\chi(\varpi^{-t-2k})\psi(-v^{-1}\varpi^{t+k}),\ &{\mathrm{if}\ k= c(\pi_1)\ \&\ t\geq -2k}\\
            0,\ &\mathrm{otherwise}
        \end{dcases}
    \end{align*}

\end{lem}
\begin{prop}\label{thm tamely ram p.s + unr p.s}
    Let $\pi_2$ be an unramified principal-series representation. Let $\chi$ (resp. $\omega$) be an unramified (resp. ramified) unitary character of $F^\times$ and let $\pi_1=I(\chi\omega,\chi^{-1})$ with $\omega(\varpi)=1$. Set $\tau:=c(\pi_1)$, $\nu:=q-1$ and $\Pi:=\pi_1\times\pi_2$. Let $A\subseteq\mathbf{C}$ be a $\mathbf{Z}[q^{-1/2},\mu_{\nu q^\tau}]$-algebra containing $\chi$ and the central character and spherical Hecke eigenvalues of $\pi_2$. Let $(\phi,g_1,g_2)\in\mathcal{S}(F^2)\times G(F)$ be a $\Pi$-integral datum. Then, \eqref{eq: Problem} holds.
\end{prop}
\begin{proof}
   The inverse $L$-factor $L(\Pi,s)^{-1}$ is as usual, contained in $A[q^s,q^{-s}]$ by the assumptions on the algebra $A$ and \eqref{eq: gl2 l-factor}. Similarly to the proof of \Cref{thm unram + steinberg}, now using \Cref{lem general coset decomp}, we work with coset representatives 
     $$(z_1g_{t,k,v},z_2n),\ \ z_i\in Z(F),n\in N(F),t\in\mathbf{Z}, 0\leq k\leq c(\pi_1), v\in\mathcal{O}^\times$$
     for $\gamma\in P(F)\backslash G(F)/K_\Pi$
     in order to analyze the rational functions 
     $$\vol_{P(F)}\left(P(F)\cap \gamma K_\Pi\gamma^{-1}\right)^{-1}\cdot \mathcal{I}_{\Pi}^\mathrm{new}(\gamma ;s).$$
     Let $C:=\omega_{\pi_1}(z_1)\omega_{\pi_2}(z_2)$ Similarly to before, one easily sees that 
     \begin{align}
       \nonumber \mathcal{I}_{\Pi}^\mathrm{new}((z_1g_{t,k,v},z_2n))&{=C\cdot \sum_{i\in\mathbf{Z}_{\geq 0}}q^{-\tfrac{i}{2}}\mathrm{Sch}_i(\alpha,\beta)q^{-i(s-1)}\int_{\mathcal{O}^\times}\psi(-n\varpi^i y)\omega(y) W_{\pi_1}^\mathrm{new}(g_{i+t,k,vy^{-1}})\ d^\times y}\\
         &{=C\cdot \sum_{i\in\mathbf{Z}_{\geq 0}}q^{-\tfrac{i}{2}}\mathrm{Sch}_i(\alpha,\beta)q^{-i(s-1)}\int_{\mathcal{O}^\times}\psi(-n\varpi^i y)\omega(y) W_{\pi_1^\vee}^{\mathrm{new}^*}(g_{i+t,k,vy^{-1}})\ d^\times y}\label{eq: 20}
     \end{align}
     where the last equality follows from \Cref{prop new vs conj. new}. We are now in a position to apply Assings' result regarding $W_{\pi_1^\vee}^{\mathrm{new}^*}(g_{i+t,k,vy^{-1}})$. To do this, we split the proof into four different cases arising from \Cref{prop Assing half-ramified princ series}.\\
     \textit{Case} $1$: In the first case, we assume that $k=0$. In this case, $W_{\pi_1^\vee}^{\mathrm{new}^*}(g_{i+t,k,vy^{-1}})=0$ unless $i\geq i_0:=\mathrm{max}\{0, -t-c(\pi_1)\}$, and \eqref{eq: 20} reads 
     \begin{align}\label{eq: 21}
        C\chi(\varpi)^{t+2c(\pi_1)}q^{-\tfrac{t+c(\pi_1)}{2}}\epsilon(\tfrac{1}{2},\omega)\cdot \sum_{i\geq i_0}q^{-\tfrac{i}{2}}\mathrm{Sch}_i(\alpha,\beta)\chi(\varpi)^iq^{-\tfrac{i}{2}}\omega(-n\varpi^i)^{-1}\int_{\varpi^{i+v(n)}\mathcal{O}^\times}\psi(y)\omega(y)\ d^\times y.
     \end{align}
     It is clear that the constant in front of the sum is contained in $A$. Passing to the normalized additive Haar measure, the integral on the far right of \eqref{eq: 21}, is given by 
     \begin{align}\label{eq: 22}\tfrac{q^{i+v(n)+1}}{q-1}\cdot \int_{\varpi^{i+v(n)}\mathcal{O}^\times}\psi(y)\omega(y)\ dy\end{align}
     which by \cite[Lemma $1.1.1$]{schmidt2002some} is non-zero if and only if $i=-v(n)-c(\pi_1)$, in which case \eqref{eq: 22} takes the value $\tfrac{q^{i+v(n)+1}}{q-1}\epsilon(0,\omega^{-1})\in\frac{1}{q-1}A$. In this $k=0$ case, the factor of $\tfrac{1}{q-1}$ is always killed by the index in the same way as in \Cref{thm unram + steinberg} and this concludes the first case.\\
     \\
     \textit{Case} $2$: In the second case, we assume that $0<k\leq \left\lfloor \tfrac{c(\pi_1)}{2}\right\rfloor$. Thus, $W_{\pi_1^\vee}^{\mathrm{new}^*}(g_{i+t,k,vy^{-1}})=0$ unless $i=-t-k-c(\pi_1)$ and $y\in vb_{\omega^{-1}} + \varpi^k\mathcal{O}$. Let $j:=-t-c(\pi_1)-k$ and assume that $j\geq 0$\textcolor{black}{, since} otherwise there's nothing to prove. Hence \eqref{eq: 20} now reads
     \begin{align}\label{eq: 23}
         Cq^{-\tfrac{j}{2}}\mathrm{Sch}_j(\alpha,\beta) q^{-j(s-1)} \chi(\varpi)^{c(\pi_1)-k}q^{-\tfrac{c(\pi_1)-k}{2}}\epsilon(0,\omega)\int_{vb_{\omega^{-1}}+\varpi^k\mathcal{O}}\psi(n\varpi^jy)\omega(y)\ d^\times y.
     \end{align}
     The collection of terms on the left of the epsilon-factor in \eqref{eq: 23} is contained in $A$. The integral on the right of \eqref{eq: 23} in the notation of \Cref{sec Gauss sums} is given by $\omega(vb_{\omega^{-1}})\cdot\mathscr{G}_k(\omega,vb_{\omega^{-1}} n\varpi^j)$ and thus it's an element of $\frac{1}{q-1}A$. In this case, we re-write the volume factor in question as $$\vol_{P(F)}(P(F)\cap (g_{0,k,v},\left[\begin{smallmatrix}
         \varpi^{-t} & \\
         & 1
\end{smallmatrix}\right]n)K_{\Pi}(g_{0,k,v},\left[\begin{smallmatrix}
         \varpi^{-t} & \\
         & 1
\end{smallmatrix}\right]n)^{-1}).$$
If we call this subgroup $\mathcal{U}$, then it is clearly contained in $P(F)\cap g_{0,k,v}K_{c(\pi_1)}g_{0,k,v}^{-1}$. Thus,
     an element $x=\left[\begin{smallmatrix}
         a & b\\
         & 1
     \end{smallmatrix}\right]\in\mathcal{U}$, satisfies $g_{0,k,v}^{-1}xg_{0,k,v}=\left[\begin{smallmatrix}
            * & * \\
            -b & a-vb\varpi^{-k}
        \end{smallmatrix}\right]\in K_{c(\pi_1)}.$ Hence $b\equiv 0\mod\varpi^{c(\pi_1)}$ and $a-vb\varpi^{-k}\equiv 1\mod\varpi^{c(\pi_1)}$\textcolor{black}{, thus $a-vb\varpi^{-k}\equiv 1\mod \varpi$}. But since $0<k<c(\pi_1)$, we have $b\varpi^{-k}\equiv 0\mod\varpi$, which forces $a\equiv 1\mod\varpi$.
     Thus, the corresponding volume factor in question annihilates the factor of $\tfrac{1}{q-1}$. This concludes the second case.\\
     \\
     \textit{Case} $3$: In the third case, we assume that $\left\lceil \tfrac{c(\pi_1)}{2}\right\rceil \leq k < c(\pi_1)$. Then, again $W_{\pi_1^\vee}^{\mathrm{new}^*}(g_{i+t,k,vy^{-1}})=0$ unless $i=j$ and $y\in vb_{\omega^{-1}}+\varpi^{c(\pi_1)-k}\mathcal{O}$. Thus, we can write \eqref{eq: 20} as
     \begin{align}
        Cq^{-\tfrac{j}{2}}\mathrm{Sch}_j(\alpha,\beta)q^{-j(s-1)}\chi(\varpi)^{c(\pi_1)-k}q^{\tfrac{c(\pi_1)-k}{2}}\omega(-v)\int_{1+\varpi^{c(\pi_1)-k}\mathcal{O}}\psi((-n\varpi^j-v^{-1}\varpi^{-c(\pi_1)})vb_{\omega^{-1}} y)\ d^\times y.
     \end{align}
     The expression on the left of the integral is contained in $A$. The integral on the right is once again a partial Gauss sum, this time for the trivial character of $F^\times$. The integral if non-zero, is simply equal to $\tfrac{q}{q-1}$ times a $q^{c(\pi_1)}$-th root of unity. The factor of $\tfrac{1}{q-1}$ is annihilated using the same argument as in the previous case, since once again $k<c(\pi_1)$.\\
     \\
     \textit{Case} $4$: In the fourth and final case, we assume that $k=c(\pi_1)$. In this case, $W_{\pi_1^\vee}^{\mathrm{new}^*}(g_{i+t,k,vy^{-1}})=0$ unless $i\geq -t-2c(\pi_1)$. Let $i_1:=\mathrm{max}\{0,-t-2c(\pi_1)\}$. Then \eqref{eq: 20} reads
     \begin{align*}
         C\chi(\varpi)^{-t-2c(\pi_1)}q^{-\tfrac{t+2c(\pi_1)}{2}}\omega(-v)\sum_{i\geq i_1}\mathrm{Sch}_i(\alpha,\beta)q^{-is}\chi(\varpi)^{-i}\int_{\mathcal{O}^\times}\psi((-n-v^{-1}\varpi^{t+c(\pi_1)})\varpi^iy)\ d^\times y.
     \end{align*}
     The expression in front of the sum is already contained in $A$. We are now in a familiar situation to manipulate the above infinite sum and split \textcolor{black}{it} into a main term giving us an $L$-factor contribution, and an error term given by a Laurent polynomial. Comparing with \eqref{eq: gl2 l-factor}, it is a standard calculation to retrieve the $L$-factor term from $\sum_{i\geq i_1}\mathrm{Sch}_i(\alpha,\beta)q^{-is}\chi(\varpi)^{-i}$. The remaining Laurent polynomial (if it shows up), which has coefficients given by integrals of the form $\int_{\mathcal{O}^\times}\psi((-n-v^{-1}\varpi^{t+c(\pi_1)})\varpi^iy)-1\ d^\times y$, and how it is balanced with the corresponding volume factor, is treated as in the proof of \Cref{thm unram + steinberg}. This error term shows up only if $-v(n+v^{-1}\varpi^{t+c(\pi_1)})\geq i_1+1$. For this fourth and final case, we re-write the volume factor as 
     \begin{align*}
         \vol_{P(F)}(P(\mathcal{O})\cap n^{-1}g_{t,k,v} K_{c(\pi_1)} g_{t,k,v}^{-1}n).
     \end{align*}
     Thus an element $x=\left[\begin{smallmatrix}
         a & b\\
         & 1
     \end{smallmatrix}\right]$ contained in this subgroup satisfies 
     \begin{align}\label{eq: 24 matrix}
         g_{t,k,v}^{-1}nxn^{-1}g_{t,k,v}=\left[\begin{smallmatrix}
             * & * \\
             (an-b-n)\varpi^{-t} & a+(an-b-n)v\varpi^{-t-k}
         \end{smallmatrix}\right]\in K_{c(\pi_1)}.
     \end{align}
     \\
     Firstly, assume that $t\geq -c(\pi_1)$. Since $-v(n+v^{-1}\varpi^{t+c(\pi_1)})\geq i_1+1\geq 1$ implies that $v(n+v^{-1}\varpi^{t+c(\pi_1)})\leq -1$, we must have $v(n)\leq -1$. Thus, looking at the bottom left entry of \eqref{eq: 24 matrix}, we see that $an-b-n\in\mathcal{O}$ and hence $an-n\in\mathcal{O}$. Since $v(n)\leq -1$, it follows that $a\equiv 1\mod\varpi$ and we are done. \\
     Now assume that $-c(\pi_1)>t\geq -2c(\pi_1)$. In this case, if $v(n)< t+c(\pi_1)$, then the same matrix entry once again shows that $a\equiv1\mod\varpi$ and we are done. If $v(n)\geq t+c(\pi_1)$ then we consider the bottom right entry of \eqref{eq: 24 matrix}, which gives 
     \begin{align}
      \nonumber   a(1+vn\varpi^{-t-c(\pi_1)})&\equiv1 + vn\varpi^{-t-c(\pi_1)}+bv\varpi^{-t-c(\pi_1)}&&\mod\varpi^{c(\pi_1)}\\
         &\equiv1+vn\varpi^{-t-c(\pi_1)}&&\mod\varpi^{-t-c(\pi_1)}\label{eq: 29}
     \end{align}
     where the second congruence follows from the inequality $-c(\pi_1)>t\geq-2c(\pi_1)$. But we already have the inequality $v(n+v^{-1}\varpi^{t+c(\pi_1)})\leq -1$ which implies that $v(1+vn\varpi^{-t-c(\pi_1)})\leq -t-c(\pi_1)-1$. Thus \eqref{eq: 29} always implies that $a\equiv1\mod\varpi$ and we are done.\\
     Finally, we assume that $t\leq -2c(\pi_1)$. Once again, if $v(n)<t+c(\pi_1)$, the bottom left entry of \eqref{eq: 24 matrix} shows that $a\equiv 1\mod\varpi$ and we are done. If $v(n)\geq t+c(\pi_1)$, then the bottom right entry of \eqref{eq: 24 matrix} this time reads
    \begin{align}\label{eq: 30}a(1+vn\varpi^{-t-c(\pi_1)})\equiv 1+vn\varpi^{-t-c(\pi_1)}\mod\varpi^{c(\pi_1)}.\end{align}
    However, the inequality $-v(n+v^{-1}\varpi^{t+c(\pi_1)})\geq i_1+1\geq-t-2c(\pi_1)+1$ implies that $v(1+vn\varpi^{-t-c(\pi_1)})\leq c(\pi_1)-1$. Thus \eqref{eq: 30} always implies that $a\equiv1\mod\varpi$. This concludes the proof.
\end{proof}

\subsubsection{One supercuspidal representation and one unramified principal-series representation}\label{supercuspidal x unr p-s} 
Before we dive into the main result of this section, we first report on some standard facts regarding supercuspidal representations of $\GL_2(F)$ via means of the Weil representation. 
Suppose $F$ has odd residue characteristic. Let $E/F$ be a quadratic extension. Write $f=f(E/F)$ for the residue degree and $e=e(E/F)$ for the ramification index. We will always write $\eta_{E/F}$ for the quadratic character of $F^\times$ associated to $E$, via local class field theory. We also write $\mathrm{Nr}=\mathrm{Nr}_{E/F}$ and $\mathrm{Tr}=\mathrm{Tr}_{E/F}$ for the usual norm and trace maps attached to the extension $E/F$. Let $\varpi_E$ be a choice of uniformizer of $E$ for which $\mathrm{Nr}(\varpi_E)=\varpi^f$, under the convention that $\varpi_E=\varpi$ whenever $E/F$ is unramified. This canonically defines an extension of the valuation of $F$ to $E$. Let $\xi:E^\times\rightarrow\mathbf{C}^\times$ be a unitary character which is non-trivial on $\mathrm{ker}(\mathrm{Nr}_{E/F})$. We know that every unitary supercuspidal representation of $\GL_2(F)$ is of the form $\pi(\xi)$, where $\pi(\xi)$ is the Weil representation associated to $(\xi,E)$. For the details of the construction, we refer to \cite[$1.1$]{automorphic_on_GL} or \cite[$4.8$]{bump_1997}.
The central character of $\pi(\xi)$ is given by $\omega_{\pi(\xi)}=\eta_{E/F}\cdot \xi|_{F^\times}$, and the twist of $\pi(\xi)$ by a character $\chi:F^\times\rightarrow\mathbf{C}^\times$ is given by $\chi\otimes\pi(\xi)\simeq \pi(\xi\cdot (\chi\circ\mathrm{Nr})).$ By 
\cite[Theorem $2.3.2$]{schmidt2002some} we have the following formula for the conductor of $\pi(\xi)$
$$c(\pi(\xi))=f(E/F)c(\xi)+e(E/F)-1.$$
Furthermore, by \cite[Theorem $4.7$]{automorphic_on_GL}, there exists a local constant $\gamma_{E/F}\in S^1$, depending only on $E/F$, for which
$$\epsilon(s,\pi(\xi))=\gamma_{E/F}\cdot \epsilon(s,\xi,\psi\circ\mathrm{Tr}_{E/F}).$$
 \textcolor{black}{The abelian $\epsilon$-factor on the right-hand side of the above equality admits an integral representation as in \cite[\S $1.1$(iii)]{schmidt2002some} which we will explicitly recall in due time.}
 We write $d_Ex$ for the additive Haar measure on $E$, normalized to give $\mathcal{O}_E$ volume $q^{\tfrac{1-e(E/F)}{2}}$. Note that by \cite[Lemma $2.3.1$]{schmidt2002some}, the quantity $1-e(E/F)$ is precisely the conductor of the additive character $\psi\circ\mathrm{Tr}_{E/F}$ of $E$. Following \cite{Assing_2018}, we'll consider the functions
 $$K(\xi,A,B):=\int_{\mathcal{O}_E^\times}\xi(x)\psi(\mathrm{Tr}(Ax)+B\mathrm{Nr}(x))\ d_Ex,\ \ A\in E,B\in F.$$ As usual, without loss of generality, we'll assume that our supercuspidal representation has central character which is trivial on $\varpi$. 
\begin{prop}[\cite{Assing_2018} Lemma $3.1$, \cite{Assing_2019} Lemma $3.3.1$]\label{prop new vector supercuspidal}
    Let $\pi=\pi(\xi)$ be a unitary supercuspidal representation with $\omega_{\pi}(\varpi)=1$. For $t\in\mathbf{Z},  0\leq k \leq c(\pi)$ and $v\in\mathcal{O}^\times$, we have
    \begin{align*}
        W_{\pi}^\mathrm{new}(g_{t,k,v})=W_{\pi^\vee}^{\mathrm{new}^*}(g_{t,k,v})=\begin{dcases}
            \epsilon(\tfrac{1}{2},\pi),\ &\mathrm{if}\ k=0\ \&\ t=-c(\pi)\\
            \gamma_{E/F}q^{-\tfrac{t}{2}}K(\xi^{-1},\varpi_E^{t/f},v\varpi^{-k}),\ &\mathrm{if}\ k=\tfrac{c(\pi)}{2}\ \& \ -c(\pi)\leq t<0\\
            \gamma_{E/F}q^{-\tfrac{t}{2}}K(\xi^{-1},\varpi_E^{t/f},v\varpi^{-k}),\ &\mathrm{if}\ k\neq 0,\tfrac{c(\pi)}{2},c(\pi)\ \&\ t={-\mathrm{max}\{c(\pi),2k\}}\\
            \omega_{\pi}^{-1}(-v^{-1})\psi(-v^{-1}\varpi^{-c(\pi)}),\ &\mathrm{if}\ k=c(\pi)\ \&\ t=-2k\\
            0,\ &\mathrm{otherwise.}
        \end{dcases}
    \end{align*}
    Furthermore, if $E/F$ is unramified, $ W_{\pi}^\mathrm{new}(g_{t,k,v})$ vanishes for odd values of $t$.
\end{prop}
\begin{prop}\label{thm sup. x unr. p-s}
    Let $\pi_1=\pi(\xi)$ be a unitary supercuspidal representation with $\omega_{\pi_1}(\varpi)=1$, associated to a quadratic extension $E/F$ and a (regular) unitary character $\xi:E^\times\rightarrow\mathbf{C}^\times$. Let $\pi_2$ be an unramified principal-series representation, and set $\Pi:=\pi_1\times\pi_2$. Let $\tau:=c(\pi_1)$, $\nu:=(q^{f(E/F)}-1)(1+\delta_{c(\pi_1)\equiv 1(2)})$ and $A\subseteq\mathbf{C}$ be a $\mathbf{Z}[q^{-1/2},\mu_{\nu q^\tau}]$-algebra containing $\omega_{\pi_2}$ and the spherical Hecke eigenvalues of $\pi_2$. Let $(\phi,g_1,g_2)\in\mathcal{S}(F^2)\times G(F)$ be a $\Pi$-integral datum. Then, \eqref{eq: Problem} holds.
    \begin{proof}
        In this case it is clear that $L(\Pi,s)=1$. As in the proof of \Cref{thm tamely ram p.s + unr p.s}, we will work with coset representatives $\gamma=(z_1g_{t,k,v},z_2n)$ with $z_i\in Z(F), n\in N(F), t\in \mathbf{Z}, 0\leq k\leq c(\pi_1)$ and $v\in\mathcal{O}^\times$ to analyze the rational functions
        \begin{align}\label{eq: 31} \vol_{P(F)}(P(F)\cap\gamma K_\Pi\gamma^{-1})^{-1}\cdot\mathcal{I}_{\Pi}^\mathrm{new}(\gamma;s).\end{align}
        We write $C:=\omega_{\pi_1}(z_1)\omega_{\pi_2}(z_2)$. Then as in the proof of \Cref{thm tamely ram p.s + unr p.s}, we have
        \begin{align}\label{eq: 32}
            \mathcal{I}_{\Pi}^\mathrm{new}(\gamma ;s)=C \sum_{i\in\mathbf{Z}_{\geq 0}}q^{-\tfrac{i}{2}}\mathrm{Sch}_i(\alpha,\beta)q^{-i(s-1)}\int_{\mathcal{O}^\times}\omega_{\pi_1}(y)\psi(-n\varpi^iy)W_{\pi_1}^\mathrm{new}(g_{i+t,k,vy^{-1}})\ d^\times y.
        \end{align}
        We distinguish between various cases depending on the values of $0\leq k\leq c(\pi_1)$.\\
        \textit{Case} $1$: If $k=0$, then \Cref{prop new vector supercuspidal} gives that
        \begin{align*}
             \mathcal{I}_{\Pi}^\mathrm{new}(\gamma;s)&=D(q^s) \epsilon(\tfrac{1}{2},\pi_1)\int_{\mathcal{O}^\times}\omega_{\pi_1}(y)\psi(-n\varpi^{-t-c(\pi_1)}y)\ d^\times y\\
             &=D(q^s) \epsilon(\tfrac{1}{2},\pi_1)\mathscr{G}(\omega_{\pi_1},-n\varpi^{-t-c(\pi_1)})\\
             &=D(q^s)\gamma_{E/F}\ \epsilon(\tfrac{1}{2},\xi,\psi\circ\mathrm{Tr})\mathscr{G}(\omega_{\pi_1},-n\varpi^{-t-c(\pi_1)})
        \end{align*}
        where $D(q^s):=Cq^{-\tfrac{-t-c(\pi_1)}{2}}\mathrm{Sch}_{-t-c(\pi_1)}(\alpha,\beta)q^{(t+c(\pi_1))(s-1)}\in A[q^s,q^{-s}]$. By \cite{automorphic_on_GL}, the local constant, $\gamma_{E/F}$ is given explicitly by $$\gamma_{E/F}=\eta_{E/F}\left(\varpi^{c(\eta_{E/F})}\right)\frac{\int_{\mathcal{O}^\times}\eta_{E/F}^{-1}(x)\psi(x\varpi^{-c(\eta_{E/F})})\ dx}{|\int_{\mathcal{O}^\times}\eta_{E/F}^{-1}(x)\psi(x\varpi^{-c(\eta_{E/F})})\ dx|}.$$
        If $c(\eta_{E/F})=0$, then clearly $\gamma_{E/F}=1$. If
        there's ramification, then
         the numerator is still a $\mathbf{Z}[q^{-1}]$-linear combination of $q^{c(\eta_{E/F})}$-roots of unity, and the denominator is equal to $q^{-c(\eta_{E/F})/2}$ by \cite{biswas2017twisting} (or \cite{gel1969representation}). Note that $c(\eta_{E/F})\leq c(\pi_1)$ by central character considerations. 
        The $\epsilon$-factor appearing in the third equality above, is given by 
        $$\epsilon(\tfrac{1}{2},\xi,\psi\circ\mathrm{Tr})=\int_{\varpi_E^{1-e-c(\xi)}\mathcal{O}_E^\times}|x|_E^{-1/2}\xi^{-1}(x)\psi(\mathrm{Tr}(x))\ d_Ex.$$
        This reduces to a $\xi(\varpi_E^{c(\xi)+e-1})$-multiple of a $\mathbf{Z}[q^{-f/2}]$-linear combination of $\xi|_{\mathcal{O}_E^\times}$-translates of $q^{c(\xi)}$-roots of unity. \textcolor{black}{The values of $\xi|_{\mathcal{O}_E^\times}$ are roots of unity of order $(q^f-1)q^{c(\pi_1)-2}$ which are contained in $A$ by construction. If $E/F$ is unramified, $\xi(\varpi_E^{c(\xi)+e-1})$ is simply equal to $\pm 1$. If $E/F$ is ramified,  $\xi(\varpi_E^{c(\xi)+e-1})=\xi(\varpi_E^{c(\pi_1)})$ is a root of unity of order $(q-1)^{c(\pi_1)-2}$ if $c(\pi_1)$ is even, and of order $2(q-1)^{c(\pi_1)-2}$ if $c(\pi_1)$ is odd. This follows from $\omega_{\pi(\xi)}=\eta_{E/F}\cdot \xi|_{F^\times}$, and all such roots of unity are contained in $A$ from the definition of $\nu$}. Finally, the Gauss sum $\mathscr{G}(\omega_{\pi_1},-n\varpi^{-t-c(\pi_1)})$, if non-zero, is always contained in $\tfrac{1}{q-1}A$ since $c(\omega_{\pi_1})\leq c(\pi_1)$. To see that in this case the inverse volume factor appearing in \eqref{eq: 31} is an integral multiple of $q-1$, one argues as in the $k=0$ case of \Cref{thm unram + steinberg}.\\
        \textit{Case} $2$: We now assume that $0<k<c(\pi_1)$. In this case, \Cref{prop new vector supercuspidal} implies that the summands appearing in \eqref{eq: 32} are non-zero only for finitely many values of $i\in\mathbf{Z}_{\geq 0}$, and for such a value of $i$ the corresponding integral appearing in \eqref{eq: 32} is a $\gamma_{E/F}$-multiple of a $\mathbf{Z}[q^{-f/2}]$-linear combination of 
        $$\int_{\mathcal{O}^\times}\omega_{\pi_1}(y)\psi(-n\varpi^iy)\int_{\mathcal{O}_E^\times}\xi^{-1}(x)\psi\left(\mathrm{Tr}(\varpi_E^{(t+i)/f}x)+v\varpi^{-k}y^{-1}\mathrm{Nr}(x)\right)\ d_Ex\ d^\times y.$$ 
        Using the fact that $y^{-1}\mathrm{Nr}(x)=y\mathrm{Nr}(y^{-1}x)$ for $y\in\mathcal{O}^\times$, a change of variables $y^{-1}x\rightsquigarrow x$, and the fact that $\omega_{\pi_1}=\eta_{E/F}\xi|_{F^\times}$, the above is equal to
        $$\int_{\mathcal{O}^\times}\eta_{E/F}(y)\psi(-n\varpi^iy)\int_{\mathcal{O}_E^\times}\xi^{-1}(x)\psi\left(y\mathrm{Tr}(\varpi_E^{(t+i)/f}x)+yv\varpi^{-k}\mathrm{Nr}(x)\right)\ d_Ex\ d^\times y.$$
        We can now pick $M\geq 1$ large enough such that the innermost integrand is invariant on the right by $1+\varpi_E^M\mathcal{O}_E$. Then the above expression is a $\mathbf{Z}[q^{-f/2}]$-multiple of $$\sum_{\delta\in \mathcal{O}_E^\times/1+\varpi_E^M\mathcal{O}_E}\xi^{-1}(\delta)\mathscr{G}\left(\eta_{E/F},-n\varpi^i+\mathrm{Tr}(\varpi_E^{(t+i)/f}\delta)+v\varpi^{-k}\mathrm{Nr}(\delta)\right).$$
        This is an instance of \cite[Lemma $3.1.1$]{cesnavicius2022manin}.
        Each Gauss sum appearing above, if non-zero, reduces to a $\mathbf{Z}[q^{-1}]$-multiple of $1, \tfrac{1}{q-1}$, or $\tfrac{1}{q-1}\epsilon(0,\eta_{E/F})$. The latter $\epsilon$-factor reduces to a $\mathbf{Z}[q^{-1}]$ linear combination of $q^{c(\eta_{E/F})}$-roots of unity and hence \textcolor{black}{is} contained in $A$. It remains to show that the possible $\tfrac{1}{q-1}$ factor is always countered by the volume factor appearing in \eqref{eq: 32}. Indeed, up to an integral power of $q$, the volume factor in question, is given by $\vol_{P(F)}(\mathcal{U})^{-1}$, where $$\mathcal{U}:=P(F)\cap (g_{0,k,v},\left[\begin{smallmatrix}
            \varpi^{-t} & \\
            & 1
        \end{smallmatrix}\right]n)K_\Pi  (g_{0,k,v},\left[\begin{smallmatrix}
            \varpi^{-t} & \\
            & 1
        \end{smallmatrix}\right]n)^{-1}\subseteq g_{0,k,v} K_{c(\pi_1)} g_{0,k,v}^{-1}.$$
        But if $x=\left[\begin{smallmatrix}
            a & b \\
            & 1
        \end{smallmatrix}\right]\in\mathcal{U}$, then 
        $g_{0,k,v}^{-1}x g_{0,k,v}=\left[\begin{smallmatrix}
            * & * \\
            -b & a-vb\varpi^{-k}
        \end{smallmatrix}\right]\in K_{c(\pi_1)}.$ Hence $b\equiv 0\mod\varpi^{c(\pi_1)}$ and $a-vb\varpi^{-k}\equiv 1\mod\varpi^{c(\pi_1)}$\textcolor{black}{, thus $a-vb\varpi^{-k}\equiv 1\mod \varpi$}. But since $0<k<c(\pi_1)$, we have $b\varpi^{-k}\equiv 0\mod\varpi$, which forces $a\equiv 1\mod\varpi$, and we are done.\\
        \textit{Case} $3$: In the third and final case, we assume that $k=c(\pi_1)$. Then the summand in \eqref{eq: 32} is non-zero only for $i=-t-2c(\pi_1)\geq 0$, and using \Cref{prop new vector supercuspidal}, the corresponding integral appearing in \eqref{eq: 32}, reduces to
        $$\omega_{\pi_1}(v)\int_{\mathcal{O}^\times}\psi\left((-n-v^{-1}\varpi^{t+c(\pi_1)})\varpi^{-t-2c(\pi_1)}y\right)\ d^\times y.$$ This puts us in exactly the same situation as the very last part of the proof of \Cref{thm tamely ram p.s + unr p.s}, and thus the inverse volume factor in \eqref{eq: 32} can be shown to always counter the $\tfrac{1}{q-1}$ factor of the above integral, whenever it appears.
    \end{proof}
\end{prop}

\subsubsection{One supercuspidal representation and one unramified twist of the Steinberg representation}

In this section, we resolve the case $\pi_1\times \pi_2$, where one of the representations is supercuspidal and the other one is an unramified twist of the Steinberg representation. Without loss of generality, we will assume that $\pi_1$ is a supercuspidal representation. The notation will be the same as that of \Cref{supercuspidal x unr p-s} and \Cref{unr StxSt}.
\begin{prop}\label{thm sup x unr St}
    Let $\pi_1=\pi(\xi)$ be a unitary supercuspidal representation as in \emph{\Cref{thm sup. x unr. p-s}}. Let $\pi_2=\mathrm{St}_\chi$ be a twist of the Steinberg representation by an unramified character $\chi:F^\times\rightarrow\mathbf{C}^\times$, and set $\Pi:=\pi_1\times \pi_2$. Let $\tau:=c(\pi_1)$, $\nu:=(q^{f(E/F)}-1)(1+\delta_{c(\pi_1)\equiv 1(2)})$ and $A\subseteq\mathbf{C}$ be a $\mathbf{Z}[q^{-1/2},\mu_{\nu q^\tau}]$-algebra containing $\chi$. Let $(\phi,g_1,g_2)\in\mathcal{S}(F^2)\times G(F)$ be a $\Pi$-integral datum. Then, \eqref{eq: Problem} holds.
    \begin{proof}
        In this case, once again consulting \Cref{lem general coset decomp}, we can work with coset representatives $\gamma=(z_1g_{t,k_1,v},z_2ng_{0,k_2,1})$ for $P(F)\backslash G(F)/K_\Pi$, where $z_i\in Z(F), n\in N(F), t\in\mathbf{Z}, 0\leq k_1\leq c(\pi_1), 0\leq k_2\leq 1$ and $v\in\mathcal{O}^\times$, in order to analyze the functions 
        \begin{align}\label{eq: 33}
            \vol_{P(F)}\left(P(F)\cap \gamma K_\Pi \gamma^{-1}\right)^{-1}\cdot \mathcal{I}_{\Pi}^\mathrm{new}(\gamma;s).
        \end{align}
        In this case, in order to avoid potential confusion, we write $n=\left[\begin{smallmatrix}
            1 & n_0\\
            & 1
        \end{smallmatrix}\right]\in N(F)$.
        Let $C:=\omega_\Pi(z_1,z_2)$. Then
        $$\mathcal{I}_{\Pi}^\mathrm{new}(\gamma;s)=C\sum_{i\geq -k_2-1}\chi(\varpi^i)q^{i(1-s)}\int_{\mathcal{O}^\times}\omega_{\pi_1}(y)\psi(-n_0y\varpi^i)W_{\pi_1}(g_{i+t,k_1,vy^{-1}})W_{\mathrm{St}}^\mathrm{new}(g_{i,k_2,y^{-1}})\ d^\times y. $$
        \textit{Case} $1$: We firstly assume that $k_2=0$. Then, by \Cref{lem Assing for steinberg}, we have 
        $$\mathcal{I}_{\Pi}^\mathrm{new}(\gamma;s)=-Cq^{-1}\sum_{i\geq -1}\chi(\varpi^i)q^{-is}\int_{\mathcal{O}^\times}\omega_{\pi_1}(y)\psi(-n_0y\varpi^i)W_{\pi_1}(g_{i+t,k_1,vy^{-1}})\ d^\times y$$
        and this is actually a finite sum. The shape of the integrals appearing in the summands above, puts us exactly in the same situation as the proof of \Cref{thm sup. x unr. p-s} when it comes to showing that they're always contained in $\tfrac{1}{q-1}A$. In this $k_2=0$ case, the unwanted $\tfrac{1}{q-1}$ is easily killed by the volume factor in \eqref{eq: 33}, since 
        $$P(F)\cap \gamma K_\Pi\gamma^{-1}\subseteq P(F)\cap g_{0,0,1}K_1 g_{0,0,}^{-1}$$
        which is a situation we have dealt with before.\\
        \textit{Case} $2$: We now assume $k_2=1$. Then 
        $$\mathcal{I}_{\Pi}^\mathrm{new}(\gamma;s)=Cq^{-2}\sum_{i\geq -2}\chi(\varpi^i)q^{-is}\int_{\mathcal{O}^\times}\omega_{\pi_1}(y)\psi(-(n_0-\varpi)\varpi^iy)W_{\pi_1}(g_{i+t,k_1,vy^{-1}})\ d^\times y. $$
        If $0\leq k_1<c(\pi_1)$, then by replacing $n_0$ with $n_0-\varpi$, we have once again shown in the proof of \Cref{thm sup. x unr. p-s} that the above expression is always in $\tfrac{1}{q-1}A[q^s,q^{-s}]$ and the $\tfrac{1}{q-1}$ factor is always killed by the volume factor since $k_1<c(\pi_1)$. \\
        We are thus left to deal with $k_1=c(\pi_1)$. In this case, we have 
        \begin{align}\label{eq: 34}\mathcal{I}_{\Pi}^\mathrm{new}(\gamma ;s)=Cq^{-2+js}\chi(\varpi^j)\int_{\mathcal{O}^\times}\psi\left((-vn_0\varpi^{-t-c(\pi_1)}+v\varpi^{-t-c(\pi_1)+1}-1)\varpi^{-c(\pi_1)}y\right)\ d^\times y\end{align}
        where $j:=t+2c(\pi_1)\leq 2$ (if $t+2c(\pi_1) > 2$, then everything is zero and there is nothing to prove). Notice how $j\leq 2$ implies $-t-c(\pi_1)\geq c(\pi_1)-2\geq 0$ since $\pi_1$ is a supercuspidal representation.\\
            The volume factor $\vol_{P(F)}(P(F)\cap (g_{t,c(\pi_1),v},n g_{0,1,1}) K_\Pi  (g_{t,c(\pi_1),v},n g_{0,1,1})^{-1}) $ is the same as the volume factor $ \vol_{P(F)}(P(F)\cap (n^{-1}g_{t,c(\pi_1),v},g_{0,1,1}) K_\Pi  (n^{-1}g_{t,c(\pi_1),v},g_{0,1,1})^{-1})$. We recall that if $x=\left[\begin{smallmatrix}
                a & b \\
                & 1
            \end{smallmatrix}\right] $ is an element of $P(F)\cap (n^{-1}g_{t,c(\pi_1),v},g_{0,1,1}) K_\Pi  (n^{-1}g_{t,c(\pi_1),v},g_{0,1,1})^{-1}$, then we must have 
            \begin{align}\label{eq: 35}
    g_{0,1,1}^{-1}xg_{0,1,1}&=\left[\begin{smallmatrix}
                    * & * \\
                    -b & a-b\varpi^{-1}
                \end{smallmatrix}\right]\in K_1\\
            g_{t,c(\pi_1),v}^{-1}nxn^{-1}g_{t,c(\pi_1),v}&=\left[\begin{smallmatrix}
                    * & *\\
                    (an_0-b-n_0)\varpi^{-t} & a+(an_0-b-n_0)v\varpi^{-t-c(\pi_1)}
                \end{smallmatrix}\right]\in K_{c(\pi_1)}.\label{eq: 36}
            \end{align}
            The bottom-right entry of \eqref{eq: 35} gives $$b\varpi^{-t-c(\pi_1)}\in(a-1)\varpi^{-t-c(\pi_1)+1}+\varpi^{-t-c(\pi_1)+2}\mathcal{O}\subseteq(a-1)\varpi^{-t-c(\pi_1)+1}+\varpi^{c(\pi_1)}\mathcal{O}$$
            where the last inclusion follows from $j\leq 2$. Thus, the bottom-right entry of \eqref{eq: 36} implies
            \begin{align}\label{eq: 37}a(1+n_0v\varpi^{-t-c(\pi_1)}-v\varpi^{-t-c(\pi_1)+1})\in 1+n_0v\varpi^{-t-c(\pi_1)}-v\varpi^{-t-c(\pi_1)+1}+ \varpi^{c(\pi_1)}\mathcal{O}.\end{align}
            However, the expression in \eqref{eq: 34} has an unwanted $\tfrac{1}{q-1}$ contribution only if $v(1+vn_0\varpi^{-t-c(\pi_1)}-v\varpi^{-t-c(\pi_1)+1})$ is equal to $c(\pi_1)-1$. In this case, \eqref{eq: 37} implies that $a\equiv 1\mod \varpi$, which forces the volume factor of \eqref{eq: 33} to be an integral multiple of $q-1$, once again killing the unwanted $\tfrac{1}{q-1}$ factor. This completes all the cases and the proof.
        \end{proof}
\end{prop}
\subsubsection{One supercuspidal representation and one half-ramified principal-series representation}
In this section we treat the case $\pi_1\times \pi_2$ where one of the two representations is supercuspidal and the other is a half-ramified unitary principal-series representation. Once again without loss of generality, we will assume that $\pi_1$ is a supercuspidal representation.
\begin{prop}\label{thm sup + half-ram p-s}
    Let $\pi_1=\pi(\xi)$ be a unitary supercuspidal representation as in \emph{\Cref{thm sup. x unr. p-s}}. Let $\pi_2=I(\chi \omega,\chi^{-1})$ for unramified (resp. ramified) unitary character $\chi$ (resp, $\omega$) of $F^\times$, with $\omega(\varpi)=1$. Set $\Pi=\pi_1\times \pi_2$, $\tau:=\max\{c(\pi_1),c(\omega)\}$, $\nu:=(q^{f(E/F)}-1)(1+\delta_{c(\pi_1)\equiv 1(2)})$ and let $A\subseteq \mathbf{C}$ be a $\mathbf{Z}[q^{-1/2},\mu_{\nu q^\tau}]$-algebra containing $\chi$. Let $(\phi,g_2,g_2)\in \mathcal{S}(F^2)\times G(F)$ be a $\Pi$-integral datum.Then, \eqref{eq: Problem} holds.
    \begin{proof}
        We use coset representatives $\gamma=(z_1g_{t,k_1,v_1},z_2ng_{0,k_2,v_2})$ for $P(F)\backslash G(F)/K_\Pi$, where $z_i\in Z(F),n=\left[\begin{smallmatrix}
            1 & n_0\\
            & 1
        \end{smallmatrix}\right]\in N(F),t\in\mathbf{Z}, 0\leq k_i\leq c(\pi_i)$ and $v_i\in\mathcal{O}^\times$ to analyze, as usual, the functions
        \begin{align}\label{eq: 38}
            \vol_{P(F)}\left(P(F)\cap \gamma K_\Pi \gamma^{-1}\right)^{-1}\cdot \mathcal{I}_{\Pi}^\mathrm{new}(\gamma;s).
        \end{align}
        We start with, the actually finite sum, (writing $C:=\omega_{\Pi}(z_1,z_2)$)
        \begin{align}\label{eq: 39}
            &{\mathcal{I}_{\Pi}^\mathrm{new}(\gamma;s)=C\sum_{i\geq -k_2-c(\pi_2)}q^{i(1-s)}\int_{\mathcal{O}^\times}(\omega_{\pi_1}\omega)(y)\psi(-n_0y\varpi^i)W_{\pi_1}^\mathrm{new}(g_{i+t,k_1,v_1y^{-1}})W_{\pi_2}^\mathrm{new}(g_{i,k_2,v_2y^{-1}})\ d^\times y}.
        \end{align}
        By \Cref{prop Assing half-ramified princ series}, and similar arguments as before, it is not hard to see that every integral in the sum above, is an $A[q^s,q^{-s}]$-multiple of:
        $$\begin{dcases}
\int_{v_2b_{\omega^{-1}}+\varpi^{k_2}\mathcal{O}}(\omega_{\pi_1}\omega)(y)\psi(-n_0y\varpi^i)W_{\pi_1}^\mathrm{new}(g_{i+t,k_1,v_1y^{-1}})\ d^\times y,\ &\mathrm{if}\ 0\leq k_2\leq\lfloor \tfrac{c(\pi_2)}{2}\rfloor\\
    \int_{v_2b_{\omega^{-1}}+\varpi^{c(\pi_2)-k_2}\mathcal{O}}\omega_{\pi_1}(y)\psi(-(n_0-v_2^{-1}\varpi^{k_2})y\varpi^i)W_{\pi_1}^\mathrm{new}(g_{i+t,k_1,v_1y^{-1}})\ d^\times y,\ &\mathrm{if}\ \lceil\tfrac{c(\pi_2)}{2}\rceil\leq k_2\leq c(\pi_2)
        \end{dcases}$$
        where $v_2b_{\omega^{-1}}+\varpi^0\mathcal{O}$ should be interpreted as $\mathcal{O}^\times$.
        If $k_1$ is not equal to $c(\pi_1)$, or $k_2$ is not equal to $c(\pi_2)$, then the proof of \Cref{thm sup. x unr. p-s} gives the required result, in exactly the same way, using partial Gauss sums instead, together with \Cref{sec Gauss sums}. If $k_1=c(\pi_1)$ and $k_2=c(\pi_2)$, then the only integral involved in  \eqref{eq: 39} reads
        \begin{align}\label{eq: 40}\int_{\mathcal{O}^\times}\psi\left((-v_1n_0\varpi^{-t-c(\pi_1)}+v_1v_2^{-1}\varpi^{-t-c(\pi_1)+c(\pi_2)}-1)\varpi^{-c(\pi_1)}y\right)\ d^\times y\end{align}
        and this is only if $-t-2c(\pi_1)\geq -2c(\pi_2)$, i.e $t+2c(\pi_1)\leq 2c(\pi_2)$ (otherwise everything is zero and there's nothing to prove). If $x=\left[\begin{smallmatrix}
            a & b\\
            & 1
        \end{smallmatrix}\right]\in P(F)\cap \gamma K_\Pi\gamma^{-1}$, we have 
        \begin{align}\label{eq: 41}
    g_{0,k_2,v_2}^{-1}xg_{0,k_2,v_2}&=\left[\begin{smallmatrix}
                    * & * \\
                    -b & a-bv_2\varpi^{-c(\pi_2)}
                \end{smallmatrix}\right]\in K_{c(\pi_2)}\\
            g_{t,c(\pi_1),v_1}^{-1}nxn^{-1}g_{t,c(\pi_1),v_1}&=\left[\begin{smallmatrix}
                    * & *\\
                    (an_0-b-n_0)\varpi^{-t} & a+(an_0-b-n_0)v_1\varpi^{-t-c(\pi_1)}
                \end{smallmatrix}\right]\in K_{c(\pi_1)}.\label{eq: 42}
            \end{align}
           The bottom right entry of \eqref{eq: 41} gives that
           $$b\varpi^{-t-c(\pi_1)}\in v_2^{-1}(a-1)\varpi^{-t-c(\pi_1)+c(\pi_2)}+\varpi^{-t-c(\pi_1)+2c(\pi_2)}\mathcal{O}\subseteq v_2^{-1}(a-1)\varpi^{-t-c(\pi_1)+c(\pi_2)}+\varpi^{c(\pi_1)}\mathcal{O}$$
            where the last inclusion follows from $-t-2c(\pi_1)\geq -2c(\pi_2)$. Hence, the bottom-right entry of \eqref{eq: 42} implies that
            \begin{align}\label{eq: 43}a(1+v_1n_0\varpi^{-t-c(\pi_1)}-v_2^{-1}v_1\varpi^{-t-c(\pi_1)+c(\pi_2)})\in 1+v_1n_0\varpi^{-t-c(\pi_1)}-v_2^{-1}v_1\varpi^{-t-c(\pi_1)+c(\pi_2)}+\varpi^{c(\pi_1)}\mathcal{O.}\end{align}
            However, the expression in \eqref{eq: 40} has an unwanted $\tfrac{1}{q-1}$ factor only if $v(1+v_1n_0\varpi^{-t-c(\pi_1)}-v_2^{-1}v_1\varpi^{-t-c(\pi_1)+c(\pi_2)})$ is equal to $c(\pi_1)-1$, in which case \eqref{eq: 43} gives $a\equiv 1\mod\varpi$ and we are done.
    \end{proof}
\end{prop}
\subsubsection{Two supercuspidal representations}
We now \textcolor{black}{address the case of two supercuspidal representations $\pi_1,\pi_2$}. The proofs from now on will get more brief since the arguments are very similar and start repeating.

\begin{prop}
    Let $\pi_i=\pi(\xi_i)$ be two unitary supercuspidal representations associated to quadratic extensions $E_i/F$ and regular unitary characters $\xi_i:E_i^\times\rightarrow\mathbf{C}^\times$, with $\omega_{\pi_i}(\varpi)=1$. Set $\Pi:=\pi_1\times \pi_2$, $\nu:=(q^{\max_{i\in\{1,2\}}f(E_i/F)}-1)(1+(\delta_{\max_{i}f(E_i/F)=1})(\max_{i}\delta_{c(\pi_i)\equiv1(2)}))$ and $\tau:=\max_{i\in\{1,2\}}\{c(\pi_i)\}$. Let $A\subseteq\mathbf{C}$ be a $\mathbf{Z}[q^{-1/2},\mu_{\nu q^\tau}]$-algebra. Let $(\phi,g_1,g_2)\in \mathcal{S}(F^2)\times G(F)$ be a $\Pi$-integral datum. Then, \eqref{eq: Problem} holds.
    \begin{proof}
        The inverse $L$-factor $L(\Pi,s)^{-1}$ is always contained in $A[q^s,q^{-s}]$ even when it is not identically $1$. We use coset representatives for $P(F)\backslash G(F)/K_\Pi$ of the same form as in \Cref{thm sup + half-ram p-s}, i.e.\ $\gamma=(z_1g_{t,k_1,v_1},ng_{0,k_2,v_2})$. The \textcolor{black}{(finite)} sum corresponding to the expression for \eqref{eq: 39} becomes
        \begin{align}\label{eq: 44}
            &{\mathcal{I}_{\Pi}^\mathrm{new}(\gamma;s)=C\sum_{i\geq -k_2-c(\pi_2)}q^{i(1-s)}\int_{\mathcal{O}^\times}(\omega_{\pi_1}\omega_{\pi_2})(y)\psi(-n_0y\varpi^i)W_{\pi_1}^\mathrm{new}(g_{i+t,k_1,v_1y^{-1}})W_{\pi_2}^\mathrm{new}(g_{i,k_2,v_2y^{-1}})\ d^\times y}. 
        \end{align}
        By \Cref{prop new vector supercuspidal}, and treatment of $\epsilon$-factors and local constants $\gamma_{E_i/F}$ seen in previous proofs, for each $i\geq -k_2-c(\pi_2)$ the corresponding integral appearing in the (finite) sum above is an $A[q^s,q^{-s}]$-multiple (possibly zero) of 
        $$\begin{dcases}
            \int_{\mathcal{O}^\times}(\omega_{\pi_1}\omega_{\pi_2})(y)\psi(-n_0y\varpi^{-c(\pi_2)})W_{\pi_1}^\mathrm{new}(g_{-c(\pi_2)+t,k_1,v_1y^{-1}})\ d^\times y,\ &\mathrm{if}\ k_2=0\\
            \int_{\mathcal{O}^\times}(\omega_{\pi_1}\omega_{\pi_2})(y)\psi(-n_0y\varpi^i)W_{\pi_1}^\mathrm{new}(g_{i+t,k_1,v_1y^{-1}})K(\xi_2^{-1},\varpi_{E_2}^{i/f_2},v_2y^{-1}\varpi^{-k_2})\  d^\times y,\ &\mathrm{if}\ 0<k_2<c(\pi_2) \\
            \int_{\mathcal{O}^\times}\omega_{\pi_1}(y)\psi((-n_0\varpi^{-c(\pi_2)}+v_2^{-1})\varpi^{-c(\pi_2)}y)W_{\pi_1}^\mathrm{new}(g_{-2k_2+t,k_1,v_1y^{-1}})\ d^\times y,\ &\mathrm{if}\ k_2=c(\pi_2).
        \end{dcases}$$
        If $k_1\neq c(\pi_1)$ or $k_2\neq c(\pi_2)$, then all three possibilities appearing above are handled in the same way as in \Cref{thm sup. x unr. p-s}, showing they're contained in $\tfrac{1}{q-1}A$. Note that the case $0<k_2<c(\pi_2)$ which might lead to two $K$-functions inside the integral is not at all more complicated than \textit{Case} $2$ of the proof of \Cref{thm sup. x unr. p-s}; one just applies the same splitting argument twice. Of course, in this case where $k_1\neq c(\pi_1)$ or $k_2\neq c(\pi_2)$ we've seen that $\vol_{P(F)}(P(F)\cap\gamma K_\Pi \gamma^{-1})^{-1}$ is always an integral multiple of $q-1$, giving the required result. If both $k_1=c(\pi_1)$ and $k_2=c(\pi_2)$, then \eqref{eq: 44} is non-zero only if $t+2c(\pi_1)=2c(\pi_2)$, in which case, it is an $A$-multiple of the same exact expression as \eqref{eq: 40}. One then proceeds in an identical fashion to study the volume factor.
    \end{proof}
\end{prop}
\subsubsection{One supercuspidal representation and one fully-ramified principal-series representation}
In this section we treat the case $\pi_1\times \pi_2$ where one representation is supercuspidal and the other one is a fully-ramified principal-series representation; i.e. up to an unramified twist, a principal-series representation of the form $I(\chi_1\chi,\chi_2\chi^{-1})$ where $\chi$ is unramified and $\chi_i$ are both ramified characters of $F^\times$ with $\chi_i(\varpi)=1$. In this case the conductor is given by $c(\chi_1)+c(\chi_2)$, by \cite{schmidt2002some}. The behavior of the new vector is once again determined by Assing in \cite[Lemma $3.6$]{Assing_2018}. The result of \textit{loc.cit.} in this case is very long and distinguishes between several cases, so we will not state it here. We do remark however that there's a minor sign error in \textit{loc.cit.}, which can be found corrected in \textcolor{black}{Assing's} thesis \cite[Lemma $3.3.9$]{Assing_2019}. The statement of the results in this case, involve  the generalized functions $${K\left(\chi_1\otimes\chi_2,(\varpi^{-a_1},\varpi^{-a_2}),v\varpi^{-a}\right):=\int_{\mathcal{O}^\times}\int_{\mathcal{O}^\times}}\chi_1(x_1)\chi_2(x_2)\psi(x_1\varpi^{-a_1}+x_2\varpi^{-a_2}+vx_1x_2\varpi^{-a})\ d^\times x_1\ d^\times x_2$$
where $\chi_i$ are ramified characters of $F^\times$, $a,a_i\in\mathbf{Z}$ and $v\in \mathcal{O}^\times$. These functions should be thought of as an analogue of the $K$-functions in the supercuspidal case (i.e. $E/F$ quadratic extension) , but for the degenerate situation of the quadratic space $E=F\oplus F$.
\begin{prop}\label{thm sup + fully ram p-s}
    Let $\pi_1=\pi(\xi)$ be a unitary supercuspidal representation as in \emph{\Cref{thm sup. x unr. p-s}} and let $\pi_2=I(\chi_1\chi,\chi_2\chi^{-1})$ be a fully-ramified unitary principal-series representation with $\chi$ unramified, $\chi_1,\chi_2$ 
 distinct ramified characters and $\chi_i(\varpi)=1$. Let $\tau:=\max\{c(\pi_1),c(\chi_1),c(\chi_2)\}$, $\nu:=(q^{f(E/F)}-1)(1+\delta_{c(\pi_1)\equiv 1(2)})$ and let $A\subseteq\mathbf{C}$ be a $\mathbf{Z}[q^{-1/2},\mu_{\nu q^\tau}]$-algebra containing $\chi$. Let $(\phi,g_1,g_2)\in \mathcal{S}(F^2)\times G(F)$ be a $\Pi$-integral datum.Then, \eqref{eq: Problem} holds.
    \begin{proof}
        The local $L$-factor in this case is identically equal to $1$. Without loss of generality, we will also assume that $c(\chi_1)\geq c(\chi_2)$.  We use the usual coset representatives $\gamma=(z_1g_{t,k_1,v_1},z_2ng_{0,k_2,v_2})$ for $P(F)\backslash G(F)/K_\Pi$, where $z_i\in Z(F),n=\left[\begin{smallmatrix}
            1 & n_0\\
            & 1
        \end{smallmatrix}\right]\in N(F),t\in\mathbf{Z}, 0\leq k_i\leq c(\pi_i)$. As we do in every case involving supercuspidal representations, we start with the finite sum 
        \begin{align}\label{eq: 44}
            &{\mathcal{I}_{\Pi}^\mathrm{new}(\gamma;s)=C\sum_{i\geq -k_2-c(\pi_2)}q^{i(1-s)}\int_{\mathcal{O}^\times}(\omega_{\pi_1}\omega_{\pi_2})(y)\psi(-n_0y\varpi^i)W_{\pi_1}^\mathrm{new}(g_{i+t,k_1,v_1y^{-1}})W_{\pi_2}^\mathrm{new}(g_{i,k_2,v_2y^{-1}})\ d^\times y}.
        \end{align}
        The boundary case $k_1=c(\pi_1)$ and $k_2=c(\pi_2)$ is exactly the same as the one in \Cref{thm sup + half-ram p-s}. For the cases where either $k_1\neq c(\pi_1)$ or $k_2\neq c(\pi_2)$ the treatment is extremely similar to prior proofs. However, due to the introduction of the new type of $K$-functions in this case, we will give the details in one of the various cases arising in \cite[Lemma $3.6$]{Assing_2018}. Let's assume that $c(\chi_1)<k_2<c(\pi_2).$ Then, by \cite[Lemma $3.6$]{Assing_2018}, and \Cref{prop new vs conj. new}, we see that
        \eqref{eq: 44} is an $A[q^s,q^{-s}]$-multiple of 
        \begin{align*}
           &{\int_{\mathcal{O}^\times}(\omega_{\pi_1}\omega_{\pi_2})(y)\psi(-n_0y\varpi^{-2k_2})W_{\pi_1}^\mathrm{new}(g_{-2k_2+t,k_1,v_1y^{-1}}) {(q-1)^2}K\left(\chi_1^{-1}\otimes \chi_2^{-1}, (\varpi^{-k_2},\varpi^{-k_2}),v_2y^{-1}\varpi^{-k_2}\right)\ d^\times y.} \end{align*}
           \begin{align*}=\int_{\mathcal{O}^\times}\int_{\mathcal{O}^\times}\int_{\mathcal{O}^\times} (\omega_{\pi_1}&\omega_{\pi_2})(y)\psi(-n_0y\varpi^{-2k_2})W_{\pi_1}^\mathrm{new}(g_{-2k_2+t,k_1,v_1y^{-1}}) \\
           &(q-1)^2\chi_1^{-1}(x_1)\chi_2^{-1}(x_2)\psi(x_1\varpi^{-k_2}+x_2\varpi^{-k_2}+v_2y^{-1}x_1x_2\varpi^{-k_2})\ d^\times x_1\ d^\times x_2\ d^\times y.
        \end{align*}
        Noting that $y^{-1}x_1x_2=y(y^{-1}x_1)(y^{-1}x_2)$ and applying a double change of variables on $x_1$ and $x_2$, and then passing from $d^\times x_1$ to $dx_1$ and from $d^\times x_2$ to $dx_2$, this is (up to an inconsequential integral power of $q$) equal to
        \begin{align*}\int_{\mathcal{O}^\times}\int_{\mathcal{O}^\times}\int_{\mathcal{O}^\times} (\omega_{\pi_1}&\omega_{\pi_2})(y)\psi(-n_0y\varpi^{-2k_2})W_{\pi_1}^\mathrm{new}(g_{-2k_2+t,k_1,v_1y^{-1}}) \\
           &\chi_1^{-1}(x_1y)\chi_2^{-1}(x_2y)\psi(x_1y\varpi^{-k_2}+x_2y\varpi^{-k_2}+v_2yx_1x_2\varpi^{-k_2})\ dx_1\ dx_2\ d^\times y.
        \end{align*}
        Now, we can choose $M\geq 1$ large enough such that
        $\chi_1^{-1}(x_1y)\chi_2^{-1}(x_2y)\psi(x_1y\varpi^{-k_2}+x_2y\varpi^{-k_2}+v_2yx_1x_2\varpi^{-k_2})$ is constant for all $x_1,x_2\in 1+ \varpi^M\mathcal{O}$. Thus, everything reduces to a finite $A$-linear combination of integrals 
        $$\int_{\mathcal{O}^\times}\omega_{\pi_1}(y)\psi(-n_1y\varpi^{-2k_2}) W_{\pi_1}^\mathrm{new}(g_{-2k_2+t,k_1,v_1y^{-1}})\ d^\times y$$
        for finitely many $n_1\in F$. These are handled in the exact same manner as in \Cref{thm sup. x unr. p-s}, also recalling that in the non-boundary case the inverse volume factor $\vol_{P(F)}(P(F)\cap \gamma K_\Pi\gamma^{-1})^{-1}$ is always an integral multiple of $q-1$. Every other non-boundary case appearing in \cite[Lemma $3.6$]{Assing_2018} is dealt in a very similar or simpler manner. We also note that in one of the cases, $\epsilon(\tfrac{1}{2},\pi_2)$ appears. It is well-known that this is equal to $\epsilon(\tfrac{1}{2},\chi_1\chi)\epsilon(\tfrac{1}{2},\chi_2\chi^{-1})$ which is contained in $A$ by construction.
    \end{proof}
\end{prop}
\subsubsection{One supercuspidal representation and one ramified twist of the Steinberg representation}
Without loss of generality, we will assume that $\pi_1$ is a supercuspidal representation and $\pi_2=\mathrm{St}_\chi$ for some ramified character $\chi$ of $F^\times$. By \cite{schmidt2002some}, one has $c(\mathrm{St}_\chi)=2c(\chi)$. The behavior of the new vector for $\mathrm{St}_\chi$ is fully described in \cite[Lemma $3.3$]{Assing_2018}, or \cite[Lemma $3.3.4$]{Assing_2019} 
\begin{prop}\label{thm sup x ram St}
    Let $\pi_1=\pi(\xi)$ be a unitary supercuspidal representation, and $\pi_2=\mathrm{St}_\chi$ be a ramified unitary twist of the Steinberg representation, with $\omega_{\pi_1}(\varpi)=\chi(\varpi)=1$. Let $\tau:=\max\{c(\pi_1),c(\chi)\}$, $\nu:=(q^{f(E/F)}-1)(1+\delta_{c(\pi_1)\equiv 1(2)})$ and $A\subseteq \mathbf{C}$ be a $\mathbf{Z}[q^{-1/2},\mu_{\nu q^\tau}]$-algebra. Let $(\phi,g_1,g_2)\in\mathcal
    S(F^2)\times G(F)$ be a $\Pi$-integral datum. Then, \eqref{eq: Problem} holds.
\end{prop}
\begin{proof}
    The local $L$-factor is once more identically equal to $1$, We use the usual coset representatives $\gamma=(z_1g_{t,k_1,v_1},z_2ng_{0,k_2,v_2})$ for $P(F)\backslash G(F)/K_\Pi$, where $z_i\in Z(F),n=\left[\begin{smallmatrix}
            1 & n_0\\
            & 1
        \end{smallmatrix}\right]\in N(F),t\in\mathbf{Z}, 0\leq k_i\leq c(\pi_i)$, and start with 
    \begin{align}\label{eq: 45}
            &{\mathcal{I}_{\Pi}^\mathrm{new}(\gamma;s)=C\sum_{i\geq -k_2-c(\pi_2)}q^{i(1-s)}\int_{\mathcal{O}^\times}(\omega_{\pi_1}\chi^2)(y)\psi(-n_0y\varpi^i)W_{\pi_1}^\mathrm{new}(g_{i+t,k_1,v_1y^{-1}})W_{\pi_2}^\mathrm{new}(g_{i,k_2,v_2y^{-1}})\ d^\times y}.
        \end{align}
        By \cite[Lemma $3.3$]{Assing_2018}, the treatment of the boundary case, $k_1=c(\pi_1)$ and $k_2=c(\pi_2)$, is identical to that of \Cref{thm sup + half-ram p-s}. We now discuss the non-boundary cases, i.e. $k_1\neq c(\pi_1)$ or $k_2\neq c(\pi_2)$, where we recall that the volume factors in these cases are easily seen to be integral multiples of $q-1$. All of the input regarding the various expressions for the new vectors comes from \cite[Lemma $3.3$]{Assing_2018}. If $k_2=0$, then \eqref{eq: 45} reduces to an $A$-multiple of 
        $$\epsilon(\tfrac{1}{2},\pi_2)\int_{\mathcal{O}^\times}(\omega_{\pi_1}\chi^2)(y)\psi(-n_0y\varpi^{-c(\pi_2)})W_{\pi_1}^\mathrm{new}(g_{-c(\pi_2)+t,k_1,v_1y^{-1}})\ d^\times y$$
        and it is well-known that $\epsilon(\tfrac{1}{2},\pi_2)=\epsilon(\tfrac{1}{2},\chi)^2$, which is an element of $A$. Then this is handled in the same manner as \Cref{thm sup. x unr. p-s}. If $k_2>0$, then \eqref{eq: 45} is a finite $A[q^s,q^{-s}]$-linear sum of integrals:
        \begin{enumerate}
            \item $\displaystyle\int_{\mathcal{O}^\times}\omega_{\pi_1}(y)\psi\left((-n_0\varpi^{-2c(\pi_2)}+v_2^{-1}\varpi^{c(\pi_2)})y\right)W_{\pi_1}^\mathrm{new}(g_{-2c(\pi_2)+t,k_1,v_1y^{-1}})\ d^\times y$
            \item $\displaystyle\int_{\mathcal{O}^\times}(\omega_{\pi_1}\chi)(y)\psi\left(-n_0\varpi^{i}y\right)W_{\pi_1}^\mathrm{new}(g_{i+t,k_1,v_1y^{-1}})\ d^\times y$, with $i>-2$
            \item $\displaystyle\int_{\mathcal{O}^\times}(\omega_{\pi_1}\chi^2)(y)\psi(-n_0y\varpi^i)W_{\pi_1}^\mathrm{new}(g_{i+t,k_1,v_1y^{-1}})(q-1)^2K(\chi^{-1}\otimes \chi^{-1},(\varpi^{\tfrac{j}{2}},\varpi^{\tfrac{j}{2}}),v_2y^{-1}\varpi^{-k_2})\ d^\times y$ where $j$ is even and $k_2\neq c(\pi_2)$
            \item $\displaystyle\int_{\mathcal{O}^\times}(\omega_{\pi_1}\chi)(y)\psi\left(-n_0\varpi^{-2}y\right)W_{\pi_1}^\mathrm{new}(g_{-2+t,k_1,v_1y^{-1}}) (q-1)S(1,-b_{\chi^{-1}} v_2^{-1}y,1)\ d^\times y$.
        \end{enumerate}
        The first two cases are dealt identically, in the same fashion as \Cref{thm sup. x unr. p-s}. We've seen how to deal with the third case in  \Cref{thm sup + fully ram p-s}. The fourth and final type of integral that might appear is one that we haven't encountered so far, so we give more details. The $S$-function showing up is given by
        $$S(A,B,m):=\int_{\mathcal{O}^\times}\psi((Ax+Bx^{-1})\varpi^{-m})\ d^\times x,\ A,B\in\mathcal{O},\ m\in\mathbf{Z}_{>0}$$
        Thus, the fourth type of integral appearing is of the form
        $$(q-1)\int_{\mathcal{O}^\times}\int_{\mathcal{O}^\times}(\omega_{\pi_1}\chi)(y)\psi\left(-n_0\varpi^{-2}y\right)W_{\pi_1}^\mathrm{new}(g_{-2+t,k_1,v_1y^{-1}})\psi\left((x-b_{\chi^{-1}} v_2^{-1}yx^{-1})\varpi^{-1}\right)\ d^\times x\ d^\times y.$$
       Passing to the additive measure, this becomes (up to an integral power of $q$) equal to $$\int_{\mathcal{O}^\times}\int_{\mathcal{O}^\times}(\omega_{\pi_1}\chi)(y)\psi\left(-n_0\varpi^{-2}y\right)W_{\pi_1}^\mathrm{new}(g_{-2+t,k_1,v_1y^{-1}})\psi\left((x-b_{\chi^{-1}} v_2^{-1}yx^{-1})\varpi^{-1}\right)\ d x\ d^\times y.$$
       We notice that $\psi\left((x-b_{\chi^{-1}} v_2^{-1}yx^{-1})\varpi^{-1}\right)=\psi\left((1-b_{\chi^{-1}} v_2^{-1}y)\varpi^{-1}\right)$ for all $x\in 1+\varpi\mathcal{O}$ and $y\in\mathcal{O}^\times$, and it is simply a $q$-th root of unity. Then, this once again reduces to a finite sum of the familiar integrals (Gauss sums) dealt with before. 
\end{proof}
Now that we have introduced the functions $S(A,B,m)$ and seen how to handle them, one can go back and also prove \Cref{thm sup + fully ram p-s} in the case where $\chi_1=\chi_2$, using \cite[Lemma $3.5$]{Assing_2018}. The proof is essentially the same as \Cref{thm sup x ram St}. A small detail to keep in mind, is that in \textit{loc.cit.}, ``$G(\varpi^{-l},\chi)$'' in the second line of page $14$, should read ``$G(v\varpi^{-l},\chi)$''. This can be found corrected in \cite[Lemma $3.3.8$]{Assing_2019}.
\subsubsection{Two half-ramified principal-series representation or one half-ramified principal-series representation and one unramified twist of the Steinberg representation} In this section we deal with the case $\pi_1\times\pi_2$ where one representation is a half-ramified principal-series representation and the other is either a half-ramified principal-series representation or an unramified twist of the Steinberg representation. 
\begin{thm}\label{thm two half-ram p-s}
    Let $\pi_i=I(\chi_i\omega_i,\chi_i^{-1})$ for $i=1,2$ be half-ramified unitary principal-series representation representations with $\omega_i(\varpi)=1.$ Set $\Pi:=\pi_1\times\pi_2$, Let $\tau:=\max\{c(\omega_1),c(\omega_2)\}$, $\nu:=q-1$ and let $A\subseteq\mathbf{C}$ be a $\mathbf{Z}[q^{-1/2},\mu_{\nu q^\tau}]$-algebra containing $\chi_i$ for $i=1,2$. Let $(\phi,g_1,g_2)$ be a $\Pi$-integral datum. Then \eqref{eq: Problem} holds.
\end{thm}
\begin{proof}
    The $L$-factor $L(\Pi,s)$ is given by $L(\chi_1^{-1}\chi_2^{-1},s)$ if $c(\omega_1\omega_2)>0$, and by $L(\chi_1\chi_2,s)L(\chi_1^{-1}\chi_2^{-1},s)$ if $c(\omega_1\omega_2)=0$. We use the usual coset representatives $\gamma=(z_1g_{t,k_1,v_1},z_2ng_{0,k_2,v_2})$ for $P(F)\backslash G(F)/K_\Pi$, where $z_i\in Z(F),n=\left[\begin{smallmatrix}
            1 & n_0\\
            & 1
        \end{smallmatrix}\right]\in N(F),t\in\mathbf{Z}, 0\leq k_i\leq c(\pi_i)$, and the starting expression
    \begin{align}\label{eq: 46}
        {\mathcal{I}_{\Pi}^\mathrm{new}(\gamma;s)=C\sum_{i\geq -k_2-c(\pi_2)}q^{i(1-s)}\int_{\mathcal{O}^\times}(\omega_1\omega_2)(y)\psi(-n_0y\varpi^i)W_{\pi_1}^\mathrm{new}(g_{i+t,k_1,v_1y^{-1}})W_{\pi_2}^\mathrm{new}(g_{i,k_2,v_2y^{-1}})\ d^\times y}.
    \end{align}
        To analyze this we'll use \Cref{prop Assing half-ramified princ series}.\\
    If $(k_2,k_2)=(0,0)$, then \eqref{eq: 46} is an $A$-multiple of 
    $$\sum_{i\geq \max\{-c(\pi_2),-t-c(\pi_1)\}}q^{-is}(\chi_1\chi_2)(\varpi^i)\int_{\mathcal{O}^\times}(\omega_1\omega_2)(y)\psi(-n_0y\varpi^i)\ d^\times y.$$
    If $c(\omega_1\omega_2)>0$, then there's only one non-zero term which is a simple Gauss sum contained in $\tfrac{1}{q-1}A$. Thus assume $c(\omega_1\omega_2)=0$. Then we do the usual splitting trick as in \Cref{thm unram + steinberg}. The finite sum term is again a simple Gauss sum associated to the trivial character of $F^\times$ , which is contained in $\tfrac{1}{q-1}A$. Recall that $(k_1,k_2)=(0,0)$ is a non-boundary case thus the volume factor in question is always an integral multiple of $q-1$ thus we need not worry about the $\tfrac{1}{q-1}$ factors. The infinite sum is an  $A$-multiple of the familiar 
    $$\sum_{i\geq 0}q^{-is}(\chi_1\chi_2)(\varpi^i)=L(\chi_1\chi_2,s)\in L(\Pi,s)A[q^s,q^{-s}].$$
    If $k_1\neq c(\pi_1)$ or $k_2\neq c(\pi_2)$, then we are again in a non-boundary case, and looking at \Cref{prop Assing half-ramified princ series}, it is easy to see that \eqref{eq: 46} always collapses to a single $A$-multiple of a partial Gauss sum for the trivial character, $\omega_1,\omega_2$ or $\omega_1\omega_2$ (and $\psi)$ over an appropriate region of integration of the form $1+\varpi^\ell\mathcal{O}$ with $\ell<\tau$. Thus looking at \Cref{sec Gauss sums}, these are always contained in $\tfrac{1}{q-1}A$, and the $\tfrac{1}{q-1}$ factor is always killed by the volume factor since we are in a non-boundary case. Finally, if $(k_1,k_2)=(c(\pi_1),c(\pi_2))$, then we are in the boundary case, and \eqref{eq: 46} is an $A$-multiple of 
    \begin{align}\label{eq: 47}\sum_{i\geq M}q^{-is}(\chi_1^{-1}\chi_2^{-1})(\varpi^i)\int_{\mathcal{O}^\times}\psi(-n_0y\varpi^i)\psi(-v_1^{-1}y\varpi^{i+t+c(\pi_1)})\psi(v_2^{-1}\varpi^{i+c(\pi_2)}y)\ d^\times y\end{align}
    where $M:=\max\{-2c(\pi_2),-t-2c(\pi_1)\}$. Doing the same splitting trick as before, we obtain an $A$-multiple of $L(\chi_1^{-1}\chi_2^{-1},s)\in L(\Pi,s)A[q^s,q^{-s}]$, plus a finite $A[q^s,q^{-s}]$-linear sum of integrals (Gauss sums) of the form
    $$\int_{\mathcal{O}^\times}\psi\left((-n_0v_1\varpi^{i+c(\pi_1)}+v_1v_2^{-1}\varpi^{i+c(\pi_1)+c(\pi_2)}-\varpi^{i+t+2c(\pi_1)})\varpi^{-c(\pi_1)}y\right)\ d^\times y$$
    with $i\geq M$. The above integral is always contained in $\tfrac{1}{q-1}A$. In addition, it is not contained in $A$ if and only if 
    \begin{align}\label{eq: 48}v\left(-n_0v_1\varpi^{i+c(\pi_1)}+v_1v_2^{-1}\varpi^{i+c(\pi_1)+c(\pi_2)}-\varpi^{i+t+2c(\pi_1)}\right)=c(\pi_1)-1.\end{align}
    In this case, we write $M=N+(-t-2c(\pi_1))$ for some uniquely determined $N\in\mathbf{Z}_{\geq 0}$. Then every $i\geq M$, can be written as $i=j+N+(-t-2c(\pi_1))$ for $j\in\mathbf{Z}_{\geq 0}$. Using this, \eqref{eq: 48} can be re-written as
    \begin{align}\label{eq: 49}
        v\left(-n_0v_1\varpi^{-t-c(\pi_1)}+v_1v_2^{-1}\varpi^{-t-c(\pi_1)+c(\pi_2)}-1\right)=c(\pi_1)-N-j-1.
    \end{align}
    We now slightly adapt the argument at the end of \Cref{thm sup + half-ram p-s}. The conditions \eqref{eq: 41} and \eqref{eq: 42} remain unchanged. The bottom-right entry of \eqref{eq: 41} implies that 
    \begin{align}\label{eq: 50}b\varpi^{-t-c(\pi_1)}\in v_2^{-1}(a-1)\varpi^{-t-c(\pi_1)+c(\pi_2)}+\varpi^{-t-c(\pi_1)+2c(\pi_2)}\mathcal{O}.\end{align}
    If $M=-t-2c(\pi_1)$, then one proceeds exactly as in the final step of \Cref{thm sup + half-ram p-s}. If $M=-2c(\pi_2)$, then $-t-c(\pi_1)+2c(\pi_2)=c(\pi_1)-N< c(\pi_1)$. Hence \eqref{eq: 50} together with the bottom-right entry of \eqref{eq: 42} give us that 
    $$a\left(1+n_0v_1\varpi^{-t-c(\pi_1)}-v_1v_2^{-1}\varpi^{-t-c(\pi_1)+c(\pi_2)}\right) \in 1+n_0v_1\varpi^{-t-c(\pi_1)}-v_1v_2^{-1}\varpi^{-t-c(\pi_1)+c(\pi_2)}+\varpi^{c(\pi_1)-N}\mathcal{O}.$$
    However, \eqref{eq: 49} now implies that $a\in1+\varpi^{j+1}\mathcal{O}$, and $j\in\mathbf{Z}_{\geq 0}$ hence we are done.
\end{proof}

\begin{prop}
    Let $\pi_1=I(\chi_1\omega,\chi_1^{-1})$ be a half-ramified unitary principal-series representation with $\omega(\varpi)=1$, and $\pi_2=\mathrm{St}_\chi$ an unramified twist of the Steinberg representation. Set $\Pi:=\pi_1\times \pi_2$, let $\tau:=c(\omega)$, $\nu:=q-1$ and let $A\subseteq \mathbf{C}$ be a $\mathbf{Z}[q^{-1/2},\mu_{\nu q^\tau}]$-algebra containing $\chi_1$ and $\chi$. Let $(\phi,g_1,g_2)$ be a $\Pi$-integral datum. Then, \eqref{eq: Problem} holds.
\end{prop}
\begin{proof}
    The $L$-factor in this case is given by $L(\chi\chi_1^{-1},s)$. The proof is extremely similar to that of \Cref{thm two half-ram p-s}, and is in fact simpler. The argument for the boundary case $(k_1,k_2)=(c(\pi_1),1)$ is identical.
\end{proof}

\subsubsection{One unramified principal-series representation and one fully-ramified principal-series representation or ramified twist of the Steinberg representation}
In this section we consider $\pi_1\times\pi_2$ where one of the representations is an unramified principal-series representation and the other one is either a fully-ramified principal-series representation or a ramified twist of the Steinberg representation. 
\begin{prop}\label{thm unr p-s x fully ram p-s}
    Let $\pi_1=I(\chi_1\chi,\chi_2\chi^{-1})$ be a fully-ramified unitary principal-series representation with $\chi_1,\chi_2$ distinct, $\chi_i(\varpi)=1$, and let $\pi_2$ be an unramified principal-series representation. Set $\Pi:=\pi_1\times \pi_2$, let $\tau:=\max\{c(\chi_1),c(\chi_2)\}$, $\nu:=q-1$ and $A\subseteq\mathbf{C}$ be a $\mathbf{Z}[q^{-1/2},\mu_{\nu q^\tau}]$-algebra containing $\chi,\omega_{\pi_2}$ and the Spherical Hecke eigenvalues of $\pi_2$. Let $(\phi,g_1,g_2)\in\mathcal{S}(F^2)\times G(F)$ be a $\Pi$-integral datum. Then, \eqref{eq: Problem} holds.
\end{prop}
\begin{proof}
    In this case, the $L$-factor $L(\Pi,s)$ is identically equal to $1$. We work with coset representatives 
     $\gamma:=(z_1g_{t,k,v},z_2n),\ \ z_i\in Z(F),n\in N(F),t\in\mathbf{Z}, 0\leq k\leq c(\pi_1), v\in\mathcal{O}^\times$
     and the starting expression 
     \begin{align}\label{eq: 51}
         \mathcal{I}_{\Pi}^\mathrm{new}(\gamma;s)&{=C\cdot \sum_{i\in\mathbf{Z}_{\geq 0}}q^{-\tfrac{i}{2}}\mathrm{Sch}_i(\alpha,\beta)q^{-i(s-1)}\int_{\mathcal{O}^\times}\psi(-n_0\varpi^i y)\omega_{\pi_1}(y) W_{\pi_1}^\mathrm{new}(g_{i+t,k,vy^{-1}})\ d^\times y}
     \end{align}
     The main input towards studying this, is once again \cite[Lemma $3.6$]{Assing_2018}. We start by noting that in the boundary case $k=c(\pi_1)$, \eqref{eq: 51} collapses to an $A[q^s,q^{-s}]$-multiple of
     $$\int_{\mathcal{O}^\times}\psi((-n_0-v^{-1}\varpi^{t+c(\pi_1)})\varpi^{-t-2c(\pi_1)}y)\ d^\times y.$$
     This \textcolor{black}{$A[q^s,q^{-s}]$-multiple} is non-zero only if $-t-2c(\pi_1)\geq 0$. This integral is always contained in $\tfrac{1}{q-1}A$. It is not contained in $A$ itself if and only if $v(-n_0-v^{-1}\varpi^{t+c(\pi_1)})=t+2c(\pi_1)-1$. In this case, the $\tfrac{1}{q-1}$ factor is killed by the corresponding volume factor in exactly the same way as in the very last part of \Cref{thm tamely ram p.s + unr p.s}. The non-boundary cases, i.e. $k\neq c(\pi_1)$ are again familiar, and they fall within the cases we've dealt with before. We give details for the different types of cases that could arise. Recall that for these non-boundary cases, the corresponding volume factors $\vol_{P(F)}(P(F)\cap \gamma K_\Pi\gamma)^{-1}$ are always integral multiples of $q-1$, so we shall not worry about them. If $k=0$, then \eqref{eq: 51} collapses to an $A$-multiple of
     $$\epsilon(\tfrac{1}{2},\pi_1)\int_{\mathcal{O}^\times}\omega_{\pi_1}(y)\psi(-n_0\varpi^{-t-c(\pi_1)}y)\ d^\times y$$
     and since $\epsilon(\tfrac{1}{2},\pi_1)=\epsilon(\tfrac{1}{2},\chi_1\chi)\epsilon(\tfrac{1}{2},\chi_2\chi^{-1})$, this whole thing is an element of $\tfrac{1}{q-1}A$.\\
     If $k\notin \{0,c(\pi_1),c(\pi_2)\}$ then let's say for $c(\chi_1)<k<c(\pi_1)$ (the other two cases showing up in \cite[Lemma $3.6$]{Assing_2018} for  $k\notin \{0,c(\pi_1),c(\pi_2)\}$ are identical up to numerology), \eqref{eq: 51} collapses to an $A$-multiple of
     \begin{align}\label{eq: 52}\int_{\mathcal{O}^\times}\omega_{\pi_1}(y)\psi(-n_0\varpi^{-t-2k})(q-1)^2K(\chi_1^{-1}\otimes \chi_2^{-1},(\varpi^{-k},\varpi^{-k}),vy^{-1}\varpi^{-k})\ d^\times y.\end{align}
     But we have seen how to handle expressions of this form in \Cref{thm sup + fully ram p-s}, to reduce them into a Gauss sum for $\omega_{\pi_1}$ , which always lies in $\tfrac{1}{q-1}A$. \\
     If $k=c(\chi_i)\neq c(\chi_j)$ for $\{i,j\}=\{1,2\}$, then four cases arise, two of which are in the same spirit as \eqref{eq: 52}, so we won't discuss them further. The other two are identical between them so we give the details for $k=c(\chi_1)\neq c(\chi_2)$. In this case, \eqref{eq: 52} is made up a finite $A$-linear sum of expressions like \eqref{eq: 52}, summed up with an $A[q^s,q^{-s}]$-multiple of
     \begin{align}\label{eq: 53}
         \sum_{i\geq \max\{0,-t-c(\pi_1)\}}\mathrm{Sch}_i(\alpha,\beta)q^{-is}\int_{\mathcal{O}^\times}\omega_{\pi_1}(y)\psi(-n_0\varpi^{i}y)(q-1)^2\mathscr{G}(\chi_2^{-1}\chi_1,\varpi^{-c(\chi_1)})\mathscr{G}(\chi_1^{-1},vy^{-1}\varpi^{-c(\chi_1)})\ d^\times y.
     \end{align}
    It is clear that $(q-1)\mathscr{G}(\chi_1,vy^{-1}\varpi^{-c(\chi_1)})$ is an $A$-multiple of $\chi_1(y^{-1})$, and $(q-1)\mathscr{G}(\chi_2^{-1}\chi_1,\varpi^{-c(\chi_1)})$ is always an element of $A$ (possibly zero). Thus \eqref{eq: 53} becomes an $A$-multiple of 
    $$ \sum_{i\geq \max\{0,-t-c(\pi_1)\}}\mathrm{Sch}_i(\alpha,\beta)q^{-is}\int_{\mathcal{O}^\times}\chi_2(y)\psi(-n_0\varpi^iy)\ d^\times y.$$
    Since $\chi_2$ is ramified, this infinite sum collapses to an $A$-multiple of a single Gauss sum for $\chi_2$, and hence an element of $\tfrac{1}{q-1}A$. \\
    Finally, if $k=c(\chi_1)=c(\chi_2)$, again looking at the long explicit formulas in \cite[Lemma $3.6$]{Assing_2018} regarding the behavior of the new vector, we see that \eqref{eq: 51} is made up of a finite $A[q^s,q^{-s}]$-linear combination of integrals we've dealt with before, summed 
 with an $A$-multiple of
 \begin{align*} \sum_{i\geq \max\{0,-t-2\}}\mathrm{Sch}_i(\alpha,\beta)q^{-is}\chi(\varpi^i)&\int_{\mathcal{O}^\times}\chi_2(y)\psi(-n_0\varpi^iy)\ d^\times y \\
 &+  \sum_{i\geq \max\{0,-t-2\}}\mathrm{Sch}_i(\alpha,\beta)q^{-is}\chi^{-1}(\varpi^i)\int_{\mathcal{O}^\times}\chi_1(y)\psi(-n_0\varpi^iy)\ d^\times y.\end{align*}
 However, since both $\chi_1$ and $\chi_2$ are ramified, both infinite sums are actually finite, and made up of $A[q^s,q^{-s}]$-multiples of a single Gauss sum (for $\chi_2$ and $\chi_1$ respectively) contained in $\tfrac{1}{q-1}A$.
\end{proof}
Once again, \Cref{thm unr p-s x fully ram p-s} also holds for $\chi_1=\chi_2$ and the proof is extremely similar. One uses \cite[Lemma $3.5$]{Assing_2018} instead of \cite[Lemma $3.6$]{Assing_2018}. The $S$-functions appearing are handled as in \Cref{thm sup x ram St}. The $L$-factor in this case is again identically equal to $1$, and infinite sums collapse in the same way due to ramification. 

\begin{prop}\label{thm unr p-s x ram St}
    Let $\pi_1=\mathrm{St}_\chi$ for some ramified unitary character $\chi$, with $\chi(\varpi)=1$, and let $\pi_2$ be an unramified principal-series representation. Set $\Pi:=\pi_1\times \pi_2$, $\tau:=c(\chi)$ and $\nu:=q-1$. Let $A\subseteq\mathbf{C}$ be a $\mathbf{Z}[q^{-1},\mu_{\nu q^\tau}]$-algebra containing $\omega_{\pi_2}$, and the spherical Hecke eigenvalues of $\pi_2$. Let $(\phi,g_1,g_2)\in\mathcal{S}(F^2)\times G(F)$ be a $\Pi$-integral datum. Then, \eqref{eq: Problem} holds.
\end{prop}
\begin{proof}
    The $L$-factor in this case is again identically equal to $1$. Looking at \cite[Lemma $3.3$]{Assing_2018} once more, it is not hard to see that the proof in this case is extremely similar to \Cref{thm unr p-s x fully ram p-s}, and is in fact simpler since there are fewer cases to check. The boundary case $k=c(\pi_1)$ is exactly the same. The case $k=c(\chi)$, where the new vector does not have ``finite'' support, collapses in the same manner as in \Cref{thm unr p-s x fully ram p-s}, since $\chi$ is ramified. The $S$-function appearing in this case is handled in exactly the same manner as \Cref{thm sup x ram St}.
\end{proof}
\subsubsection{One unramified twist of the Steinberg representation and one fully-ramified principal-series representation or ramified twist of the Steinberg representation}

In this section we deal with $\pi_1\times \pi_2$ where one representation is an unramified twist of the Steinberg representation, and the other one is either a fully-ramified principal-series representation or a ramified twist of the Steinberg representation.

\begin{prop}
    Let $\pi_1=I(\chi_1\chi_0,\chi_2\chi_0^{-1})$ be a fully-ramified unitary principal-series representation with $\chi_i(\varpi)=1$ \emph{(}resp. $\pi_1=\mathrm{St}_\chi$ be a unitary ramified twist of the Steinberg representation with $\chi(\varpi)=1$\emph{)} and let $\pi_2=\mathrm{St}_\mu$ for unramified $\mu$. Set $\Pi:=\pi_1\times \pi_2$ and let $\tau:=\max\{c(\chi_1),c(\chi_2)\}$ \emph{(}resp. $\tau:=c(\chi)$\emph{)} and $\nu:=q-1$. Let $A\subseteq\mathbf{C}$ be a $\mathbf{Z}[q^{-1/2},\mu_{\nu q^\tau}]$-algebra containing $\chi_0$ and $\mu$ \emph{(}resp. $\mu$\emph{)}. Let $(\phi,g_1,g_2)\in\mathcal{S}(F^2)\times G(F)$ be a $\Pi$-integral datum. Then, $\eqref{eq: Problem}$ holds.
\end{prop}
\begin{proof}
    The $L$-factor $L(\Pi,s)$ in all cases is once more identically equal to $1$. In this case, one reverts back to using the coset representatives $\gamma=(z_1g_{t,k_1,v_1},z_2ng_{0,k_2,v_2})$, where $z_i\in Z(F),n=\left[\begin{smallmatrix}
            1 & n_0\\
            & 1
        \end{smallmatrix}\right]\in N(F),t\in\mathbf{Z}, 0\leq k_i\leq c(\pi_i)$ (recall that $c(\pi_2)=1)$. The starting expressions 
        are the ones found in \textit{Case} $1$ and \textit{Case} $2$ of the proof of \Cref{thm sup x unr St}, depending on whether $k_2$ is $0$ or $1$. For the non-boundary cases, i.e. $(k_1,k_2)\neq (c(\pi_1),1)$, the treatment is identical to that of \Cref{thm unr p-s x fully ram p-s} and \Cref{thm unr p-s x ram St}. For the boundary case $(k_1,k_2)=(c(\pi_1),1)$, the starting expression (i.e. the one of \textit{Case} $2$ in \Cref{thm sup x unr St}) is essentially \eqref{eq: 34}, and thus balanced with the volume factor in the same exact way. The main input for this is \cite[Lemma $3.3$, Lemma $3.5$ \& Lemma $3.6$]{Assing_2018}.
        \end{proof}

\subsubsection{One half-ramified principal-series representation and one fully-ramified principal-series representation or ramified twist of the Steinberg representation}

In this section we settle the case $\pi_1\times \pi_2$, where one of the representations is a half-ramified principal-series representation and the other one is either a fully-ramified principal-series representation or a ramified twist of the Steinberg representation.

\begin{prop}\label{thm full ram p-s x half ram p-s}
    Let $\pi_1=I(\chi_1\chi,\chi_2\chi^{-1})$ be a fully-ramified unitary principal-series representation with $\chi_i(\varpi)=1$ and $\pi_2=I(\mu\omega,\mu^{-1})$ be a half-ramified unitary principal-series representation, with $\omega(\varpi)=1$. Set $\Pi:=\pi_1\times\pi_2$, let $\tau:=\max\{c(\chi_1),c(\chi_2),c(\omega)\}$, $\nu:=q-1$ and $A\subseteq\mathbf{C}$ be a $\mathbf{Z}[q^{-1/2},\mu_{\nu q^\tau}]$-algebra containing $\chi$ and $\mu$. Let $(\phi,g_1,g_2)\in \mathcal{S}(F^2)\times G(F)$ be a $\Pi$-integral datum. Then, \eqref{eq: Problem} holds.
\end{prop} 
\begin{proof}
    In this case, the $L$-factor $L(\Pi,s)$ is given by $L(\chi_1\chi\mu\omega,s)L(\chi_2\chi^{-1}\mu\omega,s)$. We again assume without loss of generality, that $c(\chi_1)\geq c(\chi_2).$ We use the coset representatives $\gamma=(z_1g_{t,k_1,v_1},z_2ng_{0,k_2,v_2})$, where $z_i\in Z(F),n=\left[\begin{smallmatrix}
            1 & n_0\\
            & 1
        \end{smallmatrix}\right]\in N(F),t\in\mathbf{Z}, 0\leq k_i\leq c(\pi_i)$. The starting expression is as in \eqref{eq: 39}, i.e.
        \begin{align}\label{eq: 54}
            {\mathcal{I}_{\Pi}^\mathrm{new}(\gamma;s)=C\sum_{i\geq -k_2-c(\pi_2)}q^{i(1-s)}\int_{\mathcal{O}^\times}(\omega_{\pi_1}\omega)(y)\psi(-n_0y\varpi^i)W_{\pi_1}^\mathrm{new}(g_{i+t,k_1,v_1y^{-1}})W_{\pi_2}^\mathrm{new}(g_{i,k_2,v_2y^{-1}})\ d^\times y}.
        \end{align}
        Looking at \Cref{prop Assing half-ramified princ series} and \cite[Lemma $3.6$ \& Lemma $3.5$]{Assing_2018}, one sees that the boundary case $(k_1,k_2)=(c(\pi_1),c(\pi_2))$ is the same as \eqref{eq: 40} and is thus balanced with the volume factor $\vol_{P(F)}(P(F)\cap \gamma K_\Pi\gamma^{-1})^{-1}$, in the exact same way. We now deal with the non-boundary cases, i.e. $(k_1,k_2)\neq (c(\pi_1),c(\pi_2))$ where we need not worry about the volume factors since now we know they're always integral multiples of $q-1$. If $k_1\notin \{c(\chi_1),c(\chi_2)\}$ or $k_2\notin\{0,c(\pi_2)\}$, then looking at \Cref{prop Assing half-ramified princ series} and \cite[Lemma $3.5$ \& Lemma $3.6$]{Assing_2018}, we see that \eqref{eq: 54} is simply a finite $A[q^s,q^{-s}]$-linear combination of integrals we've dealt with in previous proofs. These are just (partial) Gauss sums for characters made up of $\chi_1,\chi_2$ and $\omega$. All such integrals are in $\tfrac{1}{q-1}A$.\\ 
        Now let's assume that $k_1=c(\chi_1)\neq c(\chi_2)$ (in particular $\chi_1\neq \chi_2$, thus we are in the setting of \cite[Lemma $3.6$]{Assing_2018}) and $k_2=0$. Then, \eqref{eq: 54} is made up of a finite $A[q^s,q^{-s}]$-linear combination of familiar integrals, summed with an $A$-multiple of
        $${\sum_{i}q^{-is}\chi^{-1}(\varpi^i)\mu(\varpi^i)\int_{\mathcal{O}^\times}(\omega_{\pi_1}\omega)(y)\psi(-n_0y\varpi^i)(q-1)^2\mathscr{G}(\varpi^{-c(\chi_1)},\chi_2^{-1}\chi_1)\mathscr{G}(v_1y^{-1}\varpi^{-c(\chi_1)},\chi_1^{-1})\ d^\times y}$$
        where the sum is over integers $i \geq \max\{-c(\pi_2),-c(\chi_1)\}$.
        As in a previous case, we know that $(q-1)\mathscr{G}(\varpi^{-c(\chi_1)},\chi_2^{-1}\chi_1)\in A$, and $(q-1)\mathscr{G}(v_1y^{-1}\varpi^{-c(\chi_1)},\chi_1^{-1})$ is an $A$-multiple of $\chi_1(y^{-1}).$ Thus, the above becomes
\begin{align}\label{eq: 55}{\sum_{i\geq \max\{-c(\pi_2),-c(\chi_1)\}}q^{-is}\chi^{-1}(\varpi^i)\mu(\varpi^i)\int_{\mathcal{O}^\times}(\chi_2\omega)(y)\psi(-n_0y\varpi^i)\ d^\times y.}\end{align}
If $\chi_2\omega$ is ramified, then this infinite sum collapses to a single $A[q^s,q^{-s}]$-multiple of Gauss sum contained in $\tfrac{1}{q-1}A$, hence we are done. If $\chi_2\omega$ is unramified, then applying the usual splitting trick into a finite error term and an $L$-factor term, \eqref{eq: 55} becomes a finite $A[q^s,q^{-s}]$-linear combination of Gauss sums which are contained in $\tfrac{1}{q-1}A$, summed with an $A$-multiple of 
$$\sum_{i\geq 0}(\chi^{-1}\mu)(\varpi^i)q^{-is}=L(\chi^{-1}\mu,s)=L(\chi_2\chi^{-1}\mu\omega,s)\in L(\Pi,s)A[q^s,q^{-s}]$$
and thus we are done. The case $k_1=c(\chi_2)\neq c(\chi_1)$, the treatment is identical (with $\chi_1$ and $\chi_2$ swapped). If $k_2=c(\pi_2)$, then the treatment is the same and hence \eqref{eq: 54} is again made up of a finite $A[q^s,q^{-s}]$-linear combination of familiar integrals, summed with an $A$-multiple of 
\begin{align}\label{eq: 54.1}\sum_{i\geq \max\{-2c(\pi_2),-c(\chi_1)\}}q^{-is}\chi^{-1}(\varpi^i)\mu^{-1}(\varpi^i)\int_{\mathcal{O}^\times}\chi_2(y)\psi(-n_1y\varpi^i)\ d^\times y\end{align}
but now $\chi_2$ is always ramified and hence this always collapses to a single $A[q^s,q^{-s}]$-multiple of a Gauss sum contained in $\tfrac{1}{q-1}A$.\\
The case $k_1=c(\chi_1)=c(\chi_2)$ and still $\chi_1\neq \chi_2$, is a combination of these. Firstly assume $k_2=0$, then from \cite[Lemma $3.6$]{Assing_2018} we see that \eqref{eq: 54} is given by a finite $A[q^s,q^{-s}]$-linear combination of familiar integrals (which we have treated before and all reduce to Gauss summs contained in $\tfrac{1}{q-1}A$), summed with an $A$-linear combination of the following two terms
\begin{align*}
    &\sum_{i\geq M}q^{-is}\chi^{-1}(\varpi^i)\mu(\varpi^i)\int_{\mathcal{O}^\times}(\chi_2\omega)(y)\psi(-n_0y\varpi^i)\ d^\times y\\
    &\sum_{i\geq M}q^{-is}\chi(\varpi^i)\mu(\varpi^i)\int_{\mathcal{O}^\times}(\chi_1\omega)(y)\psi(-n_0y\varpi^i)\ d^\times y
\end{align*}
for some large enough $M$.
Then both of these are treated exactly as \eqref{eq: 55}, depending on the ramification of $\chi_i\omega$. In any case, they're always contained in $L(\Pi,s)A[q^s,q^{-s}]$. If $k_2=c(\pi_2)$, then the infinite sums arrising allways collapse to finite ones and are dealt as in \eqref{eq: 54.1}.\\
Finally, we consider the case where $\chi_1=\chi_2$, which puts us in the setting of \cite[Lemma $3.5$]{Assing_2018}. Recall that we are now treating the (non-boundary) case $k_1=c(\chi_1)$. Then, if $k_2=0$, \eqref{eq: 54} is given by an element of $\tfrac{1}{q-1}A[q^s,q^{-s}]$ summed with an $A$-linear combination of the following three expressions (notice there are two $\pm$ cases in the first line below)
\begin{align*}
    &\sum_{i\geq M}q^{-is}(\mu\chi^{\pm 1})(\varpi^i)\int_{\mathcal{O}^\times}(\chi_1\omega)(y)\psi(-n_0y\varpi^i)\ d^\times y\\
    &\sum_{i\geq M}q^{-is}\mathrm{Sch}_{i}\left((\mu\chi)(\varpi),(\mu\chi^{-1})(\varpi)\right)\int_{\mathcal{O}^\times}(\chi_1\omega)(y)\psi(-n_1y\varpi^i)\ d^\times y
\end{align*}
for some appropriately large $M$, and potentially different $n_1\in F$. We note the addition of the Schur polynomials on the second line, evaluated at the Satake parameters of the unramified principal-series representation $\mu\otimes I(\chi,\chi^{-1})$. This term comes from the finite sum appearing at the end of the second line on page $14$ of 
\cite{Assing_2018}. The appearance of this Schur polynomial resembles the unramified principal-series representation case, which is justifiable since in this $\chi_1=\chi_2$ case, $\pi_1$ is indeed a ramified twist of an unramified principal-series representation. Going back to the task at hand, if $\chi_1\omega$ is ramified, everything collapses to single terms in $\tfrac{1}{q-1}A[q^s,q^{-s}]$. Whereas if $\chi_1\omega$ is unramified, then up to a linear $\tfrac{1}{q-1}A[q^s,q^{-s}]$-translate, the three expressions above, are given by $L(\mu\chi^{\pm1},s)$ and $L(\mu\chi,s)L(\mu\chi^{-1},s)=L(\Pi,s)$ respectively. In the case where $k_2=c(\pi_2)$, the approach is the same but then there are no poles.
\end{proof}
\begin{prop}
    Let $\pi_1=\mathrm{St}_\chi$ for a unitary ramified character $\chi$ with $\chi(\varpi)=1$, and let $\pi_2=I(\mu\omega,\mu^{-1})$ be a half-ramified unitary principal-series representation with $\omega(\varpi)=1$. Set $\Pi:=\pi_1\times \pi_2$, let $\tau:=\max\{c(\chi),c(\omega)\}$, $\nu:=q-1$ and $A\subseteq\mathbf{C}$ be a $\mathbf{Z}[q^{-1/2},\mu_{\nu q^\tau}]$-algebra containing $\mu$. Let $(\phi,g_1,g_2)\in\mathcal{S}(F^2)\times G(F)$ be a $\Pi$-integral datum. Then, \eqref{eq: Problem} holds.
\end{prop}
\begin{proof}
    In this case, the $L$-factor is given by $L(\mu\omega\chi|\cdot|^{1/2},s)$.  We use the coset representatives $\gamma=(z_1g_{t,k_1,v_1},z_2ng_{0,k_2,v_2})$, where $z_i\in Z(F),n=\left[\begin{smallmatrix}
            1 & n_0\\
            & 1
        \end{smallmatrix}\right]\in N(F),t\in\mathbf{Z}, 0\leq k_i\leq c(\pi_i)$. The starting expression is as in \eqref{eq: 39}, i.e.
        \begin{align}\label{eq: 57}
            {\mathcal{I}_{\Pi}^\mathrm{new}(\gamma;s)=C\sum_{i\geq -k_2-c(\pi_2)}q^{i(1-s)}\int_{\mathcal{O}^\times}(\chi^2\omega)(y)\psi(-n_0y\varpi^i)W_{\pi_1}^\mathrm{new}(g_{i+t,k_1,v_1y^{-1}})W_{\pi_2}^\mathrm{new}(g_{i,k_2,v_2y^{-1}})\ d^\times y}.
        \end{align}
        By \cite[Lemma $3.3$]{Assing_2018} and \Cref{prop Assing half-ramified princ series}, it is not hard to see that the types of integrals appearing in every case which never gives rise to a pole (i.e. a proper $L$-factor contribution) have been treated before. They all collapse to finite $A[q^s,q^{-s}]$-linear combinations of (partial) Gauss sums for characters made up of $\chi$ and $\omega$, and these expressions are in turn contained in $\tfrac{1}{q-1}A$. Thus, it suffices to check the cases that might produce poles. There are two such cases. Namely $(k_1,k_2)=(c(\chi),0)$ and $(k_1,k_2)=(c(\chi),c(\pi_2))$, which are both non-boundary cases since $c(\pi_1)=2c(\chi)>0$, thus we need not worry about $\tfrac{1}{q-1}$ factors. \\
        Let's assume we are in the first case, then \eqref{eq: 57} is given by a finite $A[q^s,q^{-s}]$-linear combination of integrals which we have treated before, summed with an $A$-multiple of 
        $$\sum_{i\geq M}q^{-is}q^{-i/2} \mu(\varpi^i)\int_{\mathcal{O}^\times}(\chi\omega)(y)\psi(-n_0y\varpi^i)\ d^\times y$$
        for an appropriate large enough $M$.
Hence, if $\chi\omega$ is ramified, this collapses to a finite sum, and is an element of $\tfrac{1}{q-1}A[q^s,q^{-s}]$. If $\chi\omega$ is unramified, then applying the usual splitting trick, this is given by an element of $\tfrac{1}{q-1}A[q^s,q^{-s}]$ summed with an $A$-multiple of 
$$\sum_{i\geq 0}q^{-is}q^{-i/2}\mu(\varpi^i)=L(\mu|\cdot|^{1/2},s)=L(\Pi,s)$$
and hence we are done. If on the other hand $k_2=c(\pi_2)$, then there's actually no pole and the infinite sum collapses to a single familiar expression in $\tfrac{1}{q-1}A[q^s,q^{-s}]$.
\end{proof}

\subsubsection{Two fully-ramified principal-series representations or one fully-ramified principal-series representation and one ramified twist of the Steinberg representation}

In this section, we cover the case of two fully-ramified principal-series representations and the case of one fully-ramified principal-series representation and one ramified twist of the Steinberg representation. The cases are very similar and hence grouped together.
\begin{prop}\label{thm 2 full-ram p-s}
    Let $\pi_1=I(\chi_1\chi,\chi_2\chi^{-1})$ and $\pi_2=I(\mu_1\mu,\mu_2\mu^{-1})$ be two fully-ramified unitary principal-series representations with $\chi_i(\varpi)=\mu_i(\varpi)=1$. Set $\Pi:=\pi_1\times \pi_2$, let $\tau:=\max_{i\in\{1,2\}}\{c(\chi_i),c(\mu_i)\}$, $\nu:=q-1$ and let $A\subseteq\mathbf{C}$ be a $\mathbf{Z}[q^{-1/2},\mu_{\nu q^\tau}]$-algebra containing $\chi$ and $\mu$. Let $(\phi,g_1,g_2)\in\mathcal{S}(F^2)\times G(F)$ be a $\Pi$-integral datum. Then, \eqref{eq: Problem} holds.
\end{prop}

\begin{proof}
    In this case, the $L$-factor is given by the lengthy expression
$$L(\Pi,s)=L(\chi_1\chi\mu_1\mu,s)L(\chi_1\chi\mu_2\mu^{-1},s)L(\chi_2\chi^{-1}\mu_1\mu,s)L(\chi_2\chi^{-1}\mu_2\mu^{-1},s).$$
As usual, we use the coset representatives $\gamma=(z_1g_{t,k_1,v_1},z_2ng_{0,k_2,v_2})$, where $z_i\in Z(F),n=\left[\begin{smallmatrix}
            1 & n_0\\
            & 1
        \end{smallmatrix}\right]\in N(F),t\in\mathbf{Z}, 0\leq k_i\leq c(\pi_i)$, and the starting expression
        \begin{align}\label{eq: 58}
            {\mathcal{I}_{\Pi}^\mathrm{new}(\gamma;s)=C\sum_{i\geq -k_2-c(\pi_2)}q^{i(1-s)}\int_{\mathcal{O}^\times}(\omega_{\pi_1}\omega_{\pi_2})(y)\psi(-n_0y\varpi^i)W_{\pi_1}^\mathrm{new}(g_{i+t,k_1,v_1y^{-1}})W_{\pi_2}^\mathrm{new}(g_{i,k_2,v_2y^{-1}})\ d^\times y}.
        \end{align}
       Let's assume that $\chi_1\neq \chi_2$ and $\mu_1\neq \mu_2$. Then, looking once more at \cite[Lemma $3.6$]{Assing_2018}, we see that in every case which never produces a pole, \eqref{eq: 58} is a finite $A[q^s,q^{-s}]$-linear combination of integrals of the same shape as the ones appearing in \Cref{thm sup + fully ram p-s}, which by now, we know how to handle. For every such analytic case, \eqref{eq: 58} becomes a finite $A[q^s,q^{-s}]$-linear combination of integrals of the following shapes:
        \begin{enumerate}
\item\begin{align*}\displaystyle\int_{\mathcal{O}^\times}(\omega_{\pi_1}&\omega_{\pi_2})(y)\psi(-n_0y\varpi^{i})(q-1)^4 \\
            &K(\chi_1^{-1}\otimes \chi_2^{-1},(\varpi^{a_1},\varpi^{b_1}),v_1y^{-1}\varpi^{-k_1})K(\mu_1^{-1}\otimes \mu_2^{-1},(\varpi^{a_2},\varpi^{b_2}),v_2y^{-1}\varpi^{-k_2})\ d^\times y\end{align*}
            \item ${\displaystyle\int_{\mathcal{O}^\times}(\omega_{\pi_1}\mu_j)(y)\psi(-n_0y\varpi^{i})(q-1)^2K(\chi_1^{-1}\otimes \chi_2^{-1},(\varpi^{a_1},\varpi^{b_1}),v_1y^{-1}\varpi^{-k_1})\ d^\times y}$
            \item ${\displaystyle\int_{\mathcal{O}^\times}(\chi_j\omega_{\pi_2})(y)\psi(-n_0y\varpi^{i})(q-1)^2K(\mu_1^{-1}\otimes \mu_2^{-1},(\varpi^{a_2},\varpi^{b_2}),v_2y^{-1}\varpi^{-k_2})\ d^\times y}$
            \item ${\displaystyle\int_{\mathcal{O}^\times}(\chi_j\mu_r)(y)\psi(-n_0y\varpi^{i})\ d^\times y}$
        \end{enumerate}
        for some $a_1,b_1,a_2,b_2\in\mathbf{Z}$ determined by $i,t,k_1,k_2$, and $j,r\in\{1,2\}$. We've seen how to work with such integrals in order to collapse to finite $A$-linear combinations of Gauss sums for characters made up of $\chi_i$'s and $\mu_j$'s and the latter expressions are always contained in $\tfrac{1}{q-1}A$. Keep in mind that the boundary case $(k_1,k_2)=(c(\pi_1),c(\pi_2))$ is identical to previous ones.
        The last thing we need to check is the cases which might potentially give rise to a pole, depending on the ramification of the product of various characters. These cases can only occur whenever $k_1\in\{c(\chi_1),c(\chi_2)\}$ and $ k_2\in \{c(\mu_1),c(\mu_2)\}$, which are all non-boundary cases.\\
        We give the details for the case $k_1=c(\chi_1)\neq c(\chi_2)$ and $k_2=c(\mu_1)$. All other cases, are handled in the same manner. In this case \eqref{eq: 58} is given by a finite $A[q^s,q^{-s}]$-linear combination of integrals of type $(1)$ and $(2)$ above, summed with an $A$-multiple of
        $$\sum_{i\geq M}q^{-is}(\chi^{-1}\mu^{-1})(\varpi^i)\int_{\mathcal{O}^\times}(\chi_2\mu_2)(y)\psi(-n_0y\varpi^i)\ d^\times y$$
        for an appropriate large enough $M$.
        If $\chi_2\mu_2$ is ramified, then this collapses. If $\chi_2\mu_2$ is unramified, then applying the usual splitting, we can write this as an element of $\tfrac{1}{q-1}A[q^s,q^{-s}]$ summed with an $A$-multiple of
        $$\sum_{i\geq 0}q^{-is}(\chi^{-1}\mu^{-1})(\varpi^i)=L(\chi^{-1}\mu^{-1},s)=L(\chi_2\chi^{-1}\mu_2\mu^{-1},s)\in L(\Pi,s)A[q^s,q^{-s}]$$
        and this completes the proof. The cases where $\chi_1=\chi_2$ and (or) $\mu_1=\mu_2$ are handled similarly, using \cite[Lemma $3.5$]{Assing_2018}. The boundary case is identical. The possible types of integrals appearing in analytic (non-boundary) cases, include the types $(1)$-$(4)$ above, plus the integrals obtained by replacing one or both $K$-functions with the $S$-functions appearing in \textit{loc.cit.}. We know how to handle all such expressions. To obtain the $L$-factor terms whenever $\chi_1=\chi_2$ and (or) $\mu_1=\mu_2$, one argues as in \Cref{thm full ram p-s x half ram p-s}.
\end{proof}

\begin{prop}
    Let $\pi_1=I(\chi_1\chi,\chi_2\chi^{-1})$ be a fully-ramified unitary principal-series representation as in \emph{\Cref{thm 2 full-ram p-s}}, and $\pi_2=\mathrm{St}_\mu$ for some ramified unitary character $\mu$ with $\mu(\varpi)=1$. Set $\Pi:=\pi_1\times\pi_2$, let $\tau:=\max\{c(\chi_1),c(\chi_2),c(\mu)\}$, $\nu:=q-1$ and $A\subseteq\mathbf{C}$ a $\mathbf{Z}[q^{-1/2},\mu_{\nu q^\tau}]$-algebra containing $\chi$ and $\mu$. Let $(\phi,g_1,g_2)$ be a $\Pi$-integral datum. Then, \eqref{eq: Problem} holds. 
\end{prop}
\begin{proof}
    In this case, the  $L$-factor is given by $L(\chi_1\chi\mu|\cdot|^{1/2},s)L(\chi_2\chi^{-1}\mu|\cdot|^{1/2})$. We use the coset representatives $\gamma=(z_1g_{t,k_1,v_1},z_2ng_{0,k_2,v_2})$, where $z_i\in Z(F),n=\left[\begin{smallmatrix}
            1 & n_0\\
            & 1
        \end{smallmatrix}\right]\in N(F),t\in\mathbf{Z}, 0\leq k_i\leq c(\pi_i)$. The starting expression is 
        \begin{align}\label{eq: 59}
            {\mathcal{I}_{\Pi}^\mathrm{new}(\gamma;s)=C\sum_{i\geq -k_2-c(\pi_2)}q^{i(1-s)}\int_{\mathcal{O}^\times}(\omega_{\pi_1}\mu^2)(y)\psi(-n_0y\varpi^i)W_{\pi_1}^\mathrm{new}(g_{i+t,k_1,v_1y^{-1}})W_{\pi_2}^\mathrm{new}(g_{i,k_2,v_2y^{-1}})\ d^\times y}.
        \end{align}
        Looking at \cite[Lemma $3.3$, Lemma $3.5$ \& Lemma  $3.6$]{Assing_2018}, we see that every analytic case, has already been handled. The interesting integrals appearing in such cases are combinations of the ones in \Cref{thm 2 full-ram p-s} and in \Cref{thm sup x ram St}. The cases with potential poles, are $k_1\in\{c(\chi_1),c(\chi_2)\}$ with $k_2=c(\chi)$. These are non-boundary cases, thus cancellation of $\tfrac{1}{q-1}$ by the volume factor always occurs. Plugging the expressions in these cases found in \textit{op.cit.} and using the usual arguments, one obtains the required result.
\end{proof}
\subsubsection{Two ramified twists of the Steinberg representation}
In this final section, we treat the case of two ramified twists of the Steinberg representation. This will finally put an end to our case-by-case analysis.

\begin{prop}
 Let $\pi_i=\mathrm{St}_{\chi_i}$ for ramified unitary characters $\chi_i$, with $\chi_i(\varpi)=1$. Set $\Pi:=\pi_1\times\pi_2$, let $\tau:=\max\{c(\chi_1),c(\chi_2)\}$, $\nu:=q-1$ and $A\subseteq\mathbf{C}$ be a $\mathbf{Z}[q^{-1/2},\mu_{\nu q^{\tau}}]$-algebra. Let $(\phi,g_1,g_2)\in\mathcal{S}(F^2)\times G(F)$ be a $\Pi$-integral datum. Then, \eqref{eq: Problem} holds.
\end{prop}

\begin{proof}
    The $L$-factor in this case, is given by $L(\chi_1\chi_2|\cdot|,s)L(\chi_1\chi_2,s)$. As usual, we use the coset representatives $\gamma=(z_1g_{t,k_1,v_1},z_2ng_{0,k_2,v_2})$, where $z_i\in Z(F),n=\left[\begin{smallmatrix}
            1 & n_0\\
            & 1
        \end{smallmatrix}\right]\in N(F),t\in\mathbf{Z}, 0\leq k_i\leq c(\pi_i)$, and the starting expression
        \begin{align}\label{eq: 60}
            {\mathcal{I}_{\Pi}^\mathrm{new}(\gamma;s)=C\sum_{i\geq -k_2-c(\pi_2)}q^{i(1-s)}\int_{\mathcal{O}^\times}(\chi_1^2\chi_2^2)(y)\psi(-n_0y\varpi^i)W_{\pi_1}^\mathrm{new}(g_{i+t,k_1,v_1y^{-1}})W_{\pi_2}^\mathrm{new}(g_{i,k_2,v_2y^{-1}})\ d^\times y}.
        \end{align}
        The analytic cases (i.e. no poles) follow in the same manner as \Cref{thm sup x ram St} and \Cref{thm 2 full-ram p-s}, using \cite[Lemma $3.3$]{Assing_2018}. Poles might occur for $(k_1,k_2)=(c(\chi_1),c(\chi_2))$; both non-boundary cases. In this case, using the usual tricks, \eqref{eq: 60} is given by an element of $\tfrac{1}{q-1}A[q^s,q^{-s}]$ summed with an $A$-multiple of 
        $$\sum_{i\geq M}q^{i(1-s)}q^{-2i}\int_{\mathcal{O}^\times}(\chi_1\chi_2)(y)\psi(-n_0y\varpi^i)\ d^\times y$$
        for some appropriate $M$. If $\chi_1\chi_2$ is ramified then this collapses as before to an element of $\tfrac{1}{q-1}A[q^s,q^{-s}]$. If instead it's unramified, then doing the usual splitting if necessary, this is given by an element of $\tfrac{1}{q-1}A[q^s,q^{-s}]$ summed with an $A$-multiple of
        $$\sum_{i\geq 0}q^{-is}q^{-i}=L(|\cdot|,s)=L(\chi_1\chi_2|\cdot|,s)\in L(\Pi,s)A[q^s,q^{-s}]$$
        where the containment follows from the explicit expression of $L(\Pi,s)$.
\end{proof}
\subsection{The main result.}
After this very lengthy case-by-case analysis, we can finally state a complete version of the result, for all irreducible admissible generic tempered representations $\pi_1\times\pi_2$ as follows. We now set 
$$\tau=\tau(\pi_1,\pi_2)=\max\{c(\pi_1),c(\pi_2)\},\ \nu=\nu(\pi_1,\pi_2)=\begin{dcases}
    1,\ &\mathrm{if}\ c(\pi_1)=c(\pi_2)=0\\
    q^2-1,\ &\mathrm{otherwise.}
\end{dcases}$$
\begin{thm}\label{thm: main theorem}
    Let $\pi_1,\pi_2$ be irreducible admissible generic tempered representations of $\GL_2(F)$. Set $\Pi:=\pi_1\times\pi_2$, and let $A\subseteq\mathbf{C}$ be a $\mathbf{Z}[q^{-1},\mu_{\nu q^\tau}]$-algebra containing the spherical Hecke eigenvalues of $\pi_i$ \emph{(}if spherical\emph{)} and the following unramified unitary characters in each case: 
    \begin{enumerate}
        \item $\omega_{\pi_i}$\emph{:} if $\pi_i$ is an unramified principal-series representation.
        \item $\chi_i$\emph{:} if $\pi_i\simeq\mathrm{St}_{\chi_i}$ is an unramified twist of the Steinberg representation
        \item $(\omega_{\pi_i}^\mathrm{unr})^{1/2}$ and $\chi_i$\emph{:} if $\pi_i\simeq (\omega_{\pi_i}^\mathrm{unr})^{1/2}\otimes I(\chi_i\omega,\chi_i^{-1})$ is a half-ramified principal-series representation with $\omega(\varpi)=1$, or $\pi_i\simeq (\omega_{\pi_i}^\mathrm{unr})^{1/2}\otimes I(\chi_{i,1}\chi_i,\chi_{i,2}\chi_i^{-1})$ is a fully-ramified principal-series representation with $\chi_{i,1}(\varpi)=\chi_{i,2}(\varpi)=1$.
        \item $(\omega_{\pi_i}^\mathrm{unr})^{1/2}$\emph{:} if $\pi_i\simeq (\omega_{\pi_i}^\mathrm{unr})^{1/2}\otimes \mathrm{St}_{\chi_i}$ is a ramified twist of the Steinberg representation with $\chi_i(\varpi)=1$, or $\pi_i\simeq (\omega_{\pi_i}^\mathrm{unr})^{1/2}\otimes \pi(\xi_i)$ is a supercuspidal representation with $\omega_{\pi(\xi_i)}(\varpi)=1$.
    \end{enumerate}
    Let $(\phi,g_1,g_2)$ be any $\Pi$-integral datum. Then,  $\Phi(\phi,g_1W_{\pi_1}^\mathrm{new},g_2W_{\pi_2}^\mathrm{new};X)\in A[X,X^{-1}]$ and
    $$Z(\phi,g_1W_{\pi_1}^\mathrm{new},g_2W_{\pi_2}^\mathrm{new};s)=L(\Pi,s)\cdot \Phi(\phi,g_1W_{\pi_1}^\mathrm{new},g_2W_{\pi_2}^\mathrm{new};q^s)$$
    in $L(\Pi,s)A[q^s,q^{-s}]\subseteq A(q^s,q^{-s}).$
\end{thm}
\begin{proof}
    If $\pi_i$ is not an unramified principal-series representation or an unramified twist of the Steinberg representation, then we can re-write it as  $\pi_i\simeq (\omega_{\pi_i}^\mathrm{unr})^{1/2}\otimes \pi_i^{'}$ where $\omega_{\pi_i}^\mathrm{unr}$ is the unramified character defined by $\omega_{\pi_i}^\mathrm{unr}(\varpi):=\omega_{\pi_i}(\varpi)$. The representation $\pi_i^{'}$ is in the same class as $\pi_i$, but with central character that is trivial on $\varpi$, and equal to $\omega_{\pi_i}$ on $\mathcal{O}^\times$. Finally, $c(\pi_i^{'})=c(\pi_i)$. We note that all of our results in this section remain unchanged after unramified twists, and thus we can now apply the required proposition of this section, depending on the classes of $\pi_1,\pi_2$, while potentially enlarging $A$ to ensure it contains $(\omega_{\pi_i}^\mathrm{unr})^{1/2}$. This also implicitly ensures that in every case, the inverse $L$-factor $L(\Pi,s)^{-1}$ is contained in $A[q^s,q^{-s}]$.
\end{proof}

\begin{rem}\label{rem p=2}
    We want to note that \Cref{thm: main theorem} also holds for $p=2$ (i.e. $2$-adic field $F$ with residue cardinality $q=2^t$) under the extra assumption that $\pi_i$ is dihedral if it's a supercuspidal representation. In this case one takes $A\subseteq \mathbf{C}$ a $\mathbf{Z}[q^{-1},\mu_{4\nu q^\tau}]$-algebra as in \Cref{thm: main theorem}. The proofs are identical and the expressions of \cite{Assing_2018} for the Whittaker new vectors still hold under this extra dihedral assumption in the supercuspidal case.
\end{rem}

\subsubsection{Relation to trilinear forms}
 One of the main results of Prasad \cite{prasad1990trilinear} on trilinear forms for representations of $p$-adic $\GL_2$ is the following.
\begin{thm}[\cite{prasad1990trilinear}]\label{thm prasad trilinear}
    Let $\pi_1,\pi_2,\pi_3$ be irreducible admissible generic representations of $\GL_2(F)$ with $\omega_{\pi_1}\omega_{\pi_2}\omega_{\pi_3}=1$. If at least one of the $\pi_i$ is a principal-series representation, then the space 
    $\mathrm{Hom}_{\GL_2(F)}(\pi_1\otimes\pi_2\otimes\pi_3,\mathbf{1})$
    is one-dimensional. Otherwise, it is zero-dimensional.
    \begin{proof}
        This is part of \cite[Theorem $1.1$ \& Theorem $1.2$]{prasad1990trilinear}.
    \end{proof}
\end{thm}
\noindent From now on, let's assume, without loss of generality, that $\pi_3$ is a principal-series representation. Then up to twisting, we can also assume that $\pi_3\simeq I(|\cdot|^{-1/2},|\cdot|^{1/2}\nu)$ for some character $\nu\neq 1$ of $F^\times$, and $\omega_{\pi_1}\omega_{\pi_2}=\nu^{-1}$. Then, for example by \cite[Proposition $3.3$]{loeffler2024local}, the representation $\pi_3$ can be identified with the maximal quotient of $\mathcal{S}(F^2)$ on which $F^\times\simeq Z(F)$ acts via the character $\nu$. Thus, we can regard vectors in $\pi_3$ as equivalence classes of Schwartz functions, which we denote by $[\phi]_{\pi_3}$. It follows that the non-zero linear form $\mathcal{Z}:W_1\otimes W_2\otimes \phi\mapsto \mathrm{lim}_{s\rightarrow0}\frac{Z(\phi,W_1,W_2;s)}{L(\pi_1\times\pi_2,s)}$ in $\mathrm{Hom}_{\GL_2(F)}(\pi_1\otimes\pi_2\otimes\mathcal{S}(F^2),\mathbf{1})$ descends to a well-defined nonzero linear form $\mathcal{Z}^{\mathrm{tri}}: W_1\otimes W_2\otimes [\phi]_{\pi_3}\mapsto \mathcal{Z}(W_1\otimes W_2\otimes\phi)$ in $\mathrm{Hom}_{\GL_2(F)}(\pi_1\otimes\pi_2\otimes\pi_3,\mathbf{1})$. By \Cref{thm prasad trilinear} this is the unique (up to scalars) non-zero element of this Hom-space. Using \Cref{thm: main theorem} we can obtain the following integrality result regarding this unique trilinear form. This extends (and gives an alternative proof of) \cite[Corollary $6.2.1(1)$ \& Corollary $6.3.1(1)$]{groutides2024rankinselbergintegralstructureseuler} from the unramified case to the general case. 
\begin{thm}
    Let $\pi_1,\pi_2$ be irreducible, admissible generic tempered representations of $\GL_2(F)$, and let $\pi_3=I(|\cdot|^{-1/2},|\cdot|^{1/2}\nu)$ be an irreducible principal-series representation with $\omega_{\pi_1}\omega_{\pi_2}\nu=1$. Set $\Pi:=\pi_1\times\pi_2$ and let $A$ be a $\mathbf{Z}[q^{-1},\mu_{\nu q^\tau}]$-algebra as in \emph{\Cref{thm: main theorem}}. Then, the unique (up to scalars) non-zero linear form $\mathcal{Z}^\mathrm{tri}\in\mathrm{Hom}_{\GL_2(F)}(\pi_1\otimes\pi_2\otimes\pi_3,\mathbf{1})$, realized using the Rankin-Selberg zeta-integral, satisfies
    $$\mathcal{Z}^\mathrm{tri}(g_1W_{\pi_1}^\mathrm{new}\otimes g_2W_{\pi_2}^\mathrm{new}\otimes [\phi]_{\pi_3})\in A$$
    for any $\Pi$-integral datum $(\phi,g_1,g_2)\in\mathcal{S}(F^2)\times \GL_2(F)^2$.
\end{thm}

\section{An application to newforms}\label{sec application fo newforms}
\subsection{Setup}\label{sec setup}
Let $\psi_\mathbf{A}:=\prod_v\psi_v:\mathbf{Q}\backslash \mathbf{A}\rightarrow\mathbf{C}^\times$ be the standard additive character, where at each finite prime, $\psi_p$ is the standard additive character of conductor $\mathbf{Z}_p$ and $\psi_\infty:=\exp(2\pi i(-))$. For ease of notation, we write  $\psi^{(1)}:=\psi_\mathbf{A}$ and $\psi^{(2)}:=\psi_\mathbf{A}^{-1}$. Consider two normalized cuspidal new eigenforms $f_1,f_2$ of levels $N_1,N_2$, even integral weights $k_1\geq k_2\geq 2$, and Nebentypes $\epsilon_1,\epsilon_2$. Let $\omega_1,\omega_2$ be the adelization of $\epsilon_1,\epsilon_2$, and write $\omega_{1,v},\omega_{2,v}$ for their restrictions to $\mathbf{Q}_v^\times$. Let $i\in\{1,2\}$. Writing 
$$K_0(N_i):=\left\{\left[\begin{smallmatrix}
    a & b\\
    c & d
\end{smallmatrix}\right]\in \GL_2(\hat{\mathbf{Z}})\ |\ c\equiv0\mod\ N_i \right\}$$
one can define a character $\lambda_i:K_0(N_i)\rightarrow\mathbf{C}^\times$ by $\lambda_i(\left[\begin{smallmatrix}
    a & b\\
    c & d
\end{smallmatrix}\right]):=\prod_{p|N_i}\omega_{i,p}(d_p)$. Using the well-know fact that $H(\mathbf{A})=H(\mathbf{Q})H(\mathbf{R})^+K_0(N_i)$, one can use $f_i$, to define a function $\varphi_{f_i}$ on $H(\mathbf{A})$ by 
$$\varphi_{f_i}(\gamma g_\infty \kappa):=f_i(g_\infty (\sqrt{ -1}))\cdot (c\sqrt{-1}+d)^{-k_i}\cdot \det(g_\infty)^{-k_i/2}\cdot \lambda_i(\kappa)$$
where $\gamma\in H(\mathbf{Q}), g_\infty=\left[\begin{smallmatrix}
    a & b\\
    c & d
\end{smallmatrix}\right]\in H(\mathbf{R})^+$ and $\kappa\in K_0(N_i)$. By \cite{gelbart1975automorphic}, this function $\varphi_{f_i}$ is a cuspidal automorphic form on $H(\mathbf{A})$ with central character $\omega_i$. We write $\pi_{f_i}\simeq\otimes_p^{'}\pi_{f_i,p}$ for the irreducible unitary cuspidal automorphic representation of $H(\mathbf{A)}$ given by the linear span of the $H(\mathbf{A})$-translates of $\varphi_{f_i}$. The representations $\pi_{f_1},\pi_{f_2}$ are globally and locally generic and we can identify them with $\mathcal{W}(\pi_{f_i},\psi^{(i)})\simeq \otimes^{'}_v \mathcal{W}(\pi_{f_i,v},\psi_v^{(i)})$. For cusp forms $\varphi_i\in\pi_{f_i}$ and $\phi_\mathbf{A}\in\mathcal{S}(\mathbf{A}^2)$, we have the global Rankin-Selberg integral of \cite{jacquet1981euler}
$$I(\phi_\mathbf{A},\varphi_1,\varphi_2;s)=\int_{Z(\mathbf{A})H(\mathbf{Q})\backslash H(\mathbf{A})}\varphi_1(h)\varphi_2(h)E(\phi_\mathbf
A,h;s)\ dh$$
where $E(\phi_\mathbf
A,h;s)$ is the usual adelic Eisenstein series attached to $\phi_\mathbf{A}$ as in \cite{automorphic_on_GL}. This is a meromorphic function of $s\in\mathbf
C$ and $L(\pi_{f_1}\times\pi_{f_2},s)^{-1}I(\phi_\mathbf{A},\varphi_1,\varphi_2;s)$ is entire, where the automorphic $L$-function is defined as usual, through an Euler product of local $L$-factors:
$$L(\pi_{f_1}\times\pi_{f_2},s):=L(\pi_{f_1,\infty}\times\pi_{f_2,\infty},s)\prod_{p}L(\pi_{f_1,p}\times\pi_{f_2,p},s).$$
The local $L$-factors at the finite places, are given by \eqref{eq: gl2 l-factor}. At the archimedean place, the local component $\pi_{f_i,\infty}$ is the discrete series representation of lowest weight $k_i$, and following \cite{chen2020primitive}, the local $L$-factor at $\infty$ is given by $L(\pi_{f_1,\infty}\times\pi_{f_2,\infty},s)=\Gamma_\mathbf{C}(s+\tfrac{k_1-k_2}{2})\Gamma_\mathbf{C}(s+\tfrac{k_1+k_2}{2}-1)$, where $\Gamma_\mathbf{C}(s):=2(2\pi)^{-s}\Gamma(s)$, and $\Gamma(s)=\int_{\mathbf{R}_+}y^se^{-y}\ d^\times y$ is the usual Gamma function. Since $f_i$ is a normalized newform, its associated automorphic form $\varphi_{f_i}$ is the normalized new vector of $\pi_{f_i}$. In other words,
$$\varphi_{f_i}\simeq W_{\pi_{f_i},\infty}^\mathrm{new}\otimes \bigotimes_{p}W_{\pi_{f_i},p}^\mathrm{new}$$
where for primes $p\nmid\infty$, the local new vectors are the ones of the previous section, and at the archimedean place, $W_{\pi_{f_i},\infty}^\mathrm{new}$ is given by
$$W_{\pi_{f_i},\infty}^\mathrm{new}\left(\left[\begin{smallmatrix}
    z & \\
    & z
\end{smallmatrix}\right]\left[\begin{smallmatrix}
    y & x \\
    & 1
\end{smallmatrix}\right]\left[\begin{smallmatrix}
    \cos(\theta) & \sin(\theta)\\
    -\sin(\theta) & \cos(\theta)
\end{smallmatrix}\right]\right)=\ch(\mathbf{R}_+)(y)\cdot y^{k_i/2}e^{-2\pi y}\cdot \mathrm{sign}(z)^{k_i}\psi^{(i)}_\infty(x)e^{\sqrt{-1} k_i\theta}$$
for $z,y\in\mathbf{R}^\times$ and $x,\theta\in\mathbf{R}$.
\begin{defn}[\cite{chen2020primitive} Definition $4.1$]
   For an integer $k\in\mathbf{Z}_{\geq 0}$, we define the Schwartz function $\phi_\infty^{(k)}$ in $\mathcal{S}(\mathbf{R}^2)$ by 
$$\phi_\infty^{(k)}(x,y):=2^{-k}(x+\sqrt{-1}y)^ke^{-\pi(x^2+y^2)}.$$
\end{defn}
\subsection{The integral coefficients}
\begin{defn}

Given $f_1,f_2$ as above, we define the following. We firstly set constants:
\begin{enumerate}
    \item$N:=N_1N_2$.
    \item$\nu:=\prod_{\substack{p|N\\ p\ \mathrm{prime}}}\nu_p$ where $\nu_p:=p^2-1$.
    \item$\rho:=\prod_{\substack{p|N\\ p\ \mathrm{prime}}}p^{\tau_p}$ where $\tau_p:=\max\{c(\pi_{f_1,p}),c(\pi_{f_2,p}),v_p(4N)\}=v_p(4N)$. 
\end{enumerate}
Let $\Sigma_{f_i}^{\mathrm{p\text{-}s}}$ be the subset of \textcolor{black}{the set of primes dividing $N_i$} for which $\pi_{f_i,p}$ is a (ramified) principal-series representation. For every prime $p\in \Sigma_{f_i}^{\mathrm{p\text{-}s}}$, we fix a Dirichlet character $\theta^{(i,p)}$ of minimal conductor $n_{i,p}$, such that the twist $f_i^{p\text{-}\mathrm{prim}}:=f_i\otimes\theta^{(i,p)}$ is $p$-primitive. Here, $f_i\otimes\theta^{(i,p)}$ denotes the unique normalized newform whose $\ell$-th Fourier coefficient is given by $\theta^{(i,p)}(\ell)a_\ell(f_i)$ for almost all primes $\ell$. Let $L_{f_1,f_2}$ be the Galois closure of $\mathbf{Q}(f_1,f_2)$. Finally, we define the number field $E_{f_1,f_2}:=L_{f_1,f_2}(\mu_{\nu\rho})(\mu_{n_{1,p}}:p\in\Sigma_{f_1}^{\mathrm{p\text{-}s}})(\mu_{n_{2,p}}:p\in \Sigma_{f_2}^{\mathrm{p\text{-}s}})$ \textcolor{black}{which is a Galois extension of $\mathbf{Q}$.}
\end{defn}
\begin{thm}\label{thm 5.2.2}
     Let $f_1,f_2$ be normalized cuspidal new eigenforms of even integral weights $k_1\geq k_2\geq 2$ and levels $N_1,N_2$ such that the $2$-component of $\pi_{f_i}$ is dihedral, if it's a supercuspidal representation. Set $\Pi:=\pi_{f_1}\times\pi_{f_2}$ and let $(\phi_\mathrm{f},g_1,g_2)\in \mathcal{S}(\mathbf{A}_\mathrm{f}^2)\times G(\mathbf{A}_\mathrm{f})$ be an arbitrary, locally $\Pi_p$-integral datum for all primes $p$. Let $S$ be the finite set of primes made up of \textcolor{black}{the prime divisors of $N$} and all other primes for which $(\phi_{\mathrm{f},p},g_{1,p},g_{2,p})$ is not an unramified datum. Finally, let $\phi_\mathbf{A}:=2^{k_1+1}\phi_\infty^{(k_1-k_2)}\otimes\phi_\mathrm{f}\in\mathcal
     S(\mathbf{A}^2)$. Then,
\begin{align*}
    I(\phi_\mathbf{A},g_1\varphi_{f_1},g_2\varphi_{f_2};s)=L(\pi_{f_1}\times\pi_{f_2},s)\cdot \Phi(p^s : p\in S)
\end{align*}
where $\Phi(X_p :p\in S)$ is a polynomial in $ \mathcal{O}_{E_{f_1,f_2}}[S^{-1}][X_p^{\pm1} : p\in S]$. 
\end{thm}
\begin{proof}
    Recall that the $f_i$ are normalized newforms, thus under the identifications $\pi_{f_i}\simeq\otimes_v^{'}\mathcal{W}(\pi_{f_i,v},\psi_v^{(i)})$, we have $g_i\varphi_{f_i}\simeq W_{\pi_{f_i},\infty}^\mathrm{new}\otimes \bigotimes_pg_{i,p}W_{\pi_{f_i},p}^\mathrm{new}$. By \cite{jacquet1981euler}, we have $$I(\phi_\mathbf{A},g_1\varphi_{f_1},g_2\varphi_{f_2};s)=Z(\phi_\infty^{(k_1-k_2)},W_{\pi_{f_1},\infty}^\mathrm{new},W_{\pi_{f_2},\infty}^\mathrm{new};s)\prod_p Z(\phi_{\mathrm{f},p},g_{1,p}W_{\pi_{f_1},p}^\mathrm{new}, g_{2,p}W_{\pi_{f_2},p}^\mathrm{new};s).$$
    The archimedean zeta-integral is given by 
    \begin{align}\label{eq: archimedean}\int_{\mathbf{R}^\times}\int_{\mathbf{R}^\times}\int_{\mathrm{SO}_2(\mathbf{R})}W_{\pi_{f_1},\infty}^\mathrm{new}(\left[\begin{smallmatrix}
        y & \\
        & 1
    \end{smallmatrix}\right]u)W_{\pi_{f_2},\infty}^\mathrm{new}(\left[\begin{smallmatrix}
        y & \\
        & 1
    \end{smallmatrix}\right]u)|y|_\mathbf{R}^{s-1}\phi_\infty^{(k_1-k_2)}((0,t)u)\omega_\infty(t)|t|_\mathbf{R}^{2s}\ d^\times t\ d^\times y\ du\end{align}
    where $\omega_\infty:=\omega_{1,\infty}\omega_{2,\infty}$, $d^\times y,d^\times t$ are the standard Lebesgue measures, and $du$ is the normalized Haar measure on $\mathrm{
    SO}_2(\mathbf{R})$ giving it volume $1$. A computation along the same lines as \cite[Proposition $5.3$]{chen2020primitive} shows that \eqref{eq: archimedean} is given by $(\sqrt{-1})^{k_1-k_2}L(\pi_{f_1,\infty}\times\pi_{f_2,\infty},s)$. If $p\notin S$, then the $p$-adic zeta-integral is identically equal to the local $L$-factor. For the $p$-adic zeta-inegrals regarding primes $p\in S$, we wish to apply \Cref{thm: main theorem} for the $\Pi_p$-integral datum $(\phi_{\mathrm{f},p},g_{1,p},g_{2,p})\in \mathcal{S}(\mathbf{Q}_p^2)\times G(\mathbf{Q}_p)$, with $A:=\mathcal{O}_{E_{f_1,f_2}}[S ^{-1}]$.\\
    If $p\in S$, $p\nmid N$, then $\pi_{f_i,p}$ is the unramified principal-series representation $I(\chi,\mu)$, with $\chi(p):=\alpha_i/p^{(k_i-1)/2}$ and $\mu(p):=\beta_i/p^{(k_i-1)/2}$, where $\alpha_i,\beta_i$ are the roots of the Hecke polynomial $X^2-a_p(f_i)X+\epsilon_i(p)p^{k_i-1}.$ Thus, $\pi_{f_i,p}$ has spherical Hecke eigenvalues equal to $(\alpha_i +\beta_i)p^{1-\tfrac{k_i}{2}}$ and $(\epsilon_1\epsilon_2)(p)$. By a theorem of Shimura, the coefficients of the Hecke polynomial are elements of $\mathcal{O}_{E_{f_1,f_2}}$. Thus, since $k_i$ is even by assumption, we see that $(\alpha_i +\beta_i)p^{1-\tfrac{k_i}{2}}\in A$, and we also have $(\epsilon_1\epsilon_2)(p)^{\pm 1}\in A$.\\
    If $p|N$, then at least one $\pi_{f_i,p}$ is ramified. We always have $\omega_{\pi_{f_i,p}}(p)^{\pm 1/2}\in A$ since $2N|\nu\rho.$ If $\pi_{f_i,p}$ is either a supercuspidal representation or a twist of the Steinberg representation then looking at \Cref{thm: main theorem}, this is enough. If $\pi_{f_i,p}$ is a half-ramified principal-series representation (i.e. $f_i$ is $p$-primitive), we have by \cite[Proposition $2.8$]{loeffler2012computation} $\pi_{f_i,p}\simeq I(\chi_{i}\omega_{i,p},\chi_{i}^{-1})$ with $\omega_{i,p}$ ramified, and $\chi_{i}$ the unramified character given by $\chi_{i}(p)^{-1}:=a_p(f_i)/p^{(k_i-1)/2}.$ This value is contained in $A$ (since $\mu_{4p}\subseteq A$). We also need to verify that $\chi_{i}(p)$ is contained in $A$. We know that $|a_p(f_i)|=p^{(k_i-1)/2}$ and so the minimal polynomial of $a_p(f_i)/p^{(k_i-1)/2}$ over $\mathbf{Q}$ has constant term $\pm 1$. But $E_{f_1,f_2}$ is Galois, and $a_p(f_i)$ and $\mu_{4p}$ are contained in $\mathcal{O}_{E_{f_1,f_2}}$. Thus, $a_p(f_i)/p^{(k_i-1)/2}\in\mathcal{O}_{E_{f_1,f_2}}[p^{-1}]^\times$. From this we conclude that $A$ contains $\chi_{i}$. Writing 
    $$\pi_{f_i,p}\simeq (\omega_{i,p}^\mathrm{unr})^{1/2}\otimes I\left(\tilde\omega_{i,p}\chi_{i} (\omega_{i,p}^\mathrm{unr})^{1/2},\chi_{i}^{-1}(\omega_{i,p}^\mathrm{unr})^{-1/2}\right)$$
    where $\tilde\omega_{i,p}:=\omega_{i,p} \cdot(\omega_{i,p}^\mathrm{unr})^{-1}$ now satisfies $\tilde\omega_{i,p}(p)=1$, we see that the condition of \Cref{thm: main theorem} is once again satisfied.\\
    If $\pi_{f_i.p}$ is a fully-ramified principal-series representation (i.e $f_i$ is not $p$-primitive) which is not a twist of an unramified principal-series representation, then we have $\pi_{f_i,p}\simeq (\omega_{i,p}^\mathrm{unr})^{1/2}\otimes I(\chi_{i,1}\chi_i,\chi_{i,2}\chi_i^{-1})$ with $\chi_{i,1}\neq\chi_{i,2}$ both ramified, $\chi_{i,1}(p)=\chi_{i,2}(p)=1$ and $\chi_i$ unramified. Then we consider $f_i^{p\text{-}\mathrm{prim}}$ which is now $p$-primitive. Thus, again by \cite[Proposition $2.8$]{loeffler2012computation}, we have that $$\pi_{f_i^{p\text{-}\mathrm{prim}},p}\simeq(\omega_{i,p}^\mathrm{unr})^{1/2}\otimes I\left(\theta_p^{(i,p)}\chi_{i,1}\chi_i,\theta_p^{(i,p)}\chi_{i,2}\chi_i^{-1}\right)= I(\sigma_1,\sigma_2)$$
    where one and only one of the $\sigma_i$ is unramified; without loss of generality, say $\sigma_2$ is unramified, and thus it's the unramified character determined by $\sigma_2(p):=a_p(f_i^{p\text{-}\mathrm{prim}})/p^{(k_i-1)/2}$. Firstly,  by strong multiplicity one $\sigma_2(p)\in A$. The fact that $\sigma_2(p)^{-1}\in A$ follows from the same argument as before, since once again $|a_p(f_i^{p\text{-}\mathrm{prim}})/p^{(k_i-1)/2}|=1$. Thus, $\chi_i$ is also contained in $A$ as needed for \Cref{thm: main theorem}.\\
    Finally, if $\pi_{f_i,p}$ is a ramified twist of an unramified principal-series representation, then $\pi_{f_i,p}\simeq (\omega_{i,p}^\mathrm{unr})^{1/2}\otimes I(\chi_{i,1}\chi_i,\chi_{i,1}\chi_i^{-1})$ with $\chi_{i,1}$ ramified, $\chi_{i,1}(p)=1$ and $\chi_i$ unramified. In this case, $f_i^{p\text{-}\mathrm{prim}}$ has no level at $p$ and
    $\pi_{f_i^{p\text{-}\mathrm{prim}},p}$ is actually an unramified principal-series representation $I(\mu_1,\mu_2)$ with $\mu_1(p):=\gamma_i/p^{(k_i-1)/2}$ and $\mu_2(p):=\delta_i /p^{(k_i-1)/2}$ where $\gamma_i,\delta_i$ are the roots of the Hecke polynomial $X^2-a_p(f_i^{p\text{-}\mathrm{prim}})X-\epsilon_{f_i^{p\text{-}\mathrm{prim}}}(p)p^{k_i-1}\in \mathcal{O}_{E_{f_1,f_2}}[X]$. Hence, $\chi_i$ is contained in $A$ as required for \Cref{thm: main theorem}. Taking 
    $$\Phi(X_p:p\in S):=(\sqrt{-1})^{k_1-k_2}\prod_{p\in S}\Phi(\phi_{\mathrm{f},p},g_{1,p}W_{\pi_{f_1},p}^\mathrm{new},g_{2,p}W_{\pi_{f_2},p}^\mathrm{new};X_p)$$
    gives the result by \Cref{thm: main theorem}.
    \end{proof}
 \bibliography{citation} 
\bibliographystyle{amsalpha}

\address{Mathematics Institute, Zeeman Building, University of Warwick, Coventry CV4 7AL,
England}

\email{Alexandros.Groutides@warwick.ac.uk}

\end{document}